\documentclass[10pt]{article}
\advance\oddsidemargin by -1.0cm
\advance\evensidemargin by -1.0cm
\advance\topmargin by -1.0cm
\usepackage{latexsym,amssymb, yfonts}
\usepackage{amsmath}
\usepackage{amsthm}
\usepackage{enumerate}
\usepackage{psfrag}
\usepackage{graphicx}
\numberwithin{equation}{section}
\newtheorem{theo}{Theorem}[subsection]
\newtheorem{prop}{Proposition}[subsection]
\newtheorem{cor}{Corollary}[subsection]
\newtheorem{lem}{Lemma}[subsection]
\theoremstyle{definition}
\newtheorem{df}{Definition}[subsection]
\newtheorem{ex}{Example}[subsection]
\newtheorem{exs}{Example}[section]
\theoremstyle{remark}
\newtheorem{rem}{Remark}[subsection]
\newtheorem*{notation}{Notation}
\newcommand{\bdm}{\begin{displaymath}}
\newcommand{\edm}{\end{displaymath}}
\newcommand{\be}{\begin{equation}}
\newcommand{\ee}{\end{equation}}
\newcommand{\ra}{\rightarrow}
\newcommand{\cyclic}[1]{\stackrel{\scriptsize #1}{\mathfrak{S}}}
\newcommand{\R}{\mathbb{R}}
\newcommand{\C}{\mathbb{C}}
\renewcommand{\H}{\mathcal{H}}
\newcommand{\V}{\mathcal{V}}
\newcommand{\M}{\ensuremath{\mathcal{M}}}
\newcommand{\Scal}{\mathrm{Scal}}
\newcommand{\Ric}{\mathrm{Ric}}
\newcommand{\Id}{\mathrm{Id}}
\newcommand{\End}{\ensuremath{\mathrm{End}}}
\newcommand{\so}{\ensuremath{\mathfrak{so}}}

\newcommand{\Spin}{\ensuremath{\mathrm{Spin}}}

\begin{document}

\title{Generalizations of
3-Sasakian manifolds and skew torsion}
%\title{Generalizations of
%3-Sasakian manifolds and their compatible connections with skew torsion}

\author{Ilka Agricola and Giulia Dileo}
\date{}

\maketitle

%\vspace{-0.5cm}
\begin{abstract}
In the first part, we define and investigate new classes of
almost $3$-contact metric
manifolds, with two guiding ideas in mind: first, what geometric objects
are best suited for capturing the key properties of almost $3$-contact metric manifolds,
and second, the newly defined classes should admit `good' metric connections with skew torsion.
In particular, we introduce the \emph{Reeb commutator function} and the
\emph{Reeb Killing function}, we define the new classes of
\emph{canonical almost $3$-contact metric
manifolds} and of \emph{$3$-$(\alpha,\delta)$-Sasaki manifolds} (including as special cases
 $3$-Sasaki manifolds, quaternionic Heisenberg groups, and many others) and prove
that the latter are hypernormal, thus generalizing a seminal result by Kashiwada.
We study their behaviour under a new class of deformations, called $\H$-homothetic
deformations, and prove that they admit an underlying quaternionic contact structure,
from which we deduce the Ricci curvature.
For example, a $3$-$(\alpha,\delta)$-Sasaki
manifold is Einstein either if $\alpha=\delta$ (the 3-$\alpha$-Sasaki case) or
if $\delta=(2n+3)\alpha$, where $\dim M=4n+3$.

The second part is actually devoted to finding these  adapted connections.
We start with a very general notion of \emph{$\varphi$-compatible
connections}, where $\varphi$ denotes any element of the associated sphere of almost
contact structures, and make them unique by a certain extra condition, thus yielding
the notion of \emph{canonical connection} (they exist exactly on canonical manifolds, hence
the name).
For $3$-$(\alpha,\delta)$-Sasaki manifolds, we compute the torsion of this connection
explicitly and we prove that  it is parallel, we describe the holonomy, the $\nabla$-Ricci
curvature, and we show that the metric cone is a HKT-manifold. In dimension $7$, we
construct a cocalibrated $G_2$-structure inducing the canonical connection and
we prove the existence of four generalized Killing spinors.
\end{abstract}

\medskip
\noindent
{\small

}

%\vspace{2cm}

\tableofcontents

%-----------------------------------------------------------------------------------------
\section{Introduction and basic notions}
%-----------------------------------------------------------------------------------------
\subsection{Introduction and summary}
%----------------------------------------------------------------------------------------
%
Since their first definition by Kuo in 1970,
almost $3$-contact metric manifolds have been a steady, but difficult topic of research.
They are a very natural objects to consider --- they have
three almost contact metric structures with
orthonormal Reeb vector fields and compatibility relations modelled on
the multiplication rules of the quaternions.
Unfortunately, they turn out to be rather difficult to handle.
Computations become quickly lengthy and complicated. Compared to other geometries (like
almost hermitian manifolds or symplectic manifolds), their definition is not equivalent
to the reduction of the frame bundle to a certain subgroup $G\subset \mathrm{O}(n)$,
hence they do not admit a `good' classification scheme into different classes. A deeper
reason for most of the encountered problems seems to be that, together with almost
contact metric structures, they do not possess an integrable Riemannian counterpart, in
the sense that contact geometry does not appear in Berger's theorem on irreducible
Riemannian holonomies. As a consequence, the Levi-Civita connection is not well-adapted
to their geometric structure, and  the quest for other connections (like
hermitian connections for almost hermitian manifolds) turns out to be a challenging task,
with many open questions.

The current paper has the goal to address two of the sketched problems.
In the first part, we define and investigate new classes of almost $3$-contact metric
manifolds, with two guiding ideas in mind: first, what geometric objects
are best suited for capturing the key properties of almost $3$-contact metric manifolds,
and second, the newly defined classes should admit good invariant connections.
The second part is actually devoted to finding these
 adapted connections, with attention restricted to connections that are metric and
with skew torsion. Such connections are by now a widely established tool for the successful
investigation of most non-integrable geometries. As a side condition, the classes defined in
the first part should include $3$-Sasaki manifolds and quaternionic
Heisenberg groups, and the results proved in the second part should reproduce some known
partial results on these.

\subsubsection*{Part one (Sections 1--2)}
%----------------------------------------
Consider an almost $3$-contact metric manifold $(M,\varphi_i,\xi_i,\eta_i, g)$, $i=1,2,3$.
If there exists a function $\delta\in C^\infty(M)$ such that
$\eta_k([\xi_i,\xi_j])=2\delta\epsilon_{ijk}$  for any $i,j,k=1,2,3$, we call
it the \emph{Reeb commutator function} of $M$. Of course, not any $M$ will admit
a Reeb commutator function; but if it does, this function $\delta$ encodes
in a very succinct way the relative `positions' of the Reeb vector fields $\xi_i$.
If the Reeb vector fields are Killing (like for $3$-$\alpha$-Sasakian structures
and many other almost $3$-contact structures), we prove in Corollary \ref{cor.Reebcf-Killing}
that the existence of the Reeb commutator function $\delta$ follows.
 This function will  play a special role in our study. As a first class of manifolds admitting
a Reeb commutator function, we introduce \emph{$3$-$\delta$-cosymplectic manifolds}: they 
will be defined by the conditions
$d\eta_i=-2\delta\eta_j\wedge\eta_k$ and $d\Phi_i=0$,
for every even permutation $(i, j, k)$ of $(1, 2, 3)$, where $\delta$ is a real 
constant, and $\Phi_i$ denotes the
fundamental 2-form given by $\Phi_i(X,Y) = g(X,\varphi_iY)$. For $\delta= 0$ we get in 
fact a $3$-cosymplectic manifold.
 
Any $3$-Sasakian structure has constant Reeb commutator function $\delta=1$, while for a 
quaternionic Heisenberg group
the Reeb vector fields commute, hence $\delta=0$. Actually, we place both
$3$-$\alpha$-Sasakian manifolds and the quaternionic Heisenberg groups in the more
general class of almost
$3$-contact metric manifolds $(M,\varphi_i,\xi_i,\eta_i,g)$ satisfying the following
condition:
\begin{equation}\label{intro_3alpha-delta}
d\eta_i=2\alpha\Phi_i+2(\alpha-\delta)\eta_j\wedge\eta_k
\end{equation}
for every even permutation $(i,j,k)$ of $(1,2,3)$, where $\alpha$ and $\delta$ are
real constants, and $\alpha\ne 0$. We call these manifolds
\emph{$3$-$(\alpha,\delta)$-Sasaki manifolds}; they include $3$-$\alpha$-Sasaki manifolds
as a special case ($\alpha=\delta$). As we shall see many geometric features
are captured by equation \eqref{intro_3alpha-delta}. First we show that
$3$-$(\alpha,\delta)$-Sasaki manifolds are hypernormal, that is the Nijenhuis tensor fields $N_{\varphi_i}:=
[\varphi_i, \varphi_i] + d\eta_i\otimes\xi_i$ are all vanishing
 (Theorem \ref{thm.general-Kashiwada}):
this is a generalization of a seminal result of Kashiwada stating
that every $3$-contact metric manifold (corresponding to $\alpha=\delta=1$) is
$3$-Sasakian \cite{kashiwada}.
Furthermore, every $3$-$(\alpha,\delta)$-Sasaki manifold has Killing Reeb vector fields,
with constant Reeb commutator function $\delta$ (Corollary \ref{corollary_nabla_xi}).
We also study the
behavior of these structures
under a new type of deformations, called \emph{$\mathcal H$-homothetic deformations}
(Section \ref{subsec.3-AD-props-and-exa}).
We show that these deformations preserve the class of $3$-$(\alpha,\delta)$-Sasaki
structures with $\delta=0$, called \emph{degenerate}. In the non-degenerate case, the
sign of the product $\alpha\delta$ is  preserved. In particular all
$3$-$(\alpha,\delta)$-Sasaki structures with $\alpha\delta>0$ can be deformed  into a
$3$-Sasakian structure. Examples of $3$-$(\alpha,\delta)$-Sasaki structures with
$\alpha\delta<0$  do exist as well: they can be
%obtained from the negative $3$-Sasakian structure
defined on the canonical principal $SO(3)$-bundle of a quaternionic K\"ahler (not
hyperK\"ahler) manifold with negative scalar curvature \cite{konishi,tanno96}. It is
also worth observing that $3$-$(\alpha,\delta)$-Sasaki manifolds admit an underlying
quaternionic contact structure which is quaternionic contact Einstein in the sense of
the definition given in \cite{IMV14}; this allows us to determine the Ricci tensor
of the Riemannian metric $g$ (Proposition \ref{prop.3-AD-Sasaki-is-qc-Einstein}).
In particular, a $3$-$(\alpha,\delta)$-Sasaki manifold is Einstein either if
$\alpha=\delta$ (the well-known 3-$\alpha$-Sasaki case) or if $\delta=(2n+3)\alpha$, where
$\dim M=4n+3$.

A second reason why we are interested in $3$-$(\alpha,\delta)$-Sasaki
manifolds is that they provide a large class of \emph{canonical} almost $3$-contact
metric manifolds. The defining conditions of what we call a canonical structure will
be justified by Theorem \ref{theo_canonical}, where we prove that these are exactly
the manifolds admitting a unique  canonical connection. To define them,
we need to introduce the auxiliary  tensor fields $A_{ij}$ ($i,j=1,2,3$)
\begin{equation*}
A_{ij}(X,Y)\ :=\
g(({\mathcal L}_{\xi_j}\varphi_i)X,Y)+d\eta_j(X,\varphi_i Y)+d\eta_j(\varphi_i X,Y)\quad
\forall X,Y\in \Gamma(\H), \ \H \, :=\, \bigcap_{i=1}^3 \ker \eta_i.
\end{equation*}
Here ${\mathcal L}_{\xi_j}$ denotes the
Lie derivative with respect to $\xi_j$. We also put $A_i:=A_{ii}$. We say that an
almost $3$-contact metric manifold $(M,\varphi_i,\xi_i,\eta_i, g)$  admits
a \emph{Reeb Killing function} if there exists a smooth function $\beta\in C^\infty(M)$
such that for every $X,Y\in\Gamma(\H)$ and every even permutation $(i,j,k)$ of $(1,2,3)$,
\[
A_{i}(X,Y)=0,\qquad A_{ij}(X,Y)=-A_{ji}(X,Y)=\beta\Phi_k(X,Y).
\]
The intrinsic meaning of $\beta$ is subtler than that of $\delta$; one key property is
that the Reeb Killing function controls the derivatives of the structure
tensors $\varphi_i,\xi_i$, and $\eta_i$ with respect to the canonical connection
(see Remark \ref{derivatives}). We call $(M,\varphi_i,\xi_i,\eta_i,g)$ a
\emph{canonical almost $3$-contact metric manifold} if it admits a Reeb Killing
function $\beta$, all $\xi_i$ are Killing vector fields, the Nijenhuis tensors
$N_{\varphi_i}$ are skew-symmetric on $\H$, and
$N_{\varphi_i}- \varphi_i\circ d\Phi_i=N_{\varphi_j}-\varphi_j\circ d\Phi_j$ on $\Gamma(\H)$
for all $i,j=1,2,3$.
As a special case, a canonical  almost $3$-contact metric manifold will be called
\emph{parallel} if its Reeb Killing function $\beta$ vanishes; why this case is of interest
will again be  explained in the second part. We prove that
 $3$-$(\alpha,\delta)$-Sasaki manifolds are always canonical with Reeb Killing
function $\beta=2(\delta-2\alpha)$ in Corollary \ref{cor.3-AD-Sasaki-implies-canonical},
and that $3$-$\delta$-cosymplectic manifolds are always hypernormal,
 canonical and parallel
(but not $3$-$(\alpha,\delta)$-Sasakian by definition) in Corollary \ref{cor.3-d-sympl-is-pc}.

\medskip
We end this summary with a figure reviewing the different almost $3$-contact metric
structures that we shall discuss, trying to catch their most relevant features
(Figure \ref{fig.classes}).
The very interesting isolated point $S^7$ shall be treated in Example \ref{ex.S7}.

\begin{figure}[h!]
\begin{center}
\psfrag{parallel1}{{\small parallel $(\delta=2\alpha)$}}
\psfrag{parallel2}{ parallel $(\beta=0)$}
\psfrag{a}{$\alpha$}\psfrag{d}{$\delta$}
\psfrag{3S}{{\small 3-Sasaki manifolds}}\psfrag{3aS}{{\small3-$\alpha$-Sasaki m.}}
\psfrag{3adS}{3-$(\alpha,\delta)$-Sasaki m.}
\psfrag{hyper}{hypernormal canonical almost $3$-contact metric m.}
\psfrag{can}{canonical almost $3$-contact metric manifolds}
\psfrag{S7}{$S^7$}\psfrag{3dcosymplectic}{3-$\delta$-cosymplectic m.}
\psfrag{degenerate}{{\small degenerate}}
\psfrag{HKT}{HKT $\times G$}
\psfrag{exa1}{{\small exa.\,on nilpotent}}
\psfrag{exa2}{{\small Lie groups}}
\includegraphics[width=10cm]{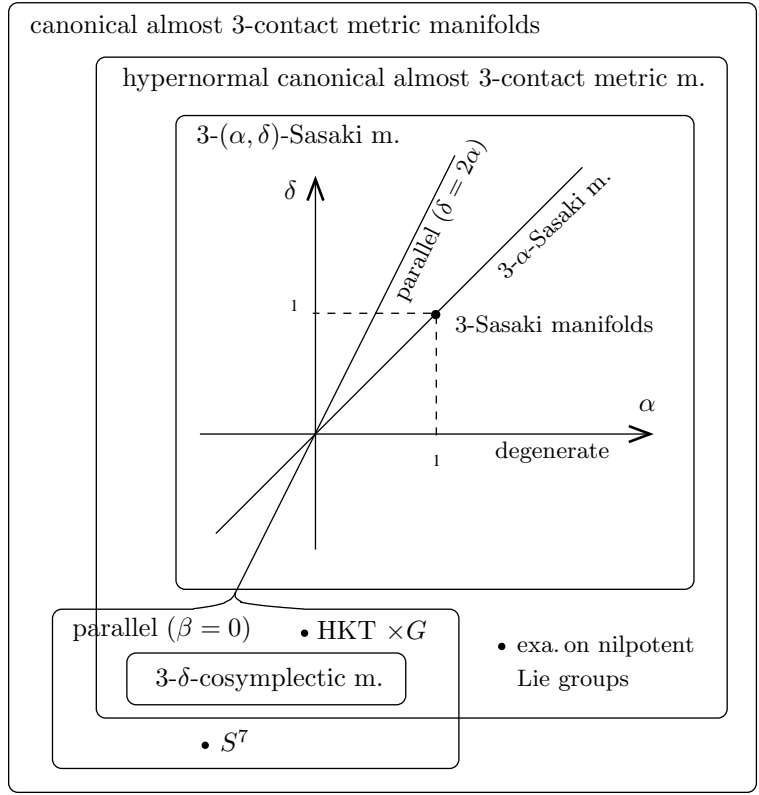}
\end{center}
\caption{The different classes of almost $3$-contact metric manifolds}\label{fig.classes}
\end{figure}

\subsubsection*{Part two (Sections 3--5)}
%----------------------------------------
%
Given a $G$-structure on a Riemannian manifold $(M,g)$,
a \emph{characteristic connection} denotes a metric connection with totally skew-symmetric
torsion (briefly, with \emph{skew torsion}) preserving the $G$-structure. For instance, an
almost hermitian manifold $(M,J,g)$ admits a (unique) hermitian connection with skew
torsion if and only if the Nijenhuis tensor is totally skew-symmetric. Similar results hold
for other geometries, for example  hyperK\"ahler manifolds with torsion, also known as
HKT-manifolds \cite{GP00}.

In the context of almost contact geometry Friedrich and Ivanov \cite{FrIv} provided
necessary and sufficient conditions for an almost contact metric manifold
$(M,\varphi,\xi,\eta,g)$ to admit a characteristic connection (which then turns out
to be unique): the Reeb vector field $\xi$ has to  be Killing and the tensor field
$N_\varphi$ has to be skew-symmetric (see Theorem
\ref{theo-contact} for details).  For example the characteristic connection
of a Sasakian manifold has torsion $T=\eta\wedge d\eta$.

It is well-known, however, that a unique characteristic connection cannot be defined in
any naive way for almost $3$-contact metric manifolds. This is easiest seen
 by looking at the $3$-Sasaki situation: In this case each of the three Sasakian
structures $(\varphi_i,\xi_i,\eta_i,g)$, $i=1,2,3$, admits a unique characteristic connection with torsion $T_i=\eta_i\wedge d\eta_i$.
However, these three connections do not coincide, hence there does not exist a
metric connection with skew torsion preserving all three Sasaki structures. To overcome this
difficulty,  a notion of canonical connection was proposed for a $7$-dimensional
$3$-Sasakian manifold in \cite{Ag-Fr} by making a detour to the canonical (cocalibrated)
$G_2$-structure associated to the $3$-Sasakian structure. This canonical connection has
torsion $T=\sum_{i=1}^3\eta_i\wedge d\eta_i$ and preserves the \emph{vertical} and
\emph{horizontal} distributions, denoted by $\mathcal V$ and $\mathcal H$ respectively,
where $\mathcal V$ is the distribution spanned by the Reeb vector fields
$\xi_1,\xi_2,\xi_3$. Furthermore, this  connection has parallel torsion and admits a
parallel spinor field that induces the three Riemannian Killing spinor 
fields of the $3$-Sasaki manifold.

A second remarkable example of almost $3$-contact metric manifolds admitting a canonical
connection is given by quaternionic Heisenberg groups. In \cite{Ag-F-S} the authors study
the geometry of these nilpotent Lie groups, describing natural left invariant almost
$3$-contact metric structures $(\varphi_i,\xi_i,\eta_i,g_\lambda)$, $\lambda>0$, which can
be defined in all dimensions $4n+3$. It is shown that the metric connection with skew
torsion $T=\sum_{i=1}^3\eta_i\wedge d\eta_i-4\lambda\eta_1\wedge\eta_2\wedge\eta_3$ preserves
the horizontal and vertical distributions and equips the group with a naturally reductive
homogeneous structure, thus highlighting again its importance. In the $7$-dimensional
case this connection can also be obtained by means of a cocalibrated $G_2$-structure.

The main objective of our study is to find good connections on almost $3$-contact metric
manifolds that generalize the canonical connections of the described examples.
We begin with the observation that the choice of the Reeb vector fields
$\xi_i$ in the vertical distribution is somewhat arbitrary. Rather, any
almost $3$-contact metric manifold $(M,\varphi_i,\xi_i,\eta_i,g)$ carries a sphere
$\Sigma_M$ of almost contact metric structures $(\varphi_a,\xi_a,\eta_a,g)$, with
$\varphi_a=a_1\varphi_1+a_2\varphi_2+a_3\varphi_3$ for every $a=(a_1,a_2,a_3)\in S^2$.
The horizontal and the vertical distributions are $\varphi$-invariant for every
$\varphi\in\Sigma_M$. Now, if $(\varphi,\xi,\eta,g)$ is a structure in $\Sigma_M$, a
metric connection $\nabla$ with skew torsion on $M$ will be called a
\emph{$\varphi$-compatible connection} if it preserves the splitting
$TM=\H\oplus \V$ of the tangent bundle and $(\nabla_X\varphi)Y=0$ for all horizontal
vector fields $X, Y$. In Theorem \ref{theo_compatible} we provide necessary and
sufficient conditions for the existence of $\varphi$-compatible connections, one of which
being the total skew-symmetry of the tensor field $N_\varphi$ on $\mathcal H$; the other
ones involve the Lie derivatives of the Riemannian metric $g,$ and they are satisfied
in the special case where the three Reeb vector fields are Killing. We also show that
if  $M$ admits $\varphi_i$-compatible connections for every $i=1,2,3$, then $M$ admits
$\varphi$-compatible connections for \emph{every}  structure $\varphi\in\Sigma_M$.

Despite the
good behavior of $\varphi$-compatibility with respect to the associated sphere $\Sigma_M$, 
this notion is
still too weak, since $\varphi$-compatible connections are not uniquely determined.
They are parametrized by  smooth functions $T(\xi_1,\xi_2,\xi_3)=:\gamma \in C^\infty(M)$,
where $T$ is the torsion of the connection. We call $\gamma$ the \emph{parameter function}
of the connection.

A suggestion for requiring some further conditions on the connection comes from the case
when the Reeb vector fields are Killing. In this case, given a $\varphi$-compatible
connection $\nabla$ with parameter function $\gamma$, the $\nabla$-derivative of each
$\xi_i$ is completely determined by $\gamma$ and the Reeb commutator function $\delta$
through
\[\nabla_X\xi_i=\textstyle\frac{2\delta+\gamma}{2}(\eta_k(X)\xi_j-\eta_j(X)\xi_k)\]
for every vector field $X$, and $(i,j,k)$ even permutation of $(1,2,3)$
(Proposition \ref{prop_Killing}). This suggests the idea that one can require the
covariant derivatives of the structure tensors $\varphi_i$ to behave in a similar way.
Therefore, we look for metric connections $\nabla$ with skew torsion such that
\begin{equation}\label{intro_derivatives}
\nabla_X\varphi_i=\beta(\eta_k(X)\varphi_j-\eta_j(X)\varphi_k)
\end{equation}
for some smooth function $\beta$ and for every even permutation $(i,j,k)$ of $(1,2,3)$.
In fact, in Theorem \ref{theo_canonical} we prove that an almost $3$-contact metric
manifold $(M,\varphi_i,\xi_i,\eta_i,g)$ admits a metric connection $\nabla$ with skew
torsion satisfying \eqref{intro_derivatives} if and only if it is canonical with
Reeb Killing function $\beta$.
If such a connection $\nabla$ exists, it is unique and it is $\varphi$-compatible for every
   structure $\varphi$ in the associated sphere $\Sigma_M$. The parameter function
of $\nabla$ is $\gamma=2(\beta-\delta)$, $\delta$ being the Reeb commutator function.
We call $\nabla$ the \emph{canonical connection} of $M$, and show that  the covariant
derivatives of the other structure tensors are given by
\[\nabla_X\xi_i=\beta(\eta_k(X)\xi_j-\eta_j(X)\xi_k),\qquad \nabla_X\eta_i
=\beta(\eta_k(X)\eta_j-\eta_j(X)\eta_k).\]
One can notice the analogy of \eqref{intro_derivatives} with the equation satisfied by the Levi-Civita connection of a quaternion-K\"ahler manifold (see Remark \ref{rem.qK-geometry}).
There are various remarkable properties of canonical almost $3$-contact metric
manifolds and their canonical connection deserving special attention. First, for a
canonical manifold $(M,\varphi_i,\xi_i,\eta,g)$, each structure $(\varphi,\xi,\eta,g)$
in the sphere $\Sigma_M$ admits a characteristic connection (Theorem
\ref{theo_canonical-implies-char}). In particular, if
$\nabla$ is the canonical connection and $\nabla^i$ the characteristic connection of
the structure $(\varphi_i,\xi_i,\eta_i,g)$, their torsions $T$ and $T_i$ are related by
(Theorem \ref{theo_canonical--char-torsion})
\[T-T_i=-\beta(\eta_j\wedge\Phi_j+\eta_k\wedge\Phi_k)\]
where $(i,j,k)$ is an even permutation of $(1,2,3)$. A surprising situation occurs for
parallel canonical manifolds ($\beta=0$): for them, the three characteristic connections
$\nabla^i$ are identical, and they coincide with the canonical connection.
Hence, all structure tensors $\varphi_i, \xi_i,\eta_i$ are $\nabla$-parallel. It is
surprising that this fact was not discovered before.

Focusing on the canonical connection of a $3$-$(\alpha,\delta)$-Sasaki manifold,
for $3$-Sasakian manifolds our canonical connection coincides, as desired, with the
connection defined in the $7$-dimensional case in \cite{Ag-Fr}. Similarly,
on quaternionic  Heisenberg groups our canonical connection coincides  with the
connection defined in \cite{Ag-F-S}. We also show that the canonical connection
$\nabla$ of a $3$-$(\alpha,\delta)$-Sasaki manifold has parallel torsion
(Theorem \ref{thm.canonical-3AD-Sasaki}),
we determine its Ricci tensor, and discuss the  $\nabla$-Einstein
condition (Theorem \ref{theo.nabla-Ricci}).

A `good' connection on $M$ should induce a good connection on the cone. Recall
that the metric cone of a $3$-Sasaki manifold is hyper-K\"ahler.
If the Reeb Killing function is constant and strictly negative, the
canonical connection $\nabla$ allows us to define a metric connection with skew torsion
$\bar\nabla$ on the cone $(\bar M,\bar g)=(M\times\mathbb{R}^+,a^2r^2g+dr^2)$, $a=-\beta/2>0$,
such that $\bar\nabla J_1=\bar\nabla J_2=\bar\nabla J_3=0$, where $J_1$, $J_2$, $J_3$
are almost hermitian structures naturally defined on $(\bar M, \bar g)$, and such
that $J_1J_2=J_3=-J_2J_1$: hence, we obtain a hyperhermitian structure.
 If furthermore, $(M,\varphi_i,\xi_i,\eta_i,g)$ is a
$3$-$(\alpha,\delta)$-Sasaki manifold, the cone is an HKT-manifold, a class of manifolds
that is much larger than the class of hyper-K\"ahler manifolds (see
Section \ref{sec.cone}).

Finally, we consider 7-dimensional $3$-$(\alpha,\delta)$-Sasaki manifolds (Section
\ref{sec.G2-and-3AD-S}) in order to investigate their relationship to $G_2$-geometry,
which only exists in this  dimension.
We prove that the canonical connection coincides with the characteristic connection
of a cocalibrated $G_2$-structure; as such, it admits a parallel spinor field $\psi_0$.
We show that $\psi_0$ and the three Clifford products $\psi_i:=\xi_i\cdot\psi_0$
are generalized Killing spinor fields, and we compute their generalized Killing numbers
(for $\alpha=\delta$, they coincide with the three Riemannian Killing spinors of a
$3$-$\alpha$-Sasaki manifold).

The appendix is devoted to the discussion of further examples. We  describe left
invariant almost $3$-contact metric structures on various nilpotent Lie groups, providing
examples of canonical structures which are not $3$-$(\alpha,\delta)$-Sasaki, and
non-canonical structures admitting $\varphi_i$-compatible connections.

\begin{notation}
%---------------
Throughout this text, $\delta$ denotes the Reeb commutator function (Definition
\ref{df.Reebcf}), which can be a real constant, and
sometimes, it appears as factor in front of a differential form. It shouldn't be confused with
a codifferential (we don't need any in this paper). 
For further ease of notation, we will often set $\eta_{ij}:=\eta_i\wedge\eta_j$ etc.,
in particular when formulas tend to become heavy.
\end{notation}
\medskip
\noindent\textbf{Acknowledgements.} G. Dileo acknowledges the financial support of DAAD for a 
research stay at Philipps-Universit\"at Marburg in the period April-June 2016, under the 
Programme \emph{Research Stays for University Academics and Scientists}. She thanks 
Philipps-Universit\"at for its kind hospitality.
%-----------------------------------------------------------------------------------------
\subsection{Review of almost contact and $3$-contact metric manifolds}
\label{sec.review}
%-----------------------------------------------------------------------------------------
%
We review some basic definitions and properties on almost contact metric manifolds.
This serves mainly as a  reference, but it is assorted by comments relevant to our work
 as we move on.
\begin{df}
%----------
An \emph{almost contact manifold} is a $(2n+1)$-dimensional smooth manifold
$M$ endowed with a structure $(\varphi,\xi,\eta)$, where $\varphi$ is a $(1,1)$-tensor
field, $\xi$ a vector field, and $\eta$ a $1$-form, such that
\begin{equation*} %\label{defi:cont}
\varphi^2=-I+\eta\otimes \xi,\ \  \eta(\xi)=1,
\end{equation*}
implying that $\varphi \xi =0$, $\eta \circ \varphi =0$, and $\varphi$ has rank $2n$.
The tangent bundle of $M$ splits as $TM=\H\oplus\langle\xi\rangle$, where $\H$ is the
$2n$-dimensional distribution defined by $\H=\mathrm{Im}(\varphi)=\ker\eta$.
\end{df}
The vector field $\xi$ is called the \emph{characteristic} or \emph{Reeb vector field}.
The almost contact structure is said to be
\emph{normal} if $ N_\varphi:=[\varphi,\varphi]+d\eta\otimes\xi$ vanishes,
where $[\varphi,\varphi]$ is the Nijenhuis torsion of $\varphi$ \cite{BLAIR}.
More precisely, for any vector fields $X$ and $Y$, $N_\varphi$ is given by
\bdm
N_\varphi(X,Y)=
[\varphi X,\varphi Y]+\varphi^2[X,Y]-\varphi[\varphi X,Y]-\varphi[X,\varphi Y]+d\eta(X,Y)\xi.
\edm
It is known that any almost contact manifold admits a compatible metric, that is a
Riemannian metric  $g$ such that, for every $X,Y\in{\frak X}(M)$,
$g(\varphi X,\varphi Y)=g(X,Y)-\eta (X) \eta(Y)$.
Then $\eta=g(\cdot,\xi)$ and ${\mathcal H}=\langle \xi\rangle^\perp$. The manifold
$(M,\varphi,\xi,\eta,g)$ is called an \emph{almost contact metric manifold}.

An \emph{$\alpha$-contact metric manifold} is defined as an almost contact metric manifold
such that
\be\label{eq.alpha-contact}
d\eta\, =\, 2\alpha\Phi, \quad \alpha\in\R^*,
\ee
where $\Phi$ is the fundamental $2$-form defined by
$\Phi(X,Y)=g(X,\varphi Y)$; a $1$-contact metric manifold is
just called a \emph{contact metric manifold} for short\footnote{The alternative name
\emph{almost $\alpha$-Sasakian manifold} can be found in the literature for what we call an
$\alpha$-contact metric manifold, see for example  \cite{janssens}; however, we find this notion less suggestive.}; the $1$-form $\eta$ turns then
out to be a
\emph{contact form}, in the sense that $\eta\wedge (d\eta)^n\ne 0$ everywhere on $M$.
An \emph{$\alpha$-Sasakian manifold} is defined as a normal $\alpha$-contact metric
manifold, and again such a manifold with $\alpha=1$ is called a
 \emph{Sasakian manifold}.
A more general class of $\alpha$-Sasakian manifolds is given by
\emph{quasi-Sasakian manifolds}, defined as normal almost contact metric manifolds with
closed $2$-form $\Phi$. We recall that the Reeb vector field of a
(quasi)-Sasakian or $\alpha$-Sasakian manifold is always Killing.
As a  comprehensive introduction to Sasakian geometry, we recommend the monography
\cite{Boyer&Galicki}. For some recent results, we refer to \cite{CNY15}.

We recall now some basic facts about connections with totally skew-symmetric torsion---we
refer to \cite{Ag} for further details. If  $(M,g)$ is a Riemannian manifold,  a metric
connection $\nabla$ with torsion $T$ is said to have \emph{totally skew-symmetric torsion},
or \emph{skew torsion} for short, if the $(0,3)$-tensor field $T$ defined by
\[T(X,Y,Z)=g(T(X,Y),Z)\]
is a $3$-form. The relation between $\nabla$ and the Levi-Civita connection
$\nabla^g$ is then given by
\begin{equation}\label{nabla}
\nabla_XY=\nabla^g_XY+\frac{1}{2}T(X,Y).
\end{equation}
In \cite{FrIv} T.\,Friedrich ad S.\,Ivanov proved the following theorem concerning
\emph{characteristic connections} on almost contact metric manifolds, i.\,e.~metric
connections with skew torsion parallelizing all structure tensors.
\begin{theo}\label{theo-contact}
%--------------------------------
Let $(M,\varphi,\xi,\eta,g)$ be an almost contact metric manifold. It admits a metric
connection $\nabla$ with skew  torsion and $\nabla\eta=\nabla\varphi=0 $ if and only if
$N_\varphi$ is totally skew-symmetric and if  $\xi$ is a Killing vector field. The
connection $\nabla$ is then uniquely determined and its torsion is given by
\bdm
%\label{Tcontact} % not used so far
T=\eta\wedge d\eta+N_{\varphi}+d^{\varphi}\Phi-\eta\wedge (\xi\lrcorner N_{\varphi}),
\edm
where $d^{\varphi}\Phi$ is defined as
$d^{\varphi}\Phi(X,Y,Z)\, :=\, -d\Phi(\varphi X,\varphi Y,\varphi Z)$.
\end{theo}
For example, quasi-Sasakian manifolds admit a unique
characteristic connection whose torsion is given by $T=\eta\wedge d\eta$.
For later, let us observe that $\nabla\eta=\nabla\varphi=0 $ implies that
the characteristic connection preserves the
distributions $\H$ and $\V$ -- a property we shall like to have later on for almost
$3$-contact manifolds as well. Further results on the characteristic connection of
almost contact metric manifolds may be found in \cite{puhle12, puhle13}; in particular,
one finds there a detailed investigation for special classes of manifolds.
The special situation of normal almost contact metric manifolds with Killing Reeb vector field
was investigated in \cite{CM14, HTY13}, leading to a notion of Sasaki manifolds with torsion.
\begin{df}
%-----------
An  \emph{almost $3$-contact manifold}  is a differentiable  manifold $M$ of dimension
$4n+3$ endowed with three almost contact structures $(\varphi_i,\xi_i,\eta_i)$, $i=1,2,3$,
satisfying the following relations,
\begin{equation}
\begin{split}\label{3-sasaki}
\varphi_k=\varphi_i\varphi_j-\eta_j\otimes\xi_i=-\varphi_j\varphi_i+\eta_i\otimes\xi_j,\quad\\
\xi_k=\varphi_i\xi_j=-\varphi_j\xi_i, \quad
\eta_k=\eta_i\circ\varphi_j=-\eta_j\circ\varphi_i,
\end{split}
\end{equation}
for any even permutation $(i,j,k)$ of $(1,2,3)$ \cite{BLAIR}.
The tangent bundle of $M$ splits as $TM=\H\oplus\V$, where
\[
\H \, :=\, \bigcap_{i=1}^{3}\ker\eta_i,\qquad
\V\, :=\, \langle\xi_1,\xi_2,\xi_3\rangle.
\]
In particular $\H$ has rank $4n$. We call any vector belonging to the distribution
$\H$ \emph{horizontal} and any vector belonging to the distribution $\mathcal V$
\emph{vertical}. The manifold is said to  be
\emph{hypernormal} if each  almost contact structure
$(\phi_i,\xi_i,\eta_i)$ is normal.
In \cite{yano1} it was proved that if two of the almost contact structures are normal,
then so is the third.
\end{df}
Any almost $3$-contact manifold admits a Riemannian metric $g$ which is compatible with
each of the three structures. Then $M$ is said to be an \emph{almost $3$-contact metric
manifold} with structure $(\varphi_i,\xi_i,\eta_i,g)$, $i=1,2,3$.
For ease of notation, we will just say that $(M,\varphi_i,\xi_i,\eta_i, g)$ is
an almost $3$-contact
metric manifold, and it is self-understood that the index is running from $1$ to $3$.
The subbundles $\H$
and $\mathcal V$ are orthogonal with respect to $g$ and the three Reeb vector
fields $\xi_1,\xi_2,\xi_3$ are orthonormal---the structure group of the tangent
bundle is in fact reducible to $\mathrm{Sp}(n)\times \{1\}$ \cite{Kuo70}; this implies in
particular that each almost $3$-contact manifold is spin.
 The following remarkable subclasses of almost $3$-contact metric
manifolds will be of particular importance to our work:
\begin{enumerate}[1)]
\item
A \emph{$3$-(quasi)-Sasakian} resp.\,\emph{$3$-$\alpha$-Sasakian manifold} is an almost
$3$-contact metric manifold for which each of the three structures is
(quasi)-Sasakian resp.\,{$\alpha$-Sasakian} \cite{CM-DN-D1,CM-DN-D}. A remarkable
result due to T. Kashiwada states that if the three structures are contact
metric structures, then the manifold is
$3$-Sasakian \cite{kashiwada}. Many results on the topology of $3$-Sasakian manifolds are 
available, see for example \cite{Galicki&S96}.
\item
A \emph{$3$-cosymplectic manifold} is an almost
$3$-contact metric manifold satisfying $d\eta_i=0$, $d\Phi_i=0$. It is known
that such a structure is hypernormal (see \cite[Theorem 4.13]{FIP}), so that each of the three structures is cosymplectic. It is also known that any $3$-cosymplectic manifold
is locally isometric to the Riemannian product of a hyper-K\"ahler manifold and
a $3$-dimensional flat abelian Lie group \cite{CM-DN}.
\end{enumerate}

Later on, we shall see that the new classes of  $3$-$(\alpha,\delta)$-Sasaki manifolds
(Definition \ref{df.3ad-Sasaki}) and $3$-$\delta$-cosymplectic manifolds (Definition \ref{df.3d-cosymplectic})
generalize  $3$-$\alpha$-Sasaki and $3$-cosymplectic manifolds, respectively, by using a vertical extra term.
%
%
%------------------------------------------------------------------------------------------
\subsection{The sphere of associated almost contact structures}
%------------------------------------------------------------------------------------------
%
Given an almost $3$-contact metric manifold, one can define a  sphere of almost contact
structures containing  $\pm \varphi_i$ ($i=1,2,3$) as antipodal points \cite{CM-DN-Y}.
This sphere is a canonical object to consider, since the choice of the elements
$\varphi_i$  is somewhat arbitrary.
\begin{df}
%------------
For any almost $3$-contact metric manifold $(M,\varphi_i,\xi_i,\eta_i, g)$
we define its associated sphere $\Sigma_M$ of almost contact structures
\bdm
\Sigma_M\, :=\, \{\varphi_a:= a_1\varphi_1+a_2\varphi_2+ a_3\varphi_3 \ |
\ a=(a_1,a_2,a_3)\in S^2 \}
\edm
as well as the associated bundle of endomorphisms
\bdm
\Upsilon_M\, :=\, \{\varphi_a:= a_1\varphi_1+a_2\varphi_2+ a_3\varphi_3 \ |
\ a=(a_1,a_2,a_3)\in \R^3 \}.
\edm
For any $\varphi_a\in \Sigma_M$, its Reeb vector field and dual $1$-form are defined respectively
as
\[
\xi_a \, :=\, a_1\xi_1+a_2\xi_2+a_3\xi_3,\quad \eta_a \, :=\, a_1\eta_1+a_2\eta_2+a_3\eta_3.
\]
The Riemannian metric $g$ is compatible with all the structures $(\varphi_a,\xi_a,\eta_a)$ 
(see \cite{CM-DN-Y} for more details).
When it has no importance, the index $a$ will be omitted.
\end{df}
Observe that for any $\varphi\in \Sigma_M$, the distributions $\H$ and  $\V$
are $\varphi$-invariant. Thus, $\varphi$ encodes more geometric information than just the
choice of an almost contact structure on $M$.
%Setting
%$J:=\varphi|_{\H}:\H\to \H$, one has $J^2=-I$, so that each structure $(\H,J)$ is an
%almost $CR$-structure on $M$ of type $(4n,3)$.
%

As seen in Theorem \ref{theo-contact}, a crucial property is whether the Nijenhuis tensor is
skew-symmetric. Although it is not, for an almost contact metric structure $\varphi$ in the
associated sphere $\Sigma_M$, just the sum of the Nijenhuis tensors of the $\varphi_i$'s,
we shall prove next that it is skew-symmetric if this property holds for each $N_{\varphi_i}$.
So, consider $\varphi\in\Sigma_M$, and define  tensor fields $N_{i,j}$
\[
N_{i,j}:=[\varphi_i,\varphi_j]+d\eta_i\otimes\xi_j+d\eta_j\otimes\xi_i,
\]
where
\begin{align*}
[\varphi_i,\varphi_j](X,Y)&:=[\varphi_i X,\varphi_jY]
-\varphi_i[\varphi_jX,Y]-\varphi_j[X,\varphi_iY]+[\varphi_jX,\varphi_iY]\\
&\quad-\varphi_j[\varphi_iX,Y]-\varphi_i[X,\varphi_jY]+(\varphi_i\varphi_j
+\varphi_j\varphi_i)[X,Y].
\end{align*}
In particular, $N_{i,i}=2N_{\varphi_i}$.
Notice that,  using \eqref{3-sasaki}, one has $\forall X,Y\in\Gamma(\H) $
\begin{equation}\label{Nij}
\begin{split}
N_{i,j}(X,Y)&=[\varphi_i X,\varphi_jY]-\varphi_i[\varphi_jX,Y]-\varphi_j[X,\varphi_iY]\\
&\quad +[\varphi_jX,\varphi_iY]-\varphi_j[\varphi_iX,Y]-\varphi_i[X,\varphi_jY].
\end{split}
\end{equation}
The following crucial equation was proved in
\cite{CM-DN-Y}:
\be\label{ident-CM-DN-Y}
N_\varphi=N_{\varphi_1}+N_{\varphi_2}+N_{\varphi_3}+a_1a_2N_{1,2}+a_1a_3N_{1,3}+a_2a_3N_{2,3}.
\ee
If the almost $3$-contact metric structure is hypernormal, the tensor fields $N_{i,j}$
are all vanishing \cite{CM-DN-Y}, and thus $N_\varphi=0$. We say that $N_\varphi$ is
\emph{skew-symmetric on $\H$} if the $(0,3)$-tensor field defined by
\bdm
N_\varphi(X,Y,Z)\ =\ g(N_\varphi(X,Y),Z)\quad \forall X,Y,Z\in \Gamma(\H)
\edm
is a $3$-form on $\H$. We collect a few equivalent conditions for  $N_\varphi$ to be
skew-symmetric on $\H$ which we will prove to be useful.
\begin{lem}\label{skew}
%------------------------
Let $(M,\varphi_i,\xi_i,\eta_i, g)$   be an almost $3$-contact metric manifold,
$\varphi\in\Sigma_M$. The following conditions are equivalent:
\begin{enumerate}[\normalfont 1)]
\item $N_\varphi$ is skew-symmetric on $\H$;
\item For any $X,Y\in\Gamma(\H)$: \  $g((\nabla^g_X\varphi)X,Y)
=g((\nabla^g_{\varphi X}\varphi)\varphi X,Y)$;
\item For any $X,Y,Z\in\Gamma (\H)$: $$g((\nabla^g_X\varphi )Y
+(\nabla^g_Y\varphi )X,Z)=g((\nabla^g_{\varphi X}\varphi )\varphi Y
+(\nabla^g_{\varphi Y}\varphi )\varphi X,Z);$$
\item For any $Y,Z\in\Gamma(\H)$: \  $g((\nabla^g_{\varphi Z}\varphi)Z,Y)
+g((\nabla^g_{Z}\varphi)\varphi Z,Y)=0$.
\end{enumerate}
\end{lem}
\begin{proof}
%--------------
For the equivalence of 1), 2), and 3), see \cite[Proposition 3.1]{DL}. If 2) holds,
we get 4) by applying it to $X=Z+\varphi Z$. Conversely, applying 4) for $Z=X+\varphi X$,
we obtain 2).
\end{proof}
We shall now prove that if each $N_{\varphi_i}$ is
skew-symmetric on $\H$, then the tensor fields $N_{i,j}$ are skew-symmetric on $\H$ as well.
We proceed in two steps.
\begin{lem}\label{lemma_skew}
%----------------------------
Let $(M,\varphi_i,\xi_i,\eta_i, g)$   be an almost $3$-contact metric manifold.
Let $i,j=1,2,3$, with $i\ne j$. The tensor field $N_{i,j}$ is skew-symmetric on $\H$ if and
only if
\bdm
g((\nabla^g_{\varphi_jY}\varphi_i)Y+(\nabla^g_{Y}\varphi_i)\varphi_jY,X)
+g((\nabla^g_{\varphi_iY}\varphi_j)Y+(\nabla^g_{Y}\varphi_j)\varphi_iY,X)=0
\edm
for every $X,Y\in\Gamma(\H)$.
\end{lem}
\begin{proof}
%--------------
Since $N_{i,j}=N_{j,i}$ we can fix an even permutation $(i,j,k)$ of $(1,2,3)$.
By \eqref{Nij}, one can check that
\begin{eqnarray*}
N_{i,j}(X,Y)  & = & (\nabla^g_{\varphi_i X}\varphi_j)Y-(\nabla^g_{\varphi_i Y}\varphi_j)X
+(\nabla^g_{\varphi_jX}\varphi_i)Y-(\nabla^g_{\varphi_jY}\varphi_i)X\\
&&{}+\varphi_i(\nabla^g_Y(\varphi_jX))-\varphi_i(\nabla^g_X(\varphi_jY))
+\varphi_j(\nabla^g_Y(\varphi_iX))-\varphi_j(\nabla^g_X(\varphi_iY))
\end{eqnarray*}
for every $X,Y\in\Gamma(\H)$, and thus
\begin{eqnarray*}
g(N_{i,j}(X,Y),Y)&= &
{}-g((\nabla^g_{\varphi_i Y}\varphi_j)X,Y)-g((\nabla^g_{\varphi_jY}\varphi_i)X,Y)
+g(\varphi_i\nabla^g_Y(\varphi_jX),Y)\\
&&{}+ g(\nabla^g_X(\varphi_jY),\varphi_iY)
+g(\varphi_j\nabla^g_Y(\varphi_iX),Y)+g(\nabla^g_X(\varphi_iY),\varphi_jY).
\end{eqnarray*}
Now, $g(\varphi_iY,\varphi_jY)=-g(Y,\varphi_i\varphi_jY)=-g(Y,\varphi_k Y)=0$. Then,
\[g(\nabla^g_X(\varphi_jY),\varphi_iY)+g(\nabla^g_X(\varphi_iY),\varphi_jY)
=X(g(\varphi_iY,\varphi_jY))=0.\]
Furthermore,
\begin{align*}
&g(\varphi_i\nabla^g_Y(\varphi_jX),Y)+g(\varphi_j\nabla^g_Y(\varphi_iX),Y)\\
&=g(\varphi_i(\nabla^g_Y\varphi_j)X+\varphi_i\varphi_j(\nabla^g_YX),Y)
+g(\varphi_j(\nabla^g_Y\varphi_i)X+\varphi_j\varphi_i(\nabla^g_YX),Y)\\
&=-g((\nabla^g_Y\varphi_j)X,\varphi_iY)-g((\nabla^g_Y\varphi_i)X,\varphi_jY),
\end{align*}
where we took into account that $(\varphi_i\varphi_j+\varphi_j\varphi_i)Y=0$.
We deduce that
\begin{eqnarray*}
\lefteqn{g(N_{i,j}(X,Y),Y) =}\\
&=&  g((\nabla^g_{\varphi_iY}\varphi_j)Y,X)+g((\nabla^g_{\varphi_jY}\varphi_i)Y,X)
+g((\nabla^g_{Y}\varphi_j)\varphi_iY,X)+g((\nabla^g_{Y}\varphi_i)\varphi_jY,X),
\end{eqnarray*}
which gives the result.
\end{proof}
\begin{prop}\label{prop.N-skew-if-Ni-skew}
%-----------------------------------------
Let $(M,\varphi_i,\xi_i,\eta_i, g)$  be an almost $3$-contact metric manifold
such that $N_{\varphi_1}$, $N_{\varphi_2}$, and $N_{\varphi_3}$ are skew-symmetric on $\H$. Then
 $N_{i,j}$ is skew-symmetric on $\H$ for every $i,j=1,2,3$. In particular,
$N_{\varphi}$ is skew-symmetric on $\H$ for any $\varphi$ in the associated sphere $\Sigma_M$.
\end{prop}
\begin{proof}
%-------------
We first prove some auxiliary formulas. In the following we always consider vector fields
$X,Y,Z\in\Gamma(\H)$. Let $(i,j,k)$ be an even permutation of $(1,2,3)$. Since
$\varphi_i X=\varphi_j\varphi_kX$ and $\varphi_i^2X=-X$,  one easily checks that
\begin{equation}\label{help1}
g((\nabla^g_X\varphi_i)Y,Z)=g((\nabla^g_X\varphi_j)\varphi_kY
+\varphi_j(\nabla^g_X\varphi_k)Y,Z),
\end{equation}
\begin{equation}\label{help2}
g((\nabla^g_X\varphi_i)\varphi_iY+\varphi_i(\nabla^g_X\varphi_i)Y,Z)=0.
\end{equation}
$N_{\varphi_i}$ being skew-symmetric on $\H$, we have by Lemma \ref{skew} 4),
\begin{equation}\label{help3}
g((\nabla^g_{\varphi_iY}\varphi_i)Y,X)+g((\nabla^g_{Y}\varphi_i)\varphi_iY,X)=0.
\end{equation}
Applying the above formula for $Y=\varphi_jZ$ we obtain
\begin{equation}\label{help4}
g((\nabla^g_{\varphi_kZ}\varphi_i)\varphi_jZ,X)+g((\nabla^g_{\varphi_jZ}\varphi_i)\varphi_kZ,X)=0.
\end{equation}
From \eqref{help3} and \eqref{help1} it follows that
\bdm
0\, =\, g((\nabla^g_{\varphi_iY}\varphi_j)\varphi_kY+\varphi_j(\nabla^g_{\varphi_iY}\varphi_k)Y,X)
+g((\nabla^g_{Y}\varphi_j)\varphi_k\varphi_iY
+\varphi_j(\nabla^g_{Y}\varphi_k)\varphi_iY,X),
\edm
and thus
\begin{equation}\label{help5}
0=g((\nabla^g_{\varphi_iY}\varphi_j)\varphi_kY,X)+g((\nabla^g_{Y}\varphi_j)\varphi_jY,X)
-g((\nabla^g_{\varphi_iY}\varphi_k)Y,\varphi_jX)
-g((\nabla^g_{Y}\varphi_k)\varphi_iY,\varphi_jX).
\end{equation}
At this point, using formulas \eqref{help1}, \eqref{help2}, \eqref{help3},
\eqref{help4}, we have
\begin{eqnarray*}
\lefteqn{g ((\nabla^g_{\varphi_iY}\varphi_j)\varphi_kY,X) \ = \ -g((\nabla^g_{\varphi_kY}\varphi_j)\varphi_iY,X)\ = \ }\\
& =& -\, g((\nabla^g_{\varphi_kY}\varphi_k)\varphi_i^2Y,X)
-g(\varphi_k(\nabla^g_{\varphi_kY}\varphi_i)\varphi_iY,X)\\
&=& + \, g((\nabla^g_{\varphi_kY}\varphi_k)Y,X)+g((\nabla^g_{\varphi_kY}\varphi_i)\varphi_iY,\varphi_kX)\\
&=&-\, g((\nabla^g_{Y}\varphi_k)\varphi_kY,X)
-g(\varphi_i(\nabla^g_{\varphi_kY}\varphi_i)Y,\varphi_kX)\\
&=&-\, g(\varphi_j(\nabla^g_{Y}\varphi_k)\varphi_kY,\varphi_jX)
-g((\nabla^g_{\varphi_kY}\varphi_i)Y,\varphi_jX)\\
&=&-g((\nabla^g_{Y}\varphi_i)\varphi_kY,\varphi_jX)
+\, g((\nabla^g_{Y}\varphi_j)\varphi_k^2Y,\varphi_jX)
 - g((\nabla^g_{\varphi_kY}\varphi_i)Y,\varphi_jX)\\
&=&-\, g((\nabla^g_{Y}\varphi_i)\varphi_kY,\varphi_jX)
-g((\nabla^g_{\varphi_kY}\varphi_i)Y,\varphi_jX) + g(\varphi_j(\nabla^g_{Y}\varphi_j)Y,X).
\end{eqnarray*}
Substituting the obtained expression for $g((\nabla^g_{\varphi_iY}\varphi_j)\varphi_kY,X)$
in \eqref{help5}, we have
\bdm
0=g((\nabla^g_{Y}\varphi_i)\varphi_kY,\varphi_jX)
+g((\nabla^g_{\varphi_kY}\varphi_i)Y,\varphi_jX)
+g((\nabla^g_{\varphi_iY}\varphi_k)Y,\varphi_jX)
+g((\nabla^g_{Y}\varphi_k)\varphi_iY,\varphi_jX),
\edm
and thus $N_{k,i}$ is skew-symmetric on $\H$, owing to Lemma \ref{lemma_skew}.
The last claim about the skew-symmetry of $N_\varphi$ now follows  from
identity \eqref{ident-CM-DN-Y}, proved in \cite{CM-DN-Y}.
\end{proof}
We end this section with a lemma that will  subsequently be used several times.
Although the formula does not look very neat, its qualitative claim is important: it states
that the tensor fields $N_{\varphi_i}$ of an almost $3$-contact metric structure can
be expressed in terms of the $1$-forms $\eta_i$ and the fundamental $2$-forms $\Phi_i$.
\begin{lem}\label{lemma_N}
%-------------------------
Let $(M,\varphi_i,\xi_i,\eta_i,g)$ be an almost $3$-contact metric manifold. Then
the following formula holds $\forall X,Y,Z \in \mathfrak{X}(M) $:
\begin{align}\label{N_almost3c}
\lefteqn{N_{\varphi_i}(X,Y,Z)\ =}\\
&=  -d\Phi_j(X,Y,\varphi_jZ)+d\Phi_j(\varphi_iX,\varphi_iY,\varphi_jZ)
+d\Phi_k(X,\varphi_iY,\varphi_jZ)+d\Phi_k(\varphi_iX,Y,\varphi_jZ)\nonumber\\
&  \quad- \eta_i(X) [d\eta_j(\varphi_i Y,\varphi_j Z)+d\eta_k(Y,\varphi_jZ) ]
+\eta_i(Y) [d\eta_j(\varphi_i X,\varphi_j Z)+d\eta_k(X,\varphi_jZ) ] \nonumber\\
&\quad+ \eta_j(Z) [ d\eta_j(X,Y)-d\eta_j(\varphi_i X,\varphi_iY) ]
-\eta_j(Z) [ d\eta_k(X,\varphi_iY)+d\eta_k(\varphi_iX,Y) ] \nonumber
\end{align}
where $(i,j,k)$ is an even permutation of $(1,2,3)$.
\end{lem}
\begin{proof}
%-------------
As it is known,  one can define three almost hermitian structures $(J_i,G)$ on the
product manifold $M\times \mathbb{R}$ as
\[J_i\Big(X,f\frac{d}{dt}\Big)=\Big(\varphi_i X-f\xi_i,\eta_i(X)
\frac{d}{dt}\Big),\qquad G=g+dt^2,\]
where $X\in\frak{X}(M)$ and $f$ is a differentiable function on $M\times \mathbb{R}$.
These almost hermitian structures satisfy $J_1J_2=J_3=-J_2J_1$. Denoting by $\Omega_i$
the associated K\"ahler forms, one has
\begin{eqnarray*}
\Omega_i\Big(\Big(X,f\frac{d}{dt}\Big),\Big(Y,f'\frac{d}{dt}\Big)\Big)
&=&G\Big(\Big(X,f\frac{d}{dt}\Big),J_i\Big(Y,f'\frac{d}{dt}\Big)\Big)\\
&=& g(X,\varphi_iY-f'\xi_i)+f\eta_i(Y)\\
&=& \Phi_i(X,Y)-\eta_i(X)f'+f\eta_i(Y).
\end{eqnarray*}
Using the same notations $\Phi_i$ and $\eta_i$ for differential forms on
$M\times\mathbb{R}$ such that $\frac{d}{dt}\lrcorner\,\Phi_i=0$ and
$\eta_i(\frac{d}{dt})=0$, we have $\Omega_i=\Phi_i-\eta_i\wedge dt$, and thus
\begin{equation}\label{dOmega}
d\Omega_i=d\Phi_i-d\eta_i\wedge dt.
\end{equation}
Now, the Nijenhuis tensors of the tensor fields $J_i$ satisfy
\begin{equation}\label{N_hyperK}
\begin{split}
G([J_i,J_i](X,Y),Z)={}-d\Omega_j(X,Y,J_jZ)
+d\Omega_j(J_iX,J_iY,J_jZ)\\
+d\Omega_k(X,J_iY,J_jZ)+d\Omega_k(J_iX,Y,J_jZ)
\end{split}
\end{equation}
for all vector fields $X,Y,Z$ on $M\times \mathbb{R}$, where $(i,j,k)$ is an even
permutation of $(1,2,3)$ (see \cite[Lemma 3.2]{kashiwada0}).
Taking $X,Y,Z\in\frak{X}(M)$, the left-hand side in \eqref{N_hyperK} coincides
with $g(N_{\varphi_i}(X,Y),Z)$. Applying \eqref{dOmega} and being
$\frac{d}{dt}\lrcorner\, d\Phi_i=0$ and $\frac{d}{dt}\lrcorner\, d\eta_i=0$, we have
\begin{align*}
N_{\varphi_i}(X,Y,Z)=&{}-d\Phi_j(X,Y,\varphi_jZ)+\eta_j(Z)d\eta_j(X,Y)\\
&{}+d\Phi_j(\varphi_iX,\varphi_iY,\varphi_jZ)-\eta_i(X)d\eta_j(\varphi_iY,\varphi_jZ)\\
&{}-\eta_i(Y)d\eta_j(\varphi_jZ,\varphi_iX)-\eta_j(Z)d\eta_j(\varphi_iX,\varphi_iY)\\
&{}+d\Phi_k(X,\varphi_iY,\varphi_jZ)-\eta_i(Y)d\eta_k(\varphi_jZ,X)-\eta_j(Z)
d\eta_k(X,\varphi_iY)\\
&{}+d\Phi_k(\varphi_iX,Y,\varphi_jZ)-\eta_i(X)d\eta_k(Y,\varphi_jZ)-\eta_j(Z)
d\eta_k(\varphi_iX,Y)
\end{align*}
thus proving \eqref{N_almost3c}.
\end{proof}
\begin{rem}\label{rem.dPhi-N=0}
%-------------------------------
As a consequence of the above lemma, we observe:
if $(M,\varphi_i,\xi_i,\eta_i,g)$ is an almost
$3$-contact metric manifold such that $d\Phi_i(X,Y,Z)=0$ for all $i=1,2,3$ and
for all horizontal vector fields $X,Y,Z$, then $N_{\varphi_i}(X,Y,Z)=0$ for all
$X,Y,Z\in\Gamma(\H)$.
Hence, in this case conditions 1) and 3) in Definition \ref{df.can-3-contact} of
canonical structures below are satisfied.
\end{rem}
%
%
%------------------------------------------------------------------------------------------
\section{New classes of almost $3$-contact metric manifolds}%
%\section{New classes of almost $3$-contact metric manifolds and their properties}%
%\label{sec.examples}
%------------------------------------------------------------------------------------------
%
%
%-----------------------------------------------------------------------------------------
\subsection{Remarkable functions  and
canonical almost $3$-contact metric manifolds}
%\subsection{Remarkable functions for almost $3$-contact metric manifolds and
%the new class of canonical almost $3$-contact metric manifolds}
%-----------------------------------------------------------------------------------------
%
%
\begin{df}\label{df.Reebcf}
%--------------------------
We say that an almost $3$-contact metric manifold $(M,\varphi_i,\xi_i,\eta_i, g)$ admits
a \emph{Reeb commutator function} if there exists a function $\delta\in C^\infty(M)$
satisfying
\bdm
\eta_k([\xi_i,\xi_j]) \ =\ 2\delta\epsilon_{ijk} \text{ for every } \ i,j,k=1,2,3,
\edm
where $\epsilon_{ijk}$ is the totally skew-symmetric symbol.
We shall call the function $\delta$ the \emph{Reeb commutator function}.
\end{df}
Clearly, the existence of a constant Reeb commutator $\delta$
expresses that the three Reeb vector fields form a Lie algebra under the restriction
of the commutator to  $\mathcal{V}$, which is abelian in the case $\delta=0$, or
isomorphic to $\mathfrak{so}(3)$ if $\delta\ne 0$.
\begin{lem}[Existence of a Reeb commutator function]\label{lemma_xi_i}
%----------------------------------------------------------------------
%
Let $(M,\varphi_i,\xi_i,\eta_i, g)$ be an almost $3$-contact metric manifold. Then the
following conditions are equivalent:
\begin{enumerate}[\normalfont 1)]
\item $({\mathcal L}_{\xi_i}g)(\xi_j,\xi_k)=0$ for every $i,j,k=1,2,3$;
\item  $\eta_k([\xi_i,\xi_j])=2\delta\epsilon_{ijk}$ for some function $\delta\in C^\infty(M)$
and for every $i,j,k=1,2,3$;
\item $\eta_k(\nabla^g_{\xi_i}\xi_j)=\delta\epsilon_{ijk}$ for
some function $\delta\in C^\infty(M)$ and for every $i,j,k=1,2,3$.
\end{enumerate}
\end{lem}
\begin{proof}
%-------------
The equivalence of 1) and 2) is consequence of the following equations, which hold
for every $i,j,k=1,2,3$:
\begin{align*}
({\mathcal L}_{\xi_i}g)(\xi_j,\xi_k)&=-\eta_k([\xi_i,\xi_j])+\eta_j([\xi_k,\xi_i]),\\
({\mathcal L}_{\xi_i}g)(\xi_i,\xi_k)&=\eta_i([\xi_k,\xi_i]),\quad
({\mathcal L}_{\xi_i}g)(\xi_k,\xi_k)=2\,\eta_k([\xi_k,\xi_i]).
\end{align*}
Now, let us assume that 2) holds. Since
$\eta_i([\xi_j,\xi_k])=\eta_j([\xi_k,\xi_i])=2\delta\epsilon_{ijk}$, we have
\[
g(\nabla^g_{\xi_j}\xi_k,\xi_i)-g(\nabla^g_{\xi_k}\xi_j,\xi_i)
=g(\nabla^g_{\xi_k}\xi_i,\xi_j)-g(\nabla^g_{\xi_i}\xi_k,\xi_j).
\]
It follows that $g(\nabla^g_{\xi_j}\xi_i,\xi_k)=-g(\nabla^g_{\xi_i}\xi_j,\xi_k)$, and thus
\[2\delta\epsilon_{ijk}=\eta_k([\xi_i,\xi_j])=2\,\eta_k(\nabla^g_{\xi_i}\xi_j),\]
which implies 3). Conversely, 2) immediately follows from 3).
\end{proof}
\begin{cor}\label{cor.Reebcf-Killing}
%---------------------------------------
Any almost $3$-contact metric manifold $(M,\varphi_i,\xi_i,\eta_i, g)$ for which all
$\xi_i$ are Killing vector fields admits a Reeb commutator function $\delta$.
\end{cor}
A second remarkable function catches subtle properties of the Lie derivatives.
Given  an almost $3$-contact metric manifold $(M,\varphi_i,\xi_i,\eta_i,g)$, we introduce
the tensor fields $A_{ij}$, $i,j=1,2,3$, defined on the subbundle $\H$ of $TM$ by
\begin{equation}\label{Aij}
A_{ij}(X,Y)\ :=\
g(({\mathcal L}_{\xi_j}\varphi_i)X,Y)+d\eta_j(X,\varphi_i Y)+d\eta_j(\varphi_i X,Y).
\end{equation}
We shall denote by $A_i$ the tensor field $A_{ii}$. In Proposition \ref{prop_Killing},
we will prove that an expression of this type appears as the covariant derivative of
$\varphi\in\Sigma_M$ for any $\varphi$-compatible connection, thus partially
explaining its relevance. In Section \ref{subsec.3-AD-props-and-exa}, we will encounter
many manifolds for which
$A_i=0$ and the tensor fields $A_{ij}\ (i\ne j)$ are skew-symmetric in $i,j$ and
proportional to $\Phi_k$, where $k$ is the only remaining index in
$\{1,2,3\}$ different from $i$ and $j$. The following definition captures these properties.
\begin{df}\label{df.ReebKf}
%---------------------------
An  almost $3$-contact metric manifold $(M,\varphi_i,\xi_i,\eta_i, g)$ is said to admit
a \emph{Reeb Killing function} if there exists a smooth function $\beta\in C^\infty(M)$
such that
for every $X,Y\in\Gamma(\H)$ and every even permutation $(i,j,k)$ of $(1,2,3)$,
\be\label{eq.ReebKf}
A_{i}(X,Y)=0,\qquad A_{ij}(X,Y)=-A_{ji}(X,Y)=\beta\Phi_k(X,Y).
\ee
As a special case,
$M$ will be called a \emph{parallel} almost $3$-contact metric manifold if it has vanishing
Reeb Killing function, $\beta=0$ or, equivalently, $A_{ij}=0 \ \forall i,j=1,2,3$.
\end{df}
\begin{rem}
%-----------
The intrinsic meaning of the function $\beta$ is not as obvious as for the Reeb
commutator function, but it will become clearer as we proceed. Most importantly, we shall
see later that it controls
the derivatives of the structure tensors $\varphi_i,\xi_i$, and $\eta_i$ with respect
to the canonical connection (see Remark \ref{derivatives}), and this is in fact the
justification why manifolds with vanishing $\beta$ are called parallel.
\end{rem}
The following definition turns out to be tailor-made for our purposes:
\begin{df}\label{df.can-3-contact}
%---------------------------------
Let $(M,\varphi_i,\xi_i,\eta_i, g)$  be an almost $3$-contact metric manifold.
We call it a \emph{canonical almost $3$-contact metric manifold}
if the following conditions are satisfied:
\begin{enumerate}[\normalfont 1)]
\item each $N_{\varphi_i}$ is skew-symmetric on $\H$,
\item each $\xi_i$ is a Killing vector field,
\item for any $X,Y,Z\in\Gamma(\H)$ and any $i,j=1,2,3$,
\[N_{\varphi_i}(X,Y,Z)-d\Phi_i(\varphi_i X,\varphi_i Y,\varphi_i Z)=N_{\varphi_j}(X,Y,Z)-d\Phi_j(\varphi_j X,\varphi_j Y,\varphi_j Z),
\]
\item $M$ admits a Reeb Killing function $\beta\in C^\infty(M)$.
\end{enumerate}
As before, a \emph{parallel} canonical almost $3$-contact metric manifold is
one with vanishing Reeb Killing function, $\beta=0$.
\end{df}
By Corollary \ref{cor.Reebcf-Killing}, a canonical almost $3$-contact metric manifold
admits also a Reeb commutator function $\delta$.
We shall see in Theorem \ref{theo_canonical} that canonical almost $3$-contact metric
manifolds are exactly those admitting a canonical connection, thus explaining the name.
In a first step, we prove that each of the three almost contact metric structures of
a canonical almost $3$-contact metric manifold admits a characteristic connection in the
sense of Friedrich and Ivanov (Theorem \ref{theo-contact}):
\begin{theo}[Characteristic connections of canonical manifolds]%
\label{theo_canonical-implies-char}
%----------------------------------------------------------------------------
Let $(M,\varphi_i,\xi_i,\eta_i, g)$ be a canonical almost $3$-contact metric manifold.
Then the following hold:
\begin{enumerate}[\normalfont 1)]
\item
The three Nijenhuis tensors $N_{\varphi_i}$ $(i=1,2,3)$ are skew-symmetric on $TM$, and hence
each almost contact metric structure $(\varphi_i,\xi_i,\eta_i,g)$  admits a
characteristic connection $\nabla^i$.
\item
Each almost contact metric structure $(\varphi,\xi,\eta,g)$ in the associated
sphere $\Sigma_M$ admits a characteristic connection.
\end{enumerate}
\end{theo}
\begin{proof}
%----------------
1) By assumption, the three Reeb vector fields $\xi_i$ are Killing vector fields and
the Nijenhuis tensors $N_{\varphi_i}$ are skew-symmetric on $\H$.
We prove that each $N_{\varphi_i}$ is skew-symmetric on $TM$. First of
all, the definition of the Nijenhuis tensor implies  that for every $X,Y\in\Gamma(\H)$
and for every even
permutation $(i,j,k)$ of $(1,2,3)$, the following equations hold:
\begin{align}\label{tableN}
N_{\varphi_i}(X,Y)&=[\varphi_i X,\varphi_i Y]-[X,Y]-\varphi_i[\varphi_i X,Y]
-\varphi_i[X,\varphi_i Y],\nonumber\\
N_{\varphi_i}(X,\xi_i)&=-[X,\xi_i]-\varphi_i[\varphi_iX,\xi_i],\nonumber\\
N_{\varphi_i}(X,\xi_j)&=[\varphi_i X,\xi_k]-[X,\xi_j]-\varphi_i[\varphi_i X,\xi_j]
-\varphi_i[X,\xi_k],\nonumber\\
N_{\varphi_i}(X,\xi_k)&=-[\varphi_i X,\xi_j]-[X,\xi_k]-\varphi_i[\varphi_i X,\xi_k]
+\varphi_i[X,\xi_j],\\
N_{\varphi_i}(\xi_i,\xi_j)&= -[\xi_i,\xi_j]-\varphi_i[\xi_i,\xi_k],\nonumber\\
N_{\varphi_i}(\xi_i,\xi_k)&= -[\xi_i,\xi_k]+\varphi_i[\xi_i,\xi_j],\nonumber\\
N_{\varphi_i}(\xi_j,\xi_k)&=0.\nonumber
\end{align}
Let $\delta$ be the Reeb commutator function,
i.\,e.~$\eta_t([\xi_r,\xi_s])=2\delta\epsilon_{rst}$, with
$r,s,t=1,2,3$. Using the defining relations \eqref{3-sasaki},
one easily checks that
\begin{align*}
&N_{\varphi_i}(\xi_i,\xi_j,\xi_i)=N_{\varphi_i}(\xi_i,\xi_j,\xi_j)
=N_{\varphi_i}(\xi_i,\xi_k,\xi_i)=N_{\varphi_i}(\xi_i,\xi_k,\xi_k)=0,\\
&N_{\varphi_i}(\xi_i,\xi_j,\xi_k)=N_{\varphi_i}(\xi_i,\xi_k,\xi_j)=0.
\end{align*}
For $X\in\Gamma(\H)$ and $r,s=1,2,3$, one checks that
\bdm
(\mathcal{L}_X g)(\xi_r,\xi_s) \, = \,
- (\mathcal{L}_{\xi_r} g)(X,\xi_s) - (\mathcal{L}_{\xi_s} g)(X,\xi_r).
\edm
Since the Reeb vector fields are Killing, this quantity vanishes, which is
equivalent to $\eta_r([X,\xi_s])+\eta_s([X,\xi_r])=0$. Therefore,
\begin{align*}
N_{\varphi_i}(X,\xi_i,\xi_i)&=-\eta_i([X,\xi_i])=0,\\
N_{\varphi_i}(X,\xi_j,\xi_j)&=\eta_j([\varphi_i X,\xi_k])-\eta_j([X,\xi_j])
+\eta_k([\varphi_i X,\xi_j])+\eta_k([X,\xi_k])=0.
\end{align*}
Similarly one shows the following identities:
\begin{align*}
N_{\varphi_i}(X,\xi_k,\xi_k)&=0,\quad &N_{\varphi_i}(X,\xi_i,\xi_j)+ N_{\varphi_i}(X,\xi_j,\xi_i)=0,\\
N_{\varphi_i}(X,\xi_j,\xi_k)&=N_{\varphi_i}(X,\xi_k,\xi_j)=0,\quad &
N_{\varphi_i}(X,\xi_i,\xi_k)+ N_{\varphi_i}(X,\xi_k,\xi_i)=0.
\end{align*}
Since $\xi_i$ is a Killing vector field, we have
\begin{align*}
&N_{\varphi_i}(\xi_i,\xi_j,X)+N_{\varphi_i}(\xi_i,X,\xi_j)\\
&= -g([\xi_i,\xi_j],X)+g([\xi_i,\xi_k],\varphi_i X)-g([\xi_i,X],\xi_j)
+g([\xi_i,\varphi_iX],\xi_k)\\
&=({\mathcal L}_{\xi_i}g)(\xi_j,X)-({\mathcal L}_{\xi_i}g)(\xi_k,\varphi_iX)=0
\end{align*}
and analogously, $N_{\varphi_i}(\xi_i,\xi_k,X)+N_{\varphi_i}(\xi_i,X,\xi_k)=0$.
Now the existence of a Reeb Killing function $\beta$  (see eq.\,\eqref{eq.ReebKf})
yields for $X,Y\in\Gamma(\H)$
\begin{align*}
&\quad N_{\varphi_i}(X,Y,\xi_i)+N_{\varphi_i}(X,\xi_i,Y)\\
&\quad= \eta_i([\varphi_iX,\varphi_iY])-\eta_i([X,Y])+g([\xi_i,X],Y)
-g([\xi_i,\varphi_iX],\varphi_i Y)\\
&\quad =-d\eta_i(\varphi_iX,\varphi_iY)+d\eta_i(X,Y)
-g(({\mathcal L}_{\xi_i}\varphi_i)\varphi_i X,Y)\\
&\quad =-A_i(\varphi_iX,Y)=0.
\end{align*}
Furthermore, one can compute
\be\label{NXYxi}
N_{\varphi_i}(X,Y,\xi_j)
=-d\eta_j(\varphi_iX,\varphi_iY)+d\eta_j(X,Y)
-d\eta_k(\varphi_iX,Y)-d\eta_k(X,\varphi_iY),
\ee
as well as
\begin{align*}
N_{\varphi_i}(X,\xi_j,Y)\nonumber&=  g([\varphi_i X,\xi_k],Y)-g([X,\xi_j],Y)
+g([\varphi_i X,\xi_j],\varphi_iY)+g([X,\xi_k],\varphi_iY)\nonumber\\
&=  -g(({\mathcal L}_{\xi_j}\varphi_i)\varphi_i X,Y)-g(({\mathcal L}_{\xi_k}\varphi_i) X,Y).
\end{align*}
Therefore, using again the existence of a Reeb Killing function, we conclude
\begin{eqnarray*}
N_{\varphi_i}(X,Y,\xi_j)+N_{\varphi_i}(X,\xi_j,Y)
%\\&= g([\varphi_iX,\varphi_iY],\xi_j)-g([X,Y],\xi_j)+g([X,\varphi_i Y],\xi_k)+g([\varphi_i X,Y],\xi_k)\\
%&\quad+g([\varphi_i X,\xi_k],Y)-g([X,\xi_j],Y)+g([X,\xi_k],\varphi_iY)+g([\varphi_i X,\xi_j],\varphi_iY)\\
%&={}-d\eta_j(\varphi_iX,\varphi_iY)+d\eta_j(X,Y)-d\eta_k(X,\varphi_iY)-d\eta_k(\varphi_iX,Y)\\
%&\quad-g(({\mathcal L}_{\xi_j}\varphi_i)\varphi_i X,Y)-g(({\mathcal L}_{\xi_k}\varphi_i) X,Y)\\
&= & -A_{ij}(\varphi_iX,Y)-A_{ik}(X,Y)\\
& =& -\beta\Phi_k(\varphi_iX,Y)+\beta\Phi_j(X,Y)\\
&= & -\beta g(\varphi_iX,\varphi_kY)+\beta g(X,\varphi_jY)=0.
\end{eqnarray*}
Analogously one shows that $N_{\varphi_i}(X,Y,\xi_k)+N_{\varphi_i}(X,\xi_k,Y)=0$. Finally,
\bdm
N_{\varphi_i}(\xi_i,X,X) =-g([\xi_i,X],X)-g(\varphi_i[\xi_i,\varphi_iX],X)
=g(({\mathcal L}_{\xi_i}\varphi_i)\varphi_iX,X) =A_i(\varphi_iX,X)=0,
\edm
and furthermore
\begin{eqnarray*}
N_{\varphi_i}(\xi_j,X,X)&= & g([\xi_k,\varphi_i X],X)-g([\xi_j,X],X)
 -g(\varphi_i[\xi_k,X], X)-g(\varphi_i[\xi_j,\varphi_iX],X)\\
&= & g(({\mathcal L}_{\xi_k}\varphi_i)X,X)+g(({\mathcal L}_{\xi_j}\varphi_i)\varphi_iX,X)
\ =\  A_{ik}(X,X)+A_{ij}(\varphi_iX,X)\\
&=& -\beta\Phi_j(X,X)+\beta\Phi_k(\varphi_iX,X)=0.
%&=\frac12\left\{({\mathcal L}_{\xi_i}g)(X,X)-\xi_i(g(X,X))\right\}\\
%&\quad-\frac12\{({\mathcal L}_{\xi_i}g)(\varphi_i X,\varphi_i X)-\xi_i(g(\varphi_i X,\varphi_i X))\}=0,
\end{eqnarray*}
Analogously, $N_{\varphi_i}(\xi_k,X,X)=0$, completing the proof that each
$N_{\varphi_i}$ is skew-symmetric on $TM$.
Since all $\xi_i$ are Killing vector fields,
the existence of a characteristic connection $\nabla^i$ for each $\varphi_i$ now follows
from Theorem \ref{theo-contact}.

\medskip\noindent
2) Let $(\varphi,\xi,\eta,g)$ be in the associated sphere $\Sigma_M$. Its Reeb vector field
$\xi$ is obviously Killing, and thus the main point is to prove that its Nijenhuis
tensor $N_\varphi$ is skew-symmetric on $TM$. By 1),  each tensor $N_{\varphi_i}$ is
skew-symmetric on $TM$, and consequently Proposition \ref{prop.N-skew-if-Ni-skew} implies that each $N_{i,j}$, $i\ne j$, is skew-symmetric on $\mathcal H$. In view of
\eqref{ident-CM-DN-Y}, we only need to show that each $N_{i,j}$ is skew-symmetric on $TM$.
In the following we fix an even permutation $(i,j,k)$ of $(1,2,3)$ and denote by
$X,Y,Z$ horizontal vector fields. We proceed case by case as 1), hence we shall be brief.
From the definition of $N_{i,j}$, taking into account that
$(\varphi_i\varphi_j+\varphi_j\varphi_i)X=0$, we obtain after a short calculation
\begin{align}\label{tableNij}
N_{i,j}(X,Y)&=[\varphi_i X,\varphi_jY]-[\varphi_jX,\varphi_iY]-\varphi_i[\varphi_j X,Y]
-\varphi_i[X,\varphi_j Y]\nonumber\\&\quad-\varphi_j[\varphi_i X,Y]
-\varphi_j[X,\varphi_i Y],\nonumber\\
N_{i,j}(X,\xi_i)&=-[\varphi_iX,\xi_k]-\varphi_i[\varphi_jX,\xi_i]+\varphi_i[X,\xi_k]-\varphi_j[\varphi_i X,\xi_i]\nonumber\\
&=({\mathcal L}_{\xi_k}\varphi_i)X-({\mathcal L}_{\xi_i}\varphi_i)\varphi_jX-({\mathcal L}_{\xi_i}\varphi_j)\varphi_iX,\nonumber\\
N_{i,j}(X,\xi_j)&=[\varphi_j X,\xi_k]-\varphi_i[\varphi_jX,\xi_j]-\varphi_j[\varphi_i X,\xi_j]-\varphi_j[X,\xi_k]\nonumber\\
&=-({\mathcal L}_{\xi_k}\varphi_j)X-({\mathcal L}_{\xi_j}\varphi_i)\varphi_jX-({\mathcal L}_{\xi_j}\varphi_j)\varphi_iX,\\
N_{i,j}(X,\xi_k)&=[\varphi_i X,\xi_i]-[\varphi_jX,\xi_j]-\varphi_i[\varphi_j X,\xi_k]
-\varphi_i[X,\xi_i]-\varphi_j[\varphi_iX,\xi_k]+\varphi_j[X,\xi_j] \nonumber\\
&=-({\mathcal L}_{\xi_i}\varphi_i)X+({\mathcal L}_{\xi_j}\varphi_j)X-({\mathcal L}_{\xi_k}\varphi_i)\varphi_jX-({\mathcal L}_{\xi_k}\varphi_j)\varphi_iX, \nonumber\\
N_{i,j}(\xi_i,\xi_j)&= \varphi_i[\xi_k,\xi_j]-\varphi_j[\xi_i,\xi_k],\nonumber\\
N_{i,j}(\xi_i,\xi_k)&= [\xi_k,\xi_j]+\varphi_j[\xi_i,\xi_j],\nonumber\\
N_{i,j}(\xi_j,\xi_k)&= [\xi_k,\xi_i]-\varphi_i[\xi_j,\xi_i].\nonumber
\end{align}
Now, since $M$ admits a Reeb commutator function, one easily checks that
$N_{i,j}(\xi_r,\xi_s,\xi_t)=0$ for every $r,s,t=1,2,3$. Furthermore, since
$\eta_r([X,\xi_s])+\eta_s([X,\xi_r])=0$, we deduce
$N_{i,j}(X,\xi_r,\xi_s)+N_{i,j}(X,\xi_s,\xi_r)=0$. Next we compute
\begin{align*}
&N_{i,j}(\xi_i,\xi_j,X)+N_{i,j}(\xi_i,X,\xi_j)\\
&= -g([\xi_k,\xi_j],\varphi_iX)+g([\xi_i,\xi_k],\varphi_j X)-g([\xi_k,\varphi_iX],\xi_j)
+g([\xi_i,\varphi_jX],\xi_k)\\
&=({\mathcal L}_{\xi_k}g)(\varphi_iX,\xi_j)-({\mathcal L}_{\xi_i}g)(\xi_k,\varphi_jX)=0.
\end{align*}
Analogously, using the equations in \eqref{tableNij}, one shows that
\[N_{i,j}(\xi_i,\xi_k,X)+N_{i,j}(\xi_i,X,\xi_k)=0,\qquad N_{i,j}(\xi_j,\xi_k,X)+N_{i,j}(\xi_j,X,\xi_k)=0.\]
Since $M$ admits a Reeb Killing function $\beta$, we have
\begin{align*}
&\quad N_{i,j}(X,Y,\xi_i)+N_{i,j}(X,\xi_i,Y)\\
&\quad= -d\eta_i(\varphi_iX,\varphi_jY)-d\eta_i(\varphi_jX,\varphi_iY)+d\eta_k(\varphi_iX,Y)+d\eta_k(X,\varphi_iY)\\
&\quad\quad+g(({\mathcal L}_{\xi_k}\varphi_i)X,Y)
-g(({\mathcal L}_{\xi_i}\varphi_i)\varphi_jX,Y)-g(({\mathcal L}_{\xi_i}\varphi_j)\varphi_iX,Y)\\
&\quad =A_{ik}(X,Y)-A_{i}(\varphi_jX,Y)+d\eta_i(\varphi_i\varphi_jX,Y)-A_{ji}(\varphi_iX,Y)+d\eta_i(\varphi_j\varphi_iX,Y)\\
&\quad =-\beta\Phi_j(X,Y)+\beta\Phi_k(\varphi_iX,Y)=-\beta g(X,\varphi_jY)+\beta g(\varphi_iX,\varphi_kY)=0.
\end{align*}
In the same way one shows that
\[N_{i,j}(X,Y,\xi_j)+N_{i,j}(X,\xi_j,Y)=0,\qquad N_{i,j}(X,Y,\xi_k)+N_{i,j}(X,\xi_k,Y)=0.\]
Finally,
\begin{eqnarray*}
N_{i,j}(\xi_i,X,X)&=& -g(({\mathcal L}_{\xi_k}\varphi_i)X,X)
+g(({\mathcal L}_{\xi_i}\varphi_i)\varphi_jX,X)-g(({\mathcal L}_{\xi_i}\varphi_j)\varphi_iX,X)\\
&=& -A_{ik}(X,X)+A_{i}(\varphi_jX,X)+d\eta_i(\varphi_jX,\varphi_iX)+d\eta_i(\varphi_i\varphi_jX,X)\\
&& -A_{ji}(\varphi_iX,X)+d\eta_i(\varphi_iX,\varphi_jX)+d\eta_i(\varphi_j\varphi_iX,X)\\
&=&\beta\Phi_j(X,X)+\beta\Phi_k(\varphi_iX,X)=0,
\end{eqnarray*}
and analogously $N_{i,j}(\xi_j,X,X)=N_{i,j}(\xi_k,X,X)=0$, thus completing the proof.
\end{proof}

We introduce now a slight generalization of $3$-cosymplectic manifolds.
\begin{df}\label{df.3d-cosymplectic}
A \emph{$3$-$\delta$-cosymplectic manifold} is an almost
$3$-contact metric manifold satisfying
\begin{equation}\label{3-delta}
d\eta_i=-2\delta\eta_j\wedge\eta_k,\qquad d\Phi_i=0,
\end{equation}
for some $\delta\in\R$ and for every even permutation $(i,j,k)$ of $(1,2,3)$.
\end{df}
When $\delta=0$, we get the notion of $3$-cosymplectic manifolds.
We shall describe the class of $3$-$\delta$-cosymplectic manifolds
with $\delta\ne0$, and show that all  $3$-$\delta$-cosymplectic manifolds are parallel and
canonical.
\begin{prop}\label{prop.3-d-cosymplectic-is-hn}
%----------------------------------------------
Let $(M,\varphi_i,\xi_i,\eta_i,g)$ be a $3$-$\delta$-cosymplectic manifold
with $\delta\ne0$. Then the structure is hypernormal, and the Levi-Civita
connection  satisfies
\begin{equation}\label{LC-3cosymplectic}
(\nabla^g_X\varphi_i)Y=\delta\{\eta_j(X)\eta_j(Y)+\eta_k(X)\eta_k(Y)\}\xi_i-\delta\,\eta_i(Y)\{\eta_j(X)\xi_j+\eta_k(X)\xi_k\},
\end{equation}
\begin{equation}\label{LCxi-3cosymplectic}
\nabla^g_X\xi_i=\delta\{\eta_k(X)\xi_j-\eta_j(X)\xi_k\}
\end{equation}
for every $X,Y\in\frak{X}(M)$ and for every even permutation $(i,j,k)$ of $(1,2,3)$.
Furthermore, each $\xi_i$ is a Killing vector field and $M$ is locally isometric
to the Riemannian product of a hyper-K\"ahler manifold and the $3$-dimensional
sphere of constant curvature $\delta^2$.
\end{prop}
\begin{proof}
%--------------
The fact that the structure is hypernormal is a consequence of Lemma \ref{lemma_N}, which
expressed $N_{\varphi_i}$ in terms of $\eta_i$ and $\Phi_i$. More precisely, the defining
relation \eqref{3-delta} of $3$-$\delta$-cosymplectic manifolds, when plugged into
the identity  \eqref{N_almost3c}, yields after a short calculation that $N_{\varphi_i}=0$.

By \cite[Lemma 6.1]{BLAIR}, the Levi-Civita
connection of any hypernormal structure satisfies
\begin{equation}
\begin{split}\label{leci-civita}
2g((\nabla^g_X\varphi_i)Y,Z)&
=d\Phi_i(X,\varphi_iY,\varphi_iZ)-d\Phi_i(X,Y,Z)\\
&\quad+d\eta_i(\varphi_iY,X)\eta_i(Z)-d\eta_i(\varphi_iZ,X)\eta_i(Y)
\end{split}
\end{equation}
for every $X,Y,Z\in\frak{X}(M)$. Together with equation \eqref{3-delta}, this yields
\begin{align*}
2g((\nabla^g_X\varphi_i)Y,Z)&=-2\delta(\eta_j\wedge\eta_k)(\varphi_iY,X)\eta_i(Z)+2\delta(\eta_j\wedge\eta_k)(\varphi_iZ,X)\eta_i(Y)\\
&=2\delta\{\eta_k(Y)\eta_k(X)+\eta_j(X)\eta_j(Y)\}\eta_i(Z)\\
&\quad +2\delta\{-\eta_k(Z)\eta_k(X)-\eta_j(Z)\eta_j(X)\}\eta_i(Y),
\end{align*}
and hence we proved \eqref{LC-3cosymplectic}.
Applying \eqref{LC-3cosymplectic} for $Y=\xi_i$, we have
\[\nabla^g_X\xi_i=-\varphi_i^2(\nabla^g_X\xi_i)=\varphi_i((\nabla^g_X\varphi_i)\xi_i)
=-\delta\{\eta_j(X)\xi_k-\eta_k(X)\xi_j\},\]
which proves \eqref{LCxi-3cosymplectic}. It follows that
\[g(\nabla^g_X\xi_i,Y)=\delta (\eta_k\wedge\eta_j)(X,Y)\]
for every vector fields $X$, $Y$, thus showing that $\xi_i$ is Killing. We can also
deduce that
\[\nabla^g_{\xi_i}\xi_i=0,\quad \nabla^g_{\xi_i}\xi_j=-\nabla^g_{\xi_j}\xi_i=\delta\xi_k, \quad [\xi_i,\xi_j]=2\delta\xi_k.\]
Then, $\mathcal V$ is an integrable distribution with totally geodesic leaves, globally
spanned by Killing vector fields. Since  $\mathcal H$ is integrable as well, the manifold
is locally isometric to the Riemannian product of a manifold $M'$ tangent to $\mathcal H$
and a $3$-dimensional Lie group tangent to $\mathcal V$,  which is isomorphic to $SO(3)$.
Owing to \eqref{LC-3cosymplectic}, the almost $3$-contact metric structure induces on
$M'$ a hyper-K\"ahler structure. Furthermore, the leaves of $\mathcal V$ have constant
sectional curvature $\delta^2$.
\end{proof}
Together with what was known before on $3$-$\delta$-cosymplectic manifolds with $\delta=0$
(Section \ref{sec.review}), we obtain:
\begin{cor}\label{cor.3-d-sympl-is-pc}
%---------------------------------------
Any $3$-$\delta$-cosymplectic manifold $(M,\varphi_i,\xi_i,\eta_i,g)$ is a
parallel canonical almost $3$-contact metric manifold.
\end{cor}
\begin{proof}
%-------------
Conditions 1)-3) of Definition \ref{df.can-3-contact} are satisfied since the structure
is hypernormal, the fundamental $2$-forms are closed, and the Reeb vector fields are
Killing. Since $M$ is locally isometric to the Riemannian product of a horizontal
hyper-K\"ahler manifold and a vertical Lie group, we have
$({\mathcal L}_{\xi_i}\varphi_j)X=0$ for every horizontal vector field $X$ and for every
$i,j=1,2,3$. It follows that $A_{ij}(X,Y)=0$ for every $X,Y\in\mathcal H$. In particular,
this shows that the Reeb Killing functions $\beta$ vanishes, i.\,e.~it is a parallel
canonical almost $3$-contact metric manifold.
\end{proof}
For $3$-$\delta$-cosymplectic structures on Lie groups, see Example \ref{example_3delta}.
Inspired by the previous result, we sketch a slightly more general construction of parallel
canonical almost $3$-contact metric manifolds:
\begin{ex}[Examples arising on HKT manifolds]\label{ex.from-HKT}
%----------------------------------------------------------------
A hyper-K\"ahler with torsion manifold, briefly HKT-manifold, is defined as a hyperhermitian
manifold $(M,J_i,h)$ endowed with a metric connection $\nabla^c$ with skew-symmetric
torsion such that $\nabla^c J_i=0$ for all $i=1,2,3$. This is equivalent to requiring that
\begin{equation}\label{hyperK}
J_1d\Omega_1=J_2d\Omega_2=J_3 d\Omega_3,
\end{equation}
where $\Omega_i$ is the K\"ahler form of $J_i$. The unique metric connection with skew
torsion parallelizing the complex structures has torsion $T_0=-J_id\Omega_i$.
Let us consider a HKT-manifold $(M,J_i,h)$ and a $3$-dimensional Lie group $G$ with
Lie algebra $\frak g$ spanned by vector fields $\xi_1$, $\xi_2$, $\xi_3$ such that
$[\xi_i,\xi_j]=2\delta\xi_k,$
for some $\delta\in\R$ and for every even permutation $(i,j,k)$ of $(1,2,3)$. In
particular, for $\delta=0$ we have an abelian Lie group, while for $\delta\ne 0$, $G$
is isomorphic to $SO(3)$. On the product manifold $M\times G$ one can define in a
natural way an almost $3$-contact metric structure $(\varphi_i,\xi_i,\eta_i,g)$,
by
\[\varphi_i|_{TM}=J_i,\quad \varphi_i\xi_i=0,\quad \varphi_i\xi_j=\xi_k,\quad
\varphi_i\xi_k=-\xi_j,\]
\[\eta_i|_{TM}=0,\quad \eta_i(\xi_i)=1,\qquad \eta_i(\xi_j)=\eta_i(\xi_k)=0,\]
and $g$ the product metric of $h$ and the left invariant Riemannian metric on $G$ with
respect to which $\xi_1$, $\xi_2$, $\xi_3$ are an orthonormal basis of $\frak g$. We
show that this structure is canonical with vanishing Reeb Killing function.

Since each structure $J_i$ is integrable, one can easily verify that the almost
$3$-contact structure is hypernormal. Each fundamental $2$-form $\Phi_i$ satisfies
$\Phi_i(X,Y)=-\Omega_i(X,Y)$, so that \eqref{hyperK} implies
\[d\Phi_i(\varphi_iX,\varphi_iY,\varphi_iZ)=d\Phi_j(\varphi_jX,\varphi_jY,\varphi_jZ)\]
for every $i,j=1,2,3$ and $X,Y,Z\in\mathcal H$. Moreover, each $\xi_i$ is a Killing
vector field and the tensor fields $A_{ij}$ are all vanishing.
\end{ex}
By the previous examples, one could be tempted to believe that parallel canonical
almost $3$-contact metric manifolds are always locally isometric to products, and hence of
limited interest. In Example \ref{ex.S7}, it is shown that $S^7$ carries in a natural way
a parallel canonical almost $3$-contact metric structure as well.
%
%
%------------------------------------------------------------------------------------------
\subsection{A generalization of Kashiwada's theorem and
$3$-$(\alpha,\delta)$-Sasaki manifolds}
%\subsection{A generalization of Kashiwada's theorem and the new class of
%$3$-$(\alpha,\delta)$-Sasaki manifolds}
%------------------------------------------------------------------------------------------
%
%
In \cite{kashiwada} T. Kashiwada proved that a $3$-contact metric manifold is necessarily
$3$-Sasakian. We shall show that hypernormality is in fact a key property  of a
larger class of almost $3$-contact metric manifolds, the so-called
$3$-$(\alpha,\delta)$-Sasaki manifolds.
\begin{df}\label{df.3ad-Sasaki}
%-------------------------------
An almost $3$-contact metric manifold $(M,\varphi_i,\xi_i,\eta_i,g)$ will be called a
\emph{$3$-$(\alpha,\delta)$-Sasaki manifold} if it satisfies
\begin{equation}\label{differential_eta}
d\eta_i=2\alpha\Phi_i+2(\alpha-\delta)\eta_j\wedge\eta_k
\end{equation}
for every even permutation $(i,j,k)$ of $(1,2,3)$, where $\alpha\ne0$ and
$\delta$ are real constants.
A $3$-$(\alpha,\delta)$-Sasaki manifold will be called \emph{degenerate} if
$\delta=0$ and \emph{nondegenerate} otherwise---quaternionic Heisenberg groups are
examples of degenerate  $3$-$(\alpha,\delta)$-Sasaki manifolds (Example \ref{ex.qHgroup}).
When $\alpha=\delta=1$,  we have a $3$-contact metric
manifold, and hence a $3$-Sasaki manifold by  Kashiwada's theorem \cite{kashiwada}.
For $\alpha=\delta$, one easily verifies
that the manifold is  $3$-$\alpha$-Sasakian.
\end{df}
\begin{rem}
%------------
This definition captures two different aspects. First, the manifold is what one could call
a `horizontal $3$-$\alpha$-contact metric manifold' in the sense that it satisfies
the $\alpha$-contact condition \eqref{eq.alpha-contact} for horizontal vector fields,
\bdm
d\eta_i(X,Y)=2\alpha\Phi_i(X,Y)\quad\forall X,Y\in\Gamma(\H).
\edm
The second term proportional to $\eta_j\wedge\eta_k$ is reminiscent of the definition
of $3$-$\delta$-cosymplectic manifolds, see equation \eqref{3-delta};
however, it is not a generalization of this notion, since $\alpha$ is not allowed
to vanish and the fundamental $2$-forms $\Phi_i$ need not be closed.
\end{rem}
The following consequences are immediate. In particular, the second property interprets
the constant $\delta$ as the Reeb commutator  function, and thus yields a first hint why
the distinction between degenerate and nondegenerate $3$-$(\alpha,\delta)$-Sasaki manifolds
is reasonable; the $\H$-homothetic deformations to be studied in the next section will
give further justification for this distinction.
\begin{lem}%\label{lem.props-3AD-S}
%----------------------------------
Any $3$-$(\alpha,\delta)$-Sasaki manifold $(M,\varphi_i,\xi_i,\eta_i,g)$ satisfies:
\begin{enumerate}[\normalfont 1)]
\item Each $\xi_i$ is an infinitesimal automorphism of the
distribution $\H$, i.\,e.
%
%\begin{equation}\label{deta_mixed}
\bdm
d\eta_r(X,\xi_s)=0\qquad X\in\Gamma(\H),\; r,s=1,2,3;
\edm
%\end{equation}
%
\item The constant $\delta$ is the  Reeb commutator  function,
\[d\eta_r(\xi_s,\xi_t)=-2\delta\epsilon_{rst},\qquad r,s,t=1,2,3;\]
\item The differentials $d\Phi_i$ are given by
\begin{equation}\label{differential_Phi}
d\Phi_i=2(\delta-\alpha)(\eta_k\wedge\Phi_j-\eta_j\wedge\Phi_k).
\end{equation}
\end{enumerate}
\end{lem}
\begin{proof}
%-------------
%
Only the last claim requires a proof. By differentiating \eqref{differential_eta},
one obtains
\[2\alpha\, d\Phi_i+2(\alpha-\delta)(d\eta_j\wedge\eta_k-\eta_j\wedge d\eta_k)=0. \]
Applying again \eqref{differential_eta}, the result follows since  $\alpha\ne0$.
\end{proof}
The proof that any $3$-$(\alpha,\delta)$-Sasaki manifold has Killing Reeb vector fields and admits a constant Reeb
Killing function requires more work, see Corollary \ref{corollary_nabla_xi} and Corollary \ref{cor.3-AD-Sasaki-implies-canonical}.
As a first crucial result, we prove the announced generalization of Kashiwada's theorem.
\begin{theo}[Generalized Kashiwada Theorem] \label{thm.general-Kashiwada}
%---------------------------------------------------------------------
Any $3$-$(\alpha,\delta)$-Sasaki manifold %$(M,\varphi_i,\xi_i,\eta_i,g)$
is hypernormal.
\end{theo}
\begin{proof}
%--------------
We shall compute the tensor fields $N_{\varphi_i}$, distinguishing case by case between
horizontal and vertical vector fields. Again, the crucial ingredient is our
Lemma \ref{lemma_N}, which  expressed $N_{\varphi_i}$ in terms of $\eta_i$ and $\Phi_i$.
We always denote by $X,Y,Z$
horizontal vector fields and by $(i,j,k)$ an even permutation of $(1,2,3)$.

First, equation \eqref{differential_Phi} implies $d\Phi_i(X,Y,Z)=0$  and thus
$N_{\varphi_i}(X,Y,Z)=0 \ \forall i=1,2,3 $ by Remark \ref{rem.dPhi-N=0}.
Notice that, since
\[
\xi_i\lrcorner\,\Phi_i=0,\qquad \xi_j\lrcorner\,\Phi_i
=-\eta_k,\qquad \xi_k\lrcorner\,\Phi_i=\eta_j,
\]
equation \eqref{differential_Phi} implies that
\begin{equation}\label{differential_Phi_partial}
\xi_i\lrcorner\, d\Phi_i=0,\quad \xi_j\lrcorner\,
d\Phi_i=-2(\delta-\alpha)(\Phi_k+\eta_{ij}),\quad
\xi_k\lrcorner\,d\Phi_i=2(\delta-\alpha)(\Phi_j+\eta_{ki}).
\end{equation}
Therefore,  Lemma \ref{lemma_N} implies after applying \eqref{differential_eta}
and \eqref{differential_Phi_partial}:
\begin{align*}
N_{\varphi_i}(X,\xi_i,Z)&=-d\Phi_j(X,\xi_i,\varphi_j Z)
+d\Phi_k(\varphi_iX,\xi_i,\varphi_jZ)+d\eta_j(\varphi_iX,\varphi_jZ)+d\eta_k(X,\varphi_jZ)\\
&=-2(\delta-\alpha)\Phi_k(\varphi_jZ,X)-2(\delta-\alpha)\Phi_j(\varphi_jZ,\varphi_iX)\\
&\quad{}+2\alpha\Phi_j(\varphi_iX,\varphi_jZ)+2\alpha\Phi_k(X,\varphi_jZ)\\
&=2\delta\Phi_j(\varphi_iX,\varphi_jZ)+2\delta\Phi_k(X,\varphi_jZ)
=-2\delta g(\varphi_iX,Z)-2\delta g(X,\varphi_i Z)=0,\\
N_{\varphi_i}(X,\xi_j,Z)&=d\Phi_j(\varphi_iX,\xi_k,\varphi_j Z)
+d\Phi_k(\varphi_iX,\xi_j,\varphi_jZ)\\
&=-2(\delta-\alpha)\Phi_i(\varphi_jZ,\varphi_iX)
+2(\delta-\alpha)\Phi_i(\varphi_jZ,\varphi_iX)=0,\\
N_{\varphi_i}(X,\xi_k,Z)&=-d\Phi_j(X,\xi_k,\varphi_j Z)-d\Phi_k(X,\xi_j,\varphi_jZ)\\
&=2(\delta-\alpha)\Phi_i(\varphi_jZ,X)-2(\delta-\alpha)\Phi_i(\varphi_jZ,X)=0.
\end{align*}
From the definition of $N_{\varphi_i}$ (see equation \eqref{tableN}), we have
\begin{align*}
N_{\varphi_i}(X,Y,\xi_i)&=-d\eta_i(\varphi_iX,\varphi_iY)+d\eta_i(X,Y),\\
N_{\varphi_i}(X,Y,\xi_j)&=-d\eta_j(\varphi_iX,\varphi_iY)+d\eta_j(X,Y)
-d\eta_k(\varphi_iX,Y)-d\eta_k(X,\varphi_iY),\\
N_{\varphi_i}(X,Y,\xi_k)&=-d\eta_k(\varphi_iX,\varphi_iY)
+d\eta_k(X,Y)+d\eta_j(\varphi_iX,Y)+d\eta_j(X,\varphi_iY).
\end{align*}
Using the fact that the structure is horizontal $3$-$\alpha$-contact, the above
terms are all vanishing. On the other hand, one can easily verify that this is
coherent with Lemma \ref{lemma_N}.
Notice that \eqref{differential_Phi_partial} implies  $d\Phi_r(X,\xi_s,\xi_t)=0$
for every $r,s,t=1,2,3$ and $X\in\Gamma(\H)$. Taking also into account that
$d\eta_r(X,\xi_s)=0$, we deduce from \eqref{N_almost3c} that
\[
N_{\varphi_r}(X,\xi_s,\xi_t)=N_{\varphi_r}(\xi_s,\xi_t,X)=0.
\]
Finally, \eqref{N_almost3c} implies
together with $d\eta_r(\xi_s,\xi_t)=-2\delta\epsilon_{rst}$ that
\[
N_{\varphi_i}(\xi_i,\xi_j,\xi_k)=N_{\varphi_i}(\xi_i,\xi_k,\xi_j)=N_{\varphi_i}(\xi_j,\xi_k,\xi_i)=0,
\]
completing the proof that $M$ is hypernormal.
\end{proof}
%
%
%-----------------------------------------------------------------------------------------
\subsection{Properties and examples of $3$-$(\alpha,\delta)$-Sasaki manifolds}%
\label{subsec.3-AD-props-and-exa}
%-----------------------------------------------------------------------------------------
%
%
We shall describe the behaviour of $3$-$(\alpha,\delta)$-Sasaki structures under a
special type of deformations, inspired by the classical $\mathcal{D}$-homothetic
deformations of almost contact metric structures.
Given an almost $3$-contact metric structure  $(\varphi_i,\xi_i,\eta_i,g)$, one can
consider the deformed structure
\begin{equation}\label{deformation}
\eta_i'=c\eta_i,\qquad \xi_i'=\frac{1}{c}\xi_i,\qquad \varphi_i'=\varphi_i,\qquad g'=ag+b\sum_{i=1}^3\eta_i\otimes\eta_i,
\end{equation}
where $a,b,c$ are real numbers such that $a>0$, $a+b>0$, $c\ne0$. One can show
that $(\varphi_i',\xi_i',\eta_i',g')$ is an almost $3$-contact metric structure if and
only if $c^2=a+b$. Indeed, each  $(\varphi_i',\xi_i',\eta_i')$ is an almost contact
structure and equations \eqref{3-sasaki} are readily verified. As for the Riemannian metric
$g'$, if $(i,j,k)$ is an even permutation of $(1,2,3)$, we have
\begin{align*}
g'(\varphi_i'X,\varphi_i'Y)&= ag(\varphi_iX,\varphi_iY)+b\{\eta_j(\varphi_iX)\eta_j(\varphi_iY)+\eta_k(\varphi_iX)\eta_k(\varphi_iY)\}\\
&=ag(X,Y)-a\eta_i(X)\eta_i(Y)+b\{\eta_j(X)\eta_j(Y)+\eta_k(X)\eta_k(Y)\}\\
&=g'(X,Y)-(a+b)\eta_i(X)\eta_i(Y)\\
&=g'(X,Y)-\frac{a+b}{c^2}\eta_i'(X)\eta_i'(Y).
\end{align*}
Therefore, $g'$ is compatible with the structure $(\varphi_i',\xi_i',\eta_i')$ if and
only if $c^2=a+b$. In particular, for $c=a$, we can choose $b=a(a-1)$.
\begin{df}
%----------------
The deformation \eqref{deformation} of an almost $3$-contact metric manifold with
real parameters $a,b,c$ satisfying $a>0$, $a+b>0$, $c\ne0$, $c^2=a+b$ will be called a
\emph{$\H$-homothetic deformation} of the original manifold, and an almost $3$-contact
metric manifold which can be obtained through a $\H$-homothetic deformation will be called
\emph{$\H$-homothetic} to the original manifold.
\end{df}
\begin{prop}\label{deformation-coefficients}
%--------------------------------------------
The $\H$-homothetic deformation of a $3$-$(\alpha,\delta)$-Sasaki manifold is again
a $3$-$(\alpha',\delta')$-Sasaki manifold with
\[
\alpha'=\alpha \frac{c}{a},\qquad \delta'=\frac{\delta}{c}.
\]
In particular, $\alpha'\delta'$ has the same sign as $\alpha\delta$, and
 the $\H$-homothetic deformation is degenerate if and only if the undeformed
$3$-$(\alpha,\delta)$-Sasaki manifold is degenerate.
\end{prop}
\begin{proof}
%--------------
The fundamental $2$-form of the deformed structure is given by
\begin{align*}
\Phi_i'(X,Y)&=g'(X,\varphi_iY)=ag(X,\varphi_iY)+b\{\eta_j(X)\eta_j(\varphi_iY)+\eta_k(X)\eta_k(\varphi_iY)\}\\
&=a\Phi_i(X,Y)+b\{-\eta_j(X)\eta_k(Y)+\eta_k(X)\eta_j(Y)\},
\end{align*}
where $(i,j,k)$ is an even permutation of $(1,2,3)$, and thus
$\Phi_i'\, =\, a\Phi_i-b\,\eta_{jk}$.
Then, if $(\varphi_i,\xi_i,\eta_i,g)$ is $3$-$(\alpha,\delta)$-Sasaki, we have
\begin{align*}
d\eta_i'&=2\alpha c\,\Phi_i+2(\alpha-\delta)c\,\eta_{jk}
= 2\alpha\frac{c}{a}(\Phi_i'+b\,\eta_{jk})+2(\alpha-\delta)c\,\eta_{jk}\\
&= 2\alpha\frac{c}{a}\Phi_i'+2\Big(\alpha c\frac{ b+ a}{a}-\delta c\Big)\eta_{jk}
= 2\alpha\frac{c}{a}\Phi_i'+2\Big(\alpha \frac{c}{a}-\frac{\delta}{c} \Big)\eta_j'\wedge\eta_k',
\end{align*}
where we applied $c^2=a+b$.
\end{proof}
\begin{ex}[$\H$-Deformations of $3$-Sasaki manifolds I]
%------------------------------------------------------
As a consequence of the above proposition, if $(\varphi_i,\xi_i,\eta_i,g)$ is a
$3$-Sasakian structure ($\alpha=\delta=1$), the $\H$-deformed
structure is a $3$-$(\alpha',\delta')$-Sasaki structure with $\alpha'\ne\delta'$ whenever
$b\ne0$. Indeed
$$\alpha'-\delta'=\frac{c}{a}-\frac{1}{c}=\frac{c^2-a}{ac}=\frac{b}{ac}\ne0.$$
Another interesting special case of $\H$-Deformations of $3$-Sasaki manifolds with
interesting curvature properties will be discussed in Example \ref{ex.3-S-homo-II}.
\end{ex}

Conversely, let us describe those $3$-$(\alpha,\delta)$-Sasaki manifolds which admit a
$3$-$\tilde\alpha$-Sasakian $\H$-homothetic deformation (the tilde only serves to
distinguish which parameters correspond to which manifold):
\begin{lem}
%-------------
A $3$-$(\alpha,\delta)$-Sasaki manifold is $\H$-homothetic to a $3$-$\tilde\alpha$-Sasaki
manifold if and only if it is nondegenerate with $\alpha\delta>0$.
\end{lem}
\begin{proof}
%-------------
Since $3$-$\tilde\alpha$-Sasaki manifolds are nondegenerate, this condition is obvious by
Proposition \ref{deformation-coefficients}, so we can assume $\delta\neq 0$. In order
to be $3$-$\alpha$-Sasakian, we must find admissible parameters $a,b,c$ such that
$\alpha'=\delta'$, which is equivalent to $\frac{\alpha}{\delta}=\frac{a}{c^2}$.
Since $a>0$, the necessary condition $\alpha\delta>0$ follows. Furthermore, we can assume that
$\alpha, \delta >0$, since the $\H$-homothetic deformation with parameters
$a=1, c=-1, b=0$ changes the signs of $\alpha$ and $\delta$.  For $\alpha, \delta >0$,
one checks that
\begin{equation}\label{hyperbola}
a=1,\  b=\frac{\delta}{\alpha}-1,\ c=\sqrt{\frac{\delta}{\alpha}}
\end{equation}
is the desired deformation.
\end{proof}
\begin{rem}
%--------------
Observe that under $\H$-homothetic deformations given by \eqref{hyperbola}, $\alpha\delta=$ const., so all $3$-$(\alpha,\delta)$-Sasaki
manifolds which are $\H$-homothetic  under the given deformation  to a fixed
$3$-$\alpha$-Sasaki manifold lie on a
`hyperbola' in the $(\alpha,\delta)$-plane. For $\alpha=\delta$, the $\H$-homothetic
deformation coincides with the identity.
\end{rem}
Thus, the questions arises whether $3$-$(\alpha,\delta)$-Sasaki manifolds with $\alpha\delta<0$
exist at all. In order to answer this question positively, we need
the notion of \emph{negative $3$-Sasakian manifolds}. They are defined as normal almost
$3$-contact manifolds $(M,\varphi_i,\xi_i,\eta_i)$ endowed with a compatible semi-Riemannian
metric $\tilde g$, i. e. $\tilde g(\varphi_i X,\varphi_i Y)=\tilde g(X,Y)-\eta_i(X)\eta_i(Y)$
such that $\tilde g$ has signature $(3,4n)$, where $4n+3$ is the dimension of $M$, and
$d\eta_i(X,Y)=2\tilde g(X,\varphi_iY)$. It is known that quaternionic K\"ahler (not
hyperK\"ahler) manifolds with negative scalar curvature admit a canonically associated
principal $\mathrm{SO}(3)$-bundle $P(M)$ which is endowed with a negative $3$-Sasakian
structure \cite{konishi,tanno96}.

If $(M,\varphi_i,\xi_i,\eta_i,\tilde g)$ is a negative $3$-Sasakian manifold, as in
\cite{tanno96} we consider the Riemannian metric
\[g=-\tilde g+2\sum_{i=1}^3\eta_i\otimes\eta_i\]
which is compatible with the structure $(\varphi_i,\xi_i,\eta_i)$. A simple computation shows
that
\[d\eta_i(X,Y)=-2g(X,\varphi_iY)-4(\eta_j\wedge\eta_k)(X,Y)\]
for every vector fields $X,Y$. Therefore $(M,\varphi_i,\xi_i,\eta_i, g)$ is a
$3$-$(\alpha,\delta)$-Sasaki manifold with $\alpha=-1$ and $\delta=1$. Applying the
$\H$-homothetic deformation \eqref{deformation}, one obtains a $3$-$(\alpha',\delta')$-Sasaki
structure with $\alpha'\delta'<0$, and  $\alpha'\ne-\delta'$ whenever
$b\ne0$. Indeed
$$
\alpha'+\delta'=-\frac{c}{a}+\frac{1}{c}=\frac{a-c^2}{ac}=-\frac{b}{ac}\ne0.
$$
Conversely, we have the following
\begin{lem}
%-------------
A $3$-$(\alpha,\delta)$-Sasaki manifold is $\H$-homothetic to a
$3$-$(\tilde\alpha,\tilde\delta)$-Sasaki
manifold with $\tilde\alpha=-\tilde\delta<0$ if and only if it is nondegenerate
with $\alpha\delta<0$.
\end{lem}
\begin{proof}
%-------------
By Proposition \ref{deformation-coefficients} an $\mathcal H$-homothetic deformation
of a $3$-$(\tilde\alpha,\tilde\delta)$-Sasaki
manifold with $\tilde\alpha=-\tilde\delta$, is  $3$-$(\alpha,\delta)$-Sasaki  with
$\alpha\delta<0$. Conversely, given a $3$-$(\alpha,\delta)$-Sasaki manifold with
$\alpha\delta<0$, first we can assume that
$\alpha<0$ and $\delta >0$ since the $\H$-homothetic deformation with parameters
$a=1, c=-1, b=0$ changes the signs of $\alpha$ and $\delta$. Then, we need admissible
parameters $a,b,c$ such that
$\alpha'=-\delta'$, which is equivalent to $\frac{\alpha}{\delta}=-\frac{a}{c^2}$.
The desired deformation is
\begin{equation}\label{hyperbola2}
a=1,\  b=-\frac{\delta}{\alpha}-1,\ c=\sqrt{-\frac{\delta}{\alpha}}.
\end{equation}
\end{proof}
\begin{rem}
%--------------
Under $\H$-homothetic deformations given by \eqref{hyperbola2}, $\alpha\delta=$ const. Then,
all $3$-$(\alpha,\delta)$-Sasaki
manifolds which are $\H$-homothetic  under the given deformation  to a fixed
$3$-$(\tilde\alpha,\tilde\delta)$-Sasaki manifold, with $\tilde\alpha=-\tilde\delta$, lie on a
`hyperbola' in the $(\alpha,\delta)$-plane.
\end{rem}
We will now further investigate the geometry of $3$-$(\alpha,\delta)$-Sasaki manifolds.
First we prove some preliminary formulas.
\begin{lem}
%-----------
Let $(M,\varphi_i,\xi_i,\eta_i,g)$ be an almost $3$-contact metric manifold. Then
the  associated fundamental $2$-forms satisfy for all $X,Y\in\frak{X}(M)$ and any
cyclic permutation $(i,j,k)$  of $(1,2,3)$ the following relations:
\begin{align}
\Phi_j(\varphi_iX,\varphi_iY)&=-\Phi_j(X,Y)+(\eta_i\wedge\eta_k)(X,Y),\label{2-form1}\\
\Phi_k(\varphi_iX,\varphi_iY)&=-\Phi_k(X,Y)-(\eta_i\wedge\eta_j)(X,Y),\label{2-form2}\\
\Phi_j(\varphi_iX,Y)&=-\Phi_k(X,Y)-\eta_i(X)\eta_j(Y),\label{2-form3}\\
\Phi_k(\varphi_iX,Y)&=\Phi_j(X,Y)-\eta_i(X)\eta_k(Y).\label{2-form4}
\end{align}
\end{lem}
\begin{proof}
%-----------------
Applying \eqref{3-sasaki}, we have
\begin{align*}
\Phi_j(\varphi_iX,\varphi_iY)&=g(\varphi_iX,\varphi_j\varphi_iY)=g(\varphi_iX,-\varphi_i\varphi_jY+\eta_i(Y)\xi_j+\eta_j(Y)\xi_i)\\
&=-g(X,\varphi_jY)+\eta_i(X)\eta_i(\varphi_jY)-\eta_i(Y)g(X,\varphi_i\xi_j)\\
&=-\Phi_j(X,Y)+\eta_i(X)\eta_k(Y)-\eta_i(Y)\eta_k(X)
\end{align*}
which proves \eqref{2-form1}. As regards \eqref{2-form3}, we have
\[\Phi_j(\varphi_iX,Y)=-g(X,\varphi_i\varphi_jY)=-g(X,\varphi_kY+\eta_j(Y)\xi_i)=-\Phi_k(X,Y)-\eta_i(X)\eta_j(Y).\]
Analogously, one shows \eqref{2-form2} and \eqref{2-form4}.
\end{proof}
The following two propositions contain the necessary preparations for showing that
$3$-$(\alpha,\delta)$-Sasaki manifolds are canonical, as they yield remarkable identities
for $\nabla^g\varphi_i$ and $\nabla^g\xi_i$.
\begin{prop}
%-------------
Let $(M,\varphi_i,\xi_i,\eta_i,g)$ be a $3$-$(\alpha,\delta)$-Sasaki manifold. Then
the Levi-Civita connection satisfies for all $X,Y\in\frak{X}(M)$ and any cyclic
permutation $(i,j,k)$  of $(1,2,3)$:
\begin{equation}
\begin{split}\label{levi-civita-Sasaki}
(\nabla^g_X\varphi_i)Y&=\alpha\, \left[g(X,Y)\xi_i-\eta_i(Y)X\right]
-2(\alpha-\delta)\,\left[\eta_k(X)\varphi_jY-\eta_j(X)\varphi_kY\right] \\
&\quad+(\alpha-\delta)\,\left[\eta_j(X)\eta_j(Y)+\eta_k(X)\eta_k(Y)\right]\xi_i\\
&\quad-(\alpha-\delta)\,\eta_i(Y)\,\left[\eta_j(X)\xi_j+\eta_k(X)\xi_k\right].
\end{split}
\end{equation}
\end{prop}
\begin{proof}
%-------------
Since the structure is hypernormal,  the Levi-Civita
connection satisfies by \cite[Lemma 6.1]{BLAIR} the identity \eqref{leci-civita} used
before. Let us  evaluate its terms. To start with,
the defining relation \eqref{differential_Phi} of a
$3$-$(\alpha,\delta)$-Sasaki manifold and the preceding equations
\eqref{2-form1}-\eqref{2-form4} yield
\begin{align*}
\lefteqn{d\Phi_i(X,\varphi_iY,\varphi_iZ) = }\\
&=  2(\delta-\alpha) [ \eta_k(X)\Phi_j(\varphi_iY,\varphi_iZ)
-\eta_j(X)\Phi_k(\varphi_iY,\varphi_iZ)
+  \eta_k(\varphi_iY)\Phi_j(\varphi_iZ,X)\\
&\quad- \eta_j(\varphi_iY)\Phi_k(\varphi_iZ,X)
+  \eta_k(\varphi_iZ)\Phi_j(X,\varphi_iY)-\eta_j(\varphi_iZ)\Phi_k(X,\varphi_iY) ]\\
&= 2(\delta-\alpha)[-\eta_k(X)\Phi_j(Y,Z)-\eta_k(X)(\eta_k\wedge\eta_i)(Y,Z)\\
&\quad+\eta_j(X)\Phi_k(Y,Z)+\eta_j(X)(\eta_i\wedge\eta_j)(Y,Z)\\
&\quad-\eta_j(Y)\Phi_k(Z,X)-\eta_j(Y)\eta_i(Z)\eta_j(X)+\eta_k(Y)\Phi_j(Z,X)
-\eta_k(Y)\eta_i(Z)\eta_k(X)\\
 &\quad- \eta_j(Z)\Phi_k(X,Y)+\eta_j(Z)\eta_i(Y)\eta_j(X)+\eta_k(Z)\Phi_j(X,Y)
+\eta_k(Z)\eta_i(Y)\eta_k(X)]\\
&= d\Phi_i(X,Y,Z)+4(\delta-\alpha)[-\eta_k(X)\Phi_j(Y,Z)+\eta_j(X)\Phi_k(Y,Z)]\\
&\quad+ 4(\delta-\alpha)\eta_j(X)[\eta_i(Y)\eta_j(Z)-\eta_j(Y)\eta_i(Z)]\\
 &\quad+ 4(\delta-\alpha)\eta_k(X)[\eta_i(Y)\eta_k(Z)-\eta_k(Y)\eta_i(Z)].
\end{align*}
On the other hand, again using the defining relation \eqref{differential_eta}, we obtain
\begin{align*}
\lefteqn{ d\eta_i(\varphi_iY,X)\eta_i(Z)-d\eta_i(\varphi_iZ,X)\eta_i(Y) = }\\
&=  \eta_i(Z)[2\alpha\Phi_i(\varphi_iY,X)
+2(\alpha-\delta)(\eta_j\wedge\eta_k)(\varphi_iY,X)]\\
&\quad-   \eta_i(Y)[2\alpha\Phi_i(\varphi_iZ,X)
+2(\alpha-\delta)(\eta_j\wedge\eta_k)(\varphi_iZ,X)]\\
&= 2\alpha[g(X,Y)\eta_i(Z)-g(X,Z)\eta_i(Y)]\\
&\quad +  2(\alpha-\delta)\eta_i(Z)[-\eta_k(Y)\eta_k(X)-\eta_j(X)\eta_j(Y)]\\
& \quad- 2(\alpha-\delta)\eta_i(Y)[-\eta_k(Z)\eta_k(X)-\eta_j(X)\eta_j(Z)]
\end{align*}
Inserting the above computations in \eqref{leci-civita}, we conclude that
\begin{align*}
g((\nabla^g_X\varphi_i)Y,Z)&=\alpha[g(X,Y)\eta_i(Z)-g(X,Z)\eta_i(Y)]\\
&\quad+2(\delta-\alpha)[\eta_k(X)g(\varphi_jY,Z)-\eta_j(X)g(\varphi_kY,Z)]\\
&\quad-(\alpha-\delta)\eta_i(Z)[-\eta_k(Y)\eta_k(X)-\eta_j(X)\eta_j(Y)]\\
&\quad+(\alpha-\delta)\eta_i(Y)[-\eta_k(Z)\eta_k(X)-\eta_j(X)\eta_j(Z)]
\end{align*}
which implies \eqref{levi-civita-Sasaki}.
\end{proof}
\begin{cor}\label{corollary_nabla_xi}
%--------------------------------------
Let $(M,\varphi_i,\xi_i,\eta_i,g)$ be a $3$-$(\alpha,\delta)$-Sasaki manifold.
Then, for every $X\in\frak{X}(M)$ and for every  even permutation $(i,j,k)$
of  $(1,2,3)$,
\begin{equation}\label{levi-civita-xi}
\nabla^g_X\xi_i=-\alpha\varphi_iX-(\alpha-\delta)\,[\eta_k(X)\xi_j-\eta_j(X)\xi_k],
\end{equation}
\begin{equation}\label{levi-civita-xi-xi}
\nabla^g_{\xi_i}\xi_i=0,\quad \nabla^g_{\xi_i}\xi_j=-\nabla^g_{\xi_j}\xi_i=\delta\xi_k.
\end{equation}
Consequently, each $\xi_i$ is a Killing vector field. Furthermore, the
distribution $\mathcal V$ is integrable with totally geodesic leaves.
\end{cor}
\begin{proof}
%--------------
Applying \eqref{levi-civita-Sasaki} for $Y=\xi_i$, we get
\[
(\nabla^g_X\varphi_i)\xi_i=\alpha[\eta_i(X)\xi_i-X]+(\alpha-\delta)[\eta_j(X)\xi_j
+\eta_k(X)\xi_k].\]
Applying $\varphi_i$ on both hand-sides, since
$(\nabla^g_X\varphi_i)\xi_i=-\varphi_i(\nabla^g_X\xi_i)$, we obtain \eqref{levi-civita-xi}.
Equations \eqref{levi-civita-xi-xi} are immediate consequences of \eqref{levi-civita-xi}.
It follows that the distribution $\mathcal V$ is integrable with totally geodesic leaves.
In particular $[\xi_i,\xi_j]=2\delta\xi_k$. Finally, from \eqref{levi-civita-xi} we have
\[g(\nabla^g_X\xi_i,Y)=\alpha\Phi_i(X,Y)+(\alpha-\delta)(\eta_j\wedge\eta_k)(X,Y)\]
for every $X,Y\in\frak{X}(M)$. Since $\nabla^g\xi_i$ is skew-symmetric, $\xi_i$ is
Killing.
\end{proof}
\begin{cor}%\label{corollary_lie_derivative}
%--------------------------------------------
Let $(M,\varphi_i,\xi_i,\eta_i,g)$ be a $3$-$(\alpha,\delta)$-Sasaki manifold.
Then for every  even permutation $(i,j,k)$ of  $(1,2,3)$ we have
\begin{equation}\label{lie_derivative-sasaki}
{\mathcal L}_{\xi_i}\varphi_i=0,\qquad {\mathcal L}_{\xi_i}\varphi_j
=-{\mathcal L}_{\xi_j}\varphi_i=2\delta\varphi_k.
\end{equation}
\end{cor}
\begin{proof}
%-------------
For the first Lie derivative, notice that by \eqref{levi-civita-Sasaki} we have
$\nabla^g_{\xi_i}\varphi_i=0$. Then, applying also \eqref{levi-civita-xi}, for every
vector field $X$ we have
\begin{align*}
({\mathcal L}_{\xi_i}\varphi_i)X&=[\xi_i,\varphi_iX]-\varphi_i[\xi_i,X]\\
&=\nabla^g_{\xi_i}(\varphi_iX)-\nabla^g_{\varphi_iX}\xi_i-\varphi_i(\nabla^g_{\xi_i}X)
+\varphi_i(\nabla^g_X\xi_i)\\
&=(\nabla^g_{\xi_i}\varphi_i)X-\nabla^g_{\varphi_iX}\xi_i+\varphi_i(\nabla^g_X\xi_i)\\
&=\alpha\varphi_i^2X+(\alpha-\delta)[\eta_k(\varphi_iX)\xi_j-\eta_j(\varphi_iX)\xi_k]\\
&\quad-\alpha\varphi_i^2X-(\alpha-\delta)[\eta_k(X)\varphi_i\xi_j
-\eta_j(X)\varphi_i\xi_k]=0.
\end{align*}
Now, using \eqref{levi-civita-Sasaki} for the covariant derivative $\nabla^g\varphi_j$,
for every vector field $Y$, we have
\begin{align*}
(\nabla^g_{\xi_i}\varphi_j)Y&=\alpha[\eta_i(Y)\xi_j-\eta_j(Y)\xi_i]
-2(\alpha-\delta)\varphi_kY+(\alpha-\delta)[\eta_i(Y)\xi_j-\eta_j(Y)\xi_i]\\
&=-2(\alpha-\delta)\varphi_kY+(2\alpha-\delta)[\eta_i(Y)\xi_j-\eta_j(Y)\xi_i].
\end{align*}
Therefore, applying also \eqref{levi-civita-xi}, we get
\begin{align*}
({\mathcal L}_{\xi_i}\varphi_j)X&=(\nabla^g_{\xi_i}\varphi_j)X
-\nabla^g_{\varphi_jX}\xi_i+\varphi_j(\nabla^g_X\xi_i)\\
&=-2(\alpha-\delta)\varphi_kX+(2\alpha-\delta) [ \eta_i(X)\xi_j-\eta_j(X)\xi_i ]
+\alpha\varphi_i\varphi_jX\\
&\quad+(\alpha-\delta)\eta_k(\varphi_jX)\xi_j
-\alpha\varphi_j\varphi_iX+(\alpha-\delta)\eta_j(X)\varphi_j\xi_k\\
&=2\delta\varphi_k X+\alpha[\eta_j(X)\xi_i-\eta_i(X)\xi_j]
+(2\alpha-\delta+\delta-\alpha)[\eta_i(X)\xi_j-\eta_j(X)\xi_i]\\
&=2\delta\varphi_k X.
\end{align*}
Analogously, one shows that ${\mathcal L}_{\xi_j}\varphi_i=-2\delta\varphi_k$.
\end{proof}
\begin{cor}\label{cor.3-AD-Sasaki-implies-canonical}
%---------------------------------------------------
Every $3$-$(\alpha,\delta)$-Sasaki manifold is a canonical almost $3$-contact metric manifold.
In particular, it admits a constant Reeb Killing function $\beta = 2 (\delta-2\alpha)$.
\end{cor}
\begin{proof}
%--------------
We proved that  each $\xi_i$ is a Killing vector field in Corollary \ref{corollary_nabla_xi}
and  that the structure is hypernormal in Theorem \ref{thm.general-Kashiwada}.
Furthermore, by \eqref{differential_Phi}, $d\Phi_i(X,Y,Z)=0$ for every $X,Y,Z\in\Gamma(\H)$.
Therefore, conditions 1)-3) of Definition \ref{df.can-3-contact} are satisfied. As regards
condition 4), we show that $M$ admits a (constant) Reeb Killing function
$\beta=2(\delta-2\alpha)$. Indeed, for every $X,Y\in\Gamma(\H)$
\begin{align*}
A_i(X,Y)&=g(({\mathcal L}_{\xi_i}\varphi_i)X,Y)+d\eta_i(X,\varphi_iY)+d\eta_i(\varphi_iX,Y)\\
&=2\alpha\Phi_i(X,\varphi_iY)+2\alpha\Phi_i(\varphi_iX,Y)=0,
\end{align*}
where we have applying \eqref{lie_derivative-sasaki} and \eqref{differential_eta}. Moreover,
for every even permutation $(i,j,k)$ of $(1,2,3)$ and for every $X,Y\in\Gamma (\H)$ we have
\begin{align*}
A_{ij}(X,Y)&=g(({\mathcal L}_{\xi_j}\varphi_i)X,Y)+d\eta_j(X,\varphi_iY)+d\eta_j(\varphi_iX,Y)\\
&=-2\delta g(\varphi_kX,Y)+2\alpha\Phi_j(X,\varphi_iY)+2\alpha\Phi_j(\varphi_i X,Y)\\
&= 2\delta g(X,\varphi_kY)+2\alpha g(X,\varphi_j\varphi_iY)-2\alpha g(X,\varphi_i\varphi_jY)\\
&=2(\delta-2\alpha)\Phi_k(X,Y).
\end{align*}
Analogously, one checks $A_{ji}(X,Y)=-2(\delta-2\alpha)\Phi_k(X,Y)$. Hence, $M$ is canonical.
\end{proof}
\begin{df}
%-----------
Accordingly, we will call a $3$-$(\alpha,\delta)$-Sasaki manifold with $\delta=2\alpha$
a \emph{parallel} $3$-$(\alpha,\delta)$-Sasaki manifold, compare Definition \ref{df.ReebKf}.
\end{df}
We close the section with some examples of degenerate  $3$-$(\alpha,\delta)$-Sasaki manifolds
and the observation that any $3$-$(\alpha,\delta)$-Sasaki manifold admits an underlying
quaternionic contact Einstein structure; this will allow us to compute the Riemannian Ricci
curvature of a $3$-$(\alpha,\delta)$-Sasaki manifold.
\begin{ex}[Quaternionic Heisenberg groups]\label{ex.qHgroup}
%--------------------------------------------------------------
The quaternionic Heisenberg group of dimension $4p+3$ is the connected,
simply connected
Lie group $N_p$ with Lie algebra ${\frak n}_p$ spanned by vector fields
$\xi_1,\xi_2,\xi_3, \tau_r,\tau_{p+r},\tau_{2p+r},\tau_{3p+r}$, $r=1,\ldots,p$, whose
non-vanishing commutators are:
\begin{align*}
&[\tau_r,\tau_{p+r}]=\lambda\xi_1& \quad&[\tau_r,\tau_{2p+r}]
=\lambda\xi_2&\quad &[\tau_r,\tau_{3p+r}]=\lambda\xi_3\\
&[\tau_{2p+r},\tau_{3p+r}]=\lambda\xi_1&\quad &[\tau_{3p+r},\tau_{p+r}]
=\lambda\xi_2&\quad &[\tau_{p+r},\tau_{2p+r}]=\lambda\xi_3,
\end{align*}
where $\lambda$ is a positive real number. As described in \cite{Ag-F-S}, the Lie group
$N_p$ admits a left invariant almost $3$-contact metric structure
$(\varphi_i,\xi_i,\eta_i,g_\lambda)$, $i=1,2,3$, where $g_\lambda$ is the Riemannian metric
with respect to which the above basis is orthonormal, $\eta_i$ is the dual $1$-form of
$\xi_i$, and $\varphi_i$ is given by
\[
\varphi_i=\eta_j\otimes\xi_k-\eta_k\otimes\xi_j+\sum_{r=1}^p[\theta_r\otimes\tau_{ip+r}
-\theta_{ip+r}\otimes\tau_{r}
+\theta_{jp+r}\otimes\tau_{kp+r}-\theta_{kp+r}\otimes\tau_{jp+r}]
\]
where $\theta_l$, $l=1,\dots,4p$, is the dual $1$-form of $\tau_l$, and $(i,j,k)$ is an
even permutation of $(1,2,3)$. The differential of each $1$-form $\eta_i$ is given by
\begin{equation*}\label{deta_H}
d\eta_i=-\lambda\sum_{r=1}^p[\theta_r\wedge\theta_{ip+r}+\theta_{jp+r}\wedge\theta_{kp+r}],
\end{equation*}
and the fundamental $2$-forms of the structure are
\begin{equation*}\label{Phi_H}
\Phi_i={}-\eta_j\wedge\eta_k-\sum_{r=1}^p[\theta_r\wedge\theta_{ip+r}
+\theta_{jp+r}\wedge\theta_{kp+r}].
\end{equation*}
Therefore,
\[d\eta_i=\lambda(\Phi_i+\eta_j\wedge\eta_k).\]
Then $(N_p,\varphi_i,\xi_i,\eta_i,g_\lambda)$ is a $3$-$(\alpha,\delta)$-Sasaki
manifold with $2\alpha=\lambda$ and $\delta=0$.
%Since the $1$-forms $\theta_l$ are closed, then
%%
%\begin{equation}\label{dPhi_H}
%d\Phi_i={}-d\eta_j\wedge\eta_k+\eta_j\wedge d\eta_k.
%\end{equation}
%%
\end{ex}
The previous example suggests that there might be a deeper relation between
$3$-$(\alpha,\delta)$-Sasaki manifolds and quaternionic contact manifolds. We are
now going to explain this relation.

A \emph{quaternionic contact manifold} is a $(4n+3)$-dimensional manifold with a distribution
$\H$ of codimension $3$ and a $\mathrm{Sp}(n)\mathrm{Sp}(1)$ structure locally defined by
$1$-forms $\tilde\eta_1,\tilde\eta_2,\tilde\eta_3$. Then, $\H=\cap_{i=1}^3Ker(\tilde\eta_i)$
carries a positive definite symmetric tensor $g$, called the horizontal metric and a
compatible rank-three bundle $\mathbb{Q}$ consisting of endomorphisms of $\mathcal H$ locally
generated by three almost complex structures $I_1$, $I_2$, $I_3$, such that
\begin{enumerate}[\normalfont i)]
\item $I_1I_2=I_3$,
\item $ g(I_i\cdot,I_i\cdot)= g(\cdot,\cdot)$,
\item $2g(I_iX,Y)=d\tilde\eta_i(X,Y)$, $X,Y\in \Gamma(\H)$.
\end{enumerate}
In dimension at least eleven, such a manifold admits a unique distribution $\V$, called the
vertical distribution, supplementary to $\H$, and a unique linear connection, called the
Biquard connection, satisfying distinguished conditions \cite{Biquard00}.
 The vertical distribution is
locally generated by the Reeb vector fields $\tilde\xi_1,\tilde\xi_2,\tilde\xi_3$ such that
\[
\tilde\eta_i(\tilde\xi_j)=\delta_{ij},\qquad (\tilde\xi_i\lrcorner d\tilde\eta_j)|_{\H}=0,
\qquad (\tilde\xi_i\lrcorner d\tilde\eta_j)|_{\H}
=-(\tilde\xi_j\lrcorner d\tilde\eta_i)|_{\H}.
\]
In dimension $7$ the existence of the Biquard connection still holds provided that
one assumes the existence of the Reeb vector fields \cite{Duchemin06}.

In \cite{IMV14} the authors introduce \emph{quaternionic contact Einstein manifolds},
briefly \emph{qc-Einstein manifolds}, defined as quaternionic contact manifolds such that
\[
\Ric(X,Y)=\frac{\Scal}{4n} g(X,Y),\qquad X,Y\in \Gamma(\H),
\]
where $\Ric$ and $\Scal$ are respectively the qc Ricci tensor and the qc scalar curvature
associated to the Biquard connection. In \cite{IMV} a qc-Einstein manifold is characterized
as a quaternionic contact manifold whose structure satisfies
\begin{equation}\label{qcEinstein}
d\tilde\eta_i=2\omega_i+S\tilde\eta_j\wedge\tilde\eta_k,
\end{equation}
for every even permutation $(i,j,k)$ of $(1,2,3)$, where $S$ is a constant related to
the qc scalar curvature by $8n(n+2)S=\Scal$, and each $\omega_i$ is a $2$-form defined by
\[
\xi\lrcorner\omega_i=0,\qquad 2\omega_i(X,Y)=d\tilde\eta_i(X,Y)\qquad \xi\in \Gamma(\V),\;X,Y\in \Gamma(\H).
\]
Now, the horizontal metric $g$ can be extended to a metric $h$ on $M$ by requiring that
$\H$ and $\V$ are orthogonal and $h(\tilde\xi_i,\tilde\xi_j)=\delta_{ij}$. Furthermore,
one can consider a one-parameter family of (pseudo)Riemannian metrics
$h^\lambda$, $\lambda\ne0$, defined by
\begin{equation}\label{hlambda}
h^\lambda(X,Y)=h(X,Y),\;\; h^\lambda(X,\tilde\xi_i)=0,\;\; h^\lambda(\tilde\xi_i,\tilde\xi_j)=\lambda h(\tilde\xi_i,\tilde\xi_j)=\lambda\delta_{ij}, \;\; X,Y\in \Gamma(\H).
\end{equation}
%
%or equivalently,
%\[h^\lambda=h+(\lambda-1)(\tilde\eta_1\otimes\tilde\eta_1+\tilde\eta_2\otimes\tilde\eta_2+\tilde\eta_3\otimes\tilde\eta_3).\]
%
In \cite{IMV} it is proved that  on a qc-Einstein manifold the (pseudo)Riemannian
Ricci and scalar curvatures of $h^\lambda$ are given by
\begin{equation}\label{Ricci_lambda}
\Ric^\lambda(A,B)=\Big(4n\lambda+\frac{S^2}{2\lambda}\Big)h^\lambda(A_\V,B_\V)
+\big(2S(n+2)-6\lambda\big) h^\lambda(A_\H,B_\H),
\end{equation}
\bdm
\Scal^\lambda=\frac{1}{\lambda}\Big(-12n\lambda^2+8n(n+2)S\lambda+\frac{3}{2}S^2\Big),
\edm
where, for a vector field $A$ on $M$, $A_\V$ and $A_\H$ denote the orthogonal
projections of $A$ on $\H$ and $\V$, respectively.
\begin{prop}\label{prop.3-AD-Sasaki-is-qc-Einstein}
%---------------------------------------------------
Every $3$-$(\alpha,\delta)$-Sasaki manifold $(M,\varphi_i,\xi_i,\eta_i,g)$
admits an underlying quaternionic contact structure which is qc-Einstein with $S=2\alpha\delta$,
and its Riemannian Ricci curvature is given by
\bdm
\Ric^g\,=\, 2\alpha\big(2\delta(n+2)-3\alpha\big)g+2(\alpha-\delta)\big((2n+3)\alpha-\delta\big)\sum_{i=1}^3\eta_i\otimes\eta_i.
\edm
In particular, a $3$-$(\alpha,\delta)$-Sasaki manifold is Riemannian Einstein if and only if
$\delta=\alpha$ or $\delta=(2n+3)\alpha$.
\end{prop}
\begin{proof}
%----------------
The underlying quaternionic contact structure is given by the horizontal distribution $\H$,
the Riemannian metric $g$, the almost complex structures $I_i$ and the $1$-forms
$\tilde\eta_i$ defined by
\[
I_i:=\varphi_i|_{\H},\qquad \tilde\eta_i:=-\frac{1}{\alpha}\eta_i.
\]
Indeed, for all horizontal vector fields $X,Y$
\[
d\tilde\eta_i(X,Y)=-\frac{1}{\alpha}d\eta_i(X,Y)=-2\Phi_i(X,Y)=2g(I_iX,Y).
\]
Let us show that this is in fact a qc-Einstein structure. For this,
let $\omega_i$ be the $2$-form defined by
\[
\xi_r\lrcorner\omega_i=0,\quad2\omega_i(X,Y)=d\tilde\eta_i(X,Y), \qquad X,Y\in\Gamma(\H),
\]
hence
\[
\omega_i=-\Phi_i-\eta_j\wedge\eta_k.
\]
Since $d\eta_i=2\alpha\Phi_i+2(\alpha-\delta)\eta_j\wedge\eta_k$, we have
\begin{eqnarray*}
d\tilde\eta_i&=&-2\Phi_i-2\frac{\alpha-\delta}{\alpha}\eta_j\wedge\eta_k
\ =\   2\omega_i+2\eta_j\wedge\eta_k-2\frac{\alpha-\delta}{\alpha}\eta_j\wedge\eta_k\\
&=&2\omega_i+2\frac{\delta}{\alpha}\eta_j\wedge\eta_k
\ =\ 2\omega_i+2\alpha\delta\,\tilde\eta_j\wedge\tilde\eta_k.
\end{eqnarray*}
Therefore, \eqref{qcEinstein} is satisfied with $S=2\alpha\delta$. Now, the Reeb vector
fields associated to the qc structure are $\tilde\xi_i:=-\alpha\xi_i$, and the Riemannian
metric $g$ coincides with the Riemannian metric $h^\lambda$ defined in \eqref{hlambda}
with $\lambda=\alpha^2$. Therefore, the Ricci tensor of $g$ is
\[
\Ric^g(A,B)=\left(4n\alpha^2+2\delta^2\right)g(A_\V,B_\V)
+\left(4\alpha\delta(n+2)-6\alpha^2\right)g(A_\H,B_\H).
\]
Since $g(A_\V,B_\V)=\sum_{i=1}^3\eta_i(A)\eta_i(B)$ and
$g(A_\H,B_\H)=g(A,B)-\sum_{i=1}^3\eta_i(A)\eta_i(B)$,
applying \eqref{Ricci_lambda} from \cite{IMV}, we have
\begin{align*}
\Ric^g(A,B)&=\left(4\alpha\delta(n+2)-6\alpha^2\right)g(A,B)
+2\left(\delta^2-2\alpha\delta(n+2)+(2n+3)\alpha^2\right)\sum_{i=1}^3\eta_i(A)\eta_i(B)\\
&=\left(4\alpha\delta(n+2)-6\alpha^2\right)g(A,B)+2(\alpha-\delta)\big((2n+3)\alpha-\delta\big)\sum_{i=1}^3\eta_i(A)\eta_i(B).\qedhere
\end{align*}
\end{proof}
In Remark \ref{rem.can-conn-and-qc-Einstein}, we will establish that the canonical connection of
a  $3$-$(\alpha,\delta)$-Sasaki manifold is indeed a qc connection for the
underlying quaternionic contact structure.
%
%
%------------------------------------------------------------------------------------------
\section{$\varphi$-compatible connections of almost $3$-contact metric manifolds}
%------------------------------------------------------------------------------------------
\subsection{General existence of $\varphi$-compatible connections}
%------------------------------------------------------------------------------------------
%
\begin{df}
%------------
Let $(M,\varphi_i,\xi_i,\eta_i, g)$  be an almost $3$-contact metric manifold,
$\varphi$ an almost contact metric structure in the associated sphere $\Sigma_M$. A linear
connection $\nabla$ on $M$ will be called a \emph{$\varphi$-compatible connection} if it is a
metric connection with skew torsion preserving the splitting  $TM=\H\oplus \V$, and such that
\begin{equation}\label{nabla_phi}
(\nabla_X\varphi)Y=0\qquad\forall X,Y\in\Gamma(\H).
\end{equation}
\end{df}
Notice that, given a metric connection $\nabla$  on $M$, the requirement that $\nabla$ preserves the splitting of the tangent bundle is equivalent to any of the following conditions:
\begin{enumerate}[a)]
\item $\nabla_XY\in \Gamma(\H)$ for every $X\in{\frak X}(M)$ and $Y\in\Gamma(\H)$;
\item $\nabla_X\xi_i\in \Gamma(\V)$ for every $X\in{\frak X}(M)$ and $i=1,2,3$.
\end{enumerate}
Therefore, if the splitting is preserved, $\H$ being  $\varphi$-invariant, condition
\eqref{nabla_phi} is equivalent to
\begin{equation}\label{nabla_phi1}
g((\nabla_X\varphi)Y,Z)=0\qquad \forall X,Y,Z\in\Gamma(\H).
\end{equation}
Comparing to the characteristic connection of an almost contact metric manifold,
we see that condition \eqref{nabla_phi} is weaker than the requirement $\nabla \varphi=0$,
whereas the requirement to preserve the distributions $\H$ and $\V$ was, in the previous
situation, an automatic consequence.
\begin{rem}
%---------------
Notice that for any structure $\varphi\in\Sigma_M$, setting $J:=\varphi|_{\H}:\H\to \H$, one has $J^2=-I$. Therefore, $(\H,J)$ is an almost $CR$ structure on $M$, compatible with the Riemannian metric $g$.
In \cite{DL} the authors study metric connections with skew torsion on a Riemannian
manifold $(M,g)$ endowed with a compatible almost $CR$ structure $(\H,J)$. The
characteristic connections defined in that case are required to parallelize the structure
$(\H,J)$. In our approach to almost $3$-contact metric manifolds, the condition of
$\varphi$-compatibility is weaker, since we do not require the parallelism of the tensor
$J$ along Reeb vector fields.
\end{rem}
We shall determine necessary and sufficient conditions for an almost $3$-contact metric
manifold $(M,\varphi_i,\xi_i,\eta_i,g)$ to admit $\varphi$-compatible connections.
\begin{theo}[Existence of $\varphi$-compatible connections]\label{theo_compatible}
%----------------------------------------------------------------------------------
Let $(M,\varphi_i,\xi_i,\eta_i, g)$  be an almost $3$-contact metric manifold,
$\varphi$ an almost contact metric structure in the associated sphere $\Sigma_M$.
Then $M$ admits a $\varphi$-compatible connection if and only if the following conditions are
satisfied:
\begin{enumerate}[\normalfont i)]
\item the tensor field $N_\varphi$ is skew-symmetric on $\H$;
\item $({\mathcal L}_{\xi_i}g)(X,Y)=0$ for every $X,Y\in\Gamma(\H)$ and $i=1,2,3$;
\item $({\mathcal L}_Xg)(\xi_i,\xi_j)=0$ for every $X\in\Gamma(\H)$ and $i,j=1,2,3$.
\end{enumerate}
The torsion of any $\varphi$-compatible connection satisfies
($X,Y,Z\in \Gamma (\H),\ i,j=1,2,3$):
\begin{align}
T(X,Y,Z)&= N_\varphi(X,Y,Z)-d\Phi(\varphi X,\varphi Y,\varphi Z)\label{TNS},\\
T(X,Y,\xi_i)&= d\eta_i(X,Y)\label{T1},\\
T(X,\xi_i,\xi_j)&=-g([\xi_i,\xi_j],X).\label{T2}
\end{align}
\end{theo}
\begin{rem}\label{rem.torsion-compatible}
%-------------------------------------------
Thus, the $\varphi$-compatible connections are parametrized by  smooth functions
$T(\xi_1,\xi_2,\xi_3):=\gamma \in C^\infty(M)$. Once $\gamma$ is fixed, the torsion of the
corresponding $\varphi$-compatible connection is given by
\begin{align*}
&T(X,Y,Z)=N_\varphi(X_\H,Y_\H,Z_\H)-d\Phi(\varphi X_\H,\varphi Y_\H,\varphi Z_\H)\\
&\quad+\sum_i\{\eta_i(X)d\eta_i(Y_\H,Z_\H)+\eta_i(Y)d\eta_i(Z_\H,X_\H)+\eta_i(Z)d\eta_i(X_\H,Y_\H)\}\\
&\quad-\sum_{i,j} \{\eta_i(X)\eta_j(Y)g([\xi_i,\xi_j],Z_\H)
+\eta_i(Y)\eta_j(Z)g([\xi_i,\xi_j],X_\H)\\
&\qquad\qquad+\eta_i(Z)\eta_j(X)g([\xi_i,\xi_j],Y_\H)\}
+ \gamma\eta_{123}(X,Y,Z),
\end{align*}
where $X_\H$ denotes the horizontal part of a vector field $X$.
Sometimes, we will call the chosen function $\gamma$ the \emph{parameter function} of the
$\varphi$-compatible connection $\nabla$.
\end{rem}
\begin{proof}
%-------------
First, one can easily check that for every $X,Y,Z\in\Gamma(\H)$,
\begin{equation}\label{Nphi}
N_\varphi(X,Y,Z)=g((\nabla^g_{\varphi X}\varphi)Y-(\nabla^g_{Y}\varphi)\varphi X
+(\nabla^g_{X}\varphi)\varphi Y-(\nabla^g_{\varphi Y}\varphi)X,Z).
\end{equation}
Let us suppose that $M$ admits a $\varphi$-compatible connection $\nabla$, with
torsion $T$. Then, from \eqref{nabla} and \eqref{nabla_phi1} we have
\begin{equation}
2 g((\nabla^g_X\varphi)Y,Z)+T(X,\varphi Y,Z)+T(X,Y,\varphi Z)=0\label{gammaJT}
\end{equation}
for every $X,Y,Z\in\Gamma(\H)$. Using \eqref{Nphi} and \eqref{gammaJT} we get
\begin{equation}\label{NT}
N_\varphi(X,Y,Z)=T(X,Y,Z)-T(\varphi X,\varphi Y,Z)-T(\varphi X,Y,\varphi Z)
-T(X,\varphi Y,\varphi Z),
\end{equation}
which implies that $N_\varphi$ is skew-symmetric on $\H$.
Since $\nabla g=0$ and the connection preserves the splitting of the tangent
bundle, we have
\begin{align*}
&({\mathcal L}_{\xi_i} g)(X,Y)=\\&= \xi_i(g(X,Y))-g(\nabla_{\xi_i} X-\nabla_X\xi_i-T(\xi_i,X),Y)
 -g(X,\nabla_{\xi_i} Y-\nabla_Y\xi_i-T(\xi_i,Y))\\
&=T(\xi_i,X,Y)+T(\xi_i,Y,X)=0,
\end{align*}
which proves ii). Analogously, we get iii).
Before proving the converse, we verify equations \eqref{TNS}, \eqref{T1} and \eqref{T2}.
Equations \eqref{T1} and \eqref{T2} are immediate consequence of the fact that $\nabla$
preserves the splitting $TM=\H\oplus \V$. As regards \eqref{TNS}, applying
\eqref{gammaJT}, for every $X,Y,Z\in\Gamma (\H)$, we have
\begin{equation}\label{SgammaT}
d\Phi(X,Y,Z)=-\cyclic{XYZ}\,g((\nabla^g_X\varphi )Y,Z)
=\cyclic{XYZ}\,T(X,Y,\varphi Z),
\end{equation}
where $\cyclic{XYZ}$ denotes the cyclic sum over $X,Y,Z$.
Applying \eqref{NT} and \eqref{SgammaT}, we obtain \eqref{TNS}.

As for the converse, let us suppose that i)-iii) hold. Let $T$ be the $3$-form on $M$
defined by \eqref{TNS}-\eqref{T2} and
\begin{align*}
T(\xi_i,X,Y)&=-T(X,\xi_i,Y)=T(X,Y,\xi_i),\\
T(\xi_i,\xi_j,X)&=-T(\xi_i,X,\xi_j)=T(X,\xi_i,\xi_j),\\
T(\xi_i,\xi_j,\xi_k)&=\epsilon_{ijk}\gamma,
\end{align*}
for every $X,Y\in\Gamma(\H)$ and $i,j,k=1,2,3$, where $\gamma$ is a smooth function on $M$. Let $\nabla$ be the metric
connection with totally skew-symmetric torsion $T$, which is given by \eqref{nabla}.
%
%In order to check that $T$ is a $3$-form, a simple computation using \eqref{N} gives
%\begin{equation}\label{Tgamma}
%T(X,Y,Z)=g((\Gamma_XJ)Y-(\Gamma_YJ)X,JZ)-g((\Gamma_{JZ}J)X,Y),
%\end{equation}
%for every $X,Y,Z\in\Gamma \H$,
%which implies that $T(X,X,Z)=0$. Moreover,
%\[T(X,Y,X)= -g((\Gamma_XJ)JX+(\Gamma_{JX}J)X,Y)-g((\Gamma_YJ)X,JX).\]
%Since $N$ is skew-symmetric, from c) of Proposition \ref{skew},
%it follows that $(\Gamma_XJ)JX+(\Gamma_{JX}J)X=0$. On the other hand, $g((\Gamma_YJ)X,JX)=0$ since the operators $\Gamma_YJ$ and $J$ are skew-symmetric and anticommute. Thus $T(X,Y,X)=0$.
%
%
We prove that $\nabla$ preserves the splitting $TM=\H\oplus \V$. Indeed, for
every $X,Y\in\Gamma (\H)$ and $i=1,2,3$, we have
\bdm
g(\nabla_XY,\xi_i)\, =\, g(\nabla^g_XY,\xi_i)-\frac{1}{2}\,g([X,Y],\xi_i)
\, =\, \frac{1}{2}\,g(\nabla^g_XY+\nabla^g_YX,\xi_i)
\, =\, -\frac{1}{2}\,({\mathcal L}_{\xi_i}g)(X,Y)
\edm
which vanishes because of ii). Analogously, from iii), for every $X\in\Gamma(\H)$ and
$i,j=1,2,3$,
\[g(\nabla_{\xi_j} X,\xi_i)=\frac{1}{2}\,({\mathcal L}_Xg)(\xi_i,\xi_j)=0.\]
In order to prove \eqref{nabla_phi1}, a simple computation using \eqref{TNS}
and \eqref{Nphi} gives
\begin{equation}\label{Tgamma}
T(X,Y,Z)=g((\nabla^g_X\varphi)Y-(\nabla^g_Y\varphi)X,\varphi Z)
-g((\nabla^g_{\varphi Z}\varphi)X,Y)
\end{equation}
for every $X,Y,Z\in\Gamma (\H)$. Applying the above formula and Lemma \ref{skew} 3), we have
\begin{eqnarray*}
T(X,\varphi Y,Z)+T(X,Y,\varphi Z)&= & g((\nabla^g_X\varphi )\varphi Y
-(\nabla^g_{\varphi Y}\varphi )X,\varphi Z)-g((\nabla^g_{\varphi Z}\varphi )X,\varphi Y)\\
& & -g((\nabla^g_X\varphi )Y-(\nabla^g_Y\varphi )X,Z)+g((\nabla^g_{Z}\varphi )X,Y)\\
&=&  {}-2 g((\nabla^g_X\varphi )Y,Z).
\end{eqnarray*}
Therefore,
\[
g((\nabla_X\varphi)Y,Z)= g((\nabla^g_X\varphi )Y,Z)
+\frac{1}{2}\,(T(X,\varphi Y,Z)+T(X,Y,\varphi Z))=0. \qedhere
\]
\end{proof}
\begin{rem}
%------------
Condition iii) in Theorem \ref{theo_compatible} is equivalent to
%being $$({\mathcal L}_Xg)(\xi_i,\xi_j)=-g([X,\xi_i],\xi_j)-g(\xi_i,[X,\xi_j])$$
%
%\begin{equation}\label{iii_1}
\bdm
d\eta_i(X,\xi_j)+d\eta_j(X,\xi_i)=0,
\edm
%\end{equation}
%
or also
\begin{equation}\label{iii_2}
g(\nabla^g_{\xi_i}\xi_j,X)+g(\nabla^g_{\xi_j}\xi_i,X)=0
\end{equation}
for every $X\in\Gamma(\H)$ and $i,j=1,2,3$.
\end{rem}
A particularly simple situation occurs when  $N_\varphi$ vanishes on $\H$.
We give some simple-to-check criteria when this happens.
\begin{prop}%\label{N=0}
%-------------------------
Let $(M,\varphi_i,\xi_i,\eta_i, g)$  be an almost $3$-contact metric manifold, and
$\varphi\in\Sigma_M$. Assume that $M$ admits a $\varphi$-compatible connection $\nabla$ with
torsion $T$. Then, the following conditions are equivalent:
\begin{enumerate}[\normalfont a)]
\item \ $g((\nabla^g_X\varphi)Y,Z)=0$ for every $X,Y,Z\in\Gamma(\H)$,
\item \ $T(X,Y,Z)=0$ for every $X,Y,Z\in\Gamma(\H)$,
\item \ $d\Phi(X,Y,Z)=0$ for every $X,Y,Z\in\Gamma(\H)$.
\end{enumerate}
If any of these conditions is satisfied, $N_\varphi(X,Y,Z)=0$ for any $X,Y,Z\in\Gamma(\H)$.
\end{prop}
\begin{proof}
%---------------
The equivalence of a) and b) follows from \eqref{Tgamma} and \eqref{gammaJT}.
Condition a) obviously implies c). Conversely, supposing that c) holds, from \eqref{TNS}
it follows that $T$ and $N_\varphi$ coincide on $\H$. Hence, from \eqref{SgammaT}, we have
\[N_\varphi(X,Y,\varphi Z)+N_\varphi(Y,Z,\varphi X)+N_\varphi(Z,X,\varphi Y)=0\]
for every $X,Y,Z\in\Gamma(\H)$. Since $N_\varphi$ is skew-symmetric and
$N_\varphi(\varphi X,Y)=-\varphi N_\varphi(X,Y)$, we deduce that $N_\varphi(X,Y,\varphi Z)=0$.
Therefore b), or equivalently a), holds.
\end{proof}
Assume that  $M$ admits
$\varphi_i$-compatible connections, in the sense that conditions i)-iii) in Theorem
\ref{theo_compatible} are satisfied for each structure $(\varphi_i,\xi_i,\eta,g)$.
We would like to  conclude that $M$ admits $\varphi$-compatible connections for
\emph{any} $\varphi\in\Sigma_M$. Conditions ii) and iii) do not depend on the choice of
$\varphi$, hence there is nothing to check. To verify condition i) for $\varphi$, we
need  to know that $N_\varphi$
is skew-symmetric on $\H$ if this is true for each $N_{\varphi_i}$---but this is exactly the contents of our Proposition \ref{prop.N-skew-if-Ni-skew}.
Hence, we can state the following remarkable corollary from
Proposition \ref{prop.N-skew-if-Ni-skew} and Theorem \ref{theo_compatible}.
It underlines once more that the associated sphere is a very canonical object.
\begin{cor}
%---------------
Let $(M,\varphi_i,\xi_i,\eta_i, g)$  be an almost $3$-contact metric manifold.  If $M$
admits $\varphi_i$-compatible connections for every $i=1,2,3$, then $M$ admits
$\varphi$-compatible connections for every almost contact metric structure
$\varphi$ in the associated sphere $\Sigma_M$.
\end{cor}

%------------------------------------------------------------------------------------------
\subsection{Effect  of  Killing Reeb vector fields}
%\subsection{$\varphi$-compatible connections in the presence of  Killing Reeb vector fields}
%------------------------------------------------------------------------------------------
%
%
Looking back at the conditions for the existence of $\varphi$-compatible connections from
 Theorem \ref{theo_compatible}, we see that conditions ii) and iii) on the Reeb vector fields
$\xi_i$ are rather weak---in particular, they don't need to be Killing. Nevertheless, it
is a familiar fact from examples that this is the case in many interesting classes of
almost $3$-contact metric manifolds. Thus, we investigate the situation of Killing Reeb
vector fields separately in this section.
\begin{prop}\label{H-parallel}
%-------------------------------
Let $(M,\varphi_i,\xi_i,\eta_i, g)$ be an almost $3$-contact metric manifold,
$\varphi$ an almost contact metric structure in the associated sphere $\Sigma_M$. Given a
$\varphi$-compatible connection $\nabla$, the Reeb vector fields are $\nabla$-parallel
along the distribution $\H$  if and only if $({\mathcal L}_{\xi_i}g)(X,\xi_j)=0$ for
any $X\in\Gamma(\H)$ and $i,j=1,2,3$.
\end{prop}
\begin{proof}
%------------
If $\nabla$ is a $\varphi$-compatible connection, we deduce for every
$X\in\Gamma(\H)$ and $i,j=1,2,3$ from \eqref{nabla}, \eqref{T2}
and \eqref{iii_2}
\begin{align*}
g(\nabla_X\xi_i,\xi_j)&=g(\nabla^g_X\xi_i,\xi_j)-\frac12 g([\xi_i,\xi_j],X)
=g(\nabla^g_X\xi_i,\xi_j)-\frac12 g(\nabla^g_{\xi_i}\xi_j-\nabla^g_{\xi_j}\xi_i,X)\\
&=g(\nabla^g_X\xi_i,\xi_j)+g(\nabla^g_{\xi_j}\xi_i,X) =({\mathcal L}_{\xi_i}g)(X,\xi_j).
\end{align*}
Since the distribution $\V$ is parallel with respect to the connection
$\nabla$, we have that $\nabla_X\xi_i=0$ for every $X\in\Gamma(\H)$ and $i=1,2,3$,
if and only if the Lie derivatives $({\mathcal L}_{\xi_i}g)(X,\xi_j)$ are all vanishing.
\end{proof}
The following proposition shows that when the Reeb vector fields are Killing, the existence
of $\varphi$-compatible connections and the existence of the Reeb commutator
function (Definition \ref{df.Reebcf}), garanteed by Corollary \ref{cor.Reebcf-Killing},
are intricately related. In return, the Reeb commutator function and the parameter
function of a $\varphi$-compatible connection $\nabla$ describe the $\nabla$-derivative of
one Reeb vector field through the other Reeb vector fields in a very symmetric expression.
\begin{prop}\label{prop_Killing}
%----------------------------------
Let $(M,\varphi_i, \xi_i,\eta_i,g)$
be an almost $3$-contact metric manifold, $\varphi$
an almost contact metric structure in the associated sphere. Assume that
\begin{enumerate}[\normalfont i)]
\item the tensor field $N_\varphi$ is skew-symmetric on $\H$;
\item each $\xi_i$  is a Killing vector field.
\end{enumerate}
Let $\delta$ be its Reeb commutator function. Then $M$ admits $\varphi$-compatible
connections. If $\nabla$ is any $\varphi$-compatible connection with torsion
$T$ and parameter function $\gamma$, the following equations hold:
\begin{enumerate}[\normalfont 1)]
\item For every $X\in{\frak X}(M)$, and
 for every even permutation $(i,j,k)$ of $(1,2,3)$,
\begin{equation}\label{nabla_xi}\textstyle
\nabla_X\xi_i=\frac{2\delta+\gamma}{2}(\eta_k(X)\xi_j-\eta_j(X)\xi_k).
\end{equation}
\item For every $i=1,2,3$, and for every $X,Y\in\Gamma(\H)$,
\begin{equation}\label{nabla_xi_phi}
g((\nabla_{\xi_i}\varphi)X,Y)= g(({\mathcal L}_{\xi_i}\varphi)X,Y)
+d\eta_i(\varphi X,Y)+d\eta_i(X,\varphi Y).
\end{equation}
\end{enumerate}
\end{prop}
\begin{proof}
%-------------
First we prove that condition iii) in Theorem \ref{theo_compatible} is satisfied. Indeed,
one can easily check that for every $X\in\Gamma(\H)$ and $i,j=1,2,3$,
\[
({\mathcal L}_Xg)(\xi_i,\xi_j)=-({\mathcal L}_{\xi_i}g)(X,\xi_j)
-({\mathcal L}_{\xi_j}g)(X,\xi_i)=0.
\]
Hence, by Theorem \ref{theo_compatible}, $M$ admits $\varphi$-compatible connections. Given
a $\varphi$-compatible connection $\nabla$ with torsion $T$, let
$\gamma=T(\xi_1,\xi_2,\xi_3)$ be the parameter function. Recall that by Lemma
\ref{lemma_xi_i}, the Reeb commutator function $\delta$ satisfies
\begin{equation}\label{cyclic}
\eta_k([\xi_i,\xi_j])=2\,\eta_k(\nabla^g_{\xi_i}\xi_j)=2\delta\epsilon_{ijk}
\end{equation}
for every $i,j,k=1,2,3$.
In order to verify \eqref{nabla_xi}, by Proposition \ref{H-parallel}, we have
$\nabla_X\xi_i=0$ for every $X\in\Gamma(\H)$. Recall that the distribution
$\V$ is parallel with respect to the $\varphi$-compatible connection. Therefore,
since $\nabla_{\xi_i}\xi_i=\nabla^g_{\xi_i}\xi_i\in\Gamma(\V)$, formula
\eqref{cyclic} gives $\nabla_{\xi_i}\xi_i=0$. Now, if $(i,j,k)$ is an even permutation
of $(1,2,3)$, for the covariant derivative $\nabla_{\xi_j}\xi_i$, we have
\[
g(\nabla_{\xi_j}\xi_i,\xi_i)=0,\quad g(\nabla_{\xi_j}\xi_i,\xi_j)=-g(\xi_i,\nabla_{\xi_j}\xi_j)=0,
\]
and applying again \eqref{cyclic},
\[
g(\nabla_{\xi_j}\xi_i,\xi_k)= g(\nabla^g_{\xi_j}\xi_i,\xi_k)+\frac12 T(\xi_j,\xi_i,\xi_k)
=-\frac{2\delta+\gamma}{2}.
\]
Therefore, $\nabla_{\xi_j}\xi_i=-\frac{2\delta+\gamma}{2}\xi_k$. Analogously, one checks that
$\nabla_{\xi_k}\xi_i=\frac{2\delta+\gamma}{2}\xi_j$, completing the proof of \eqref{nabla_xi}.

We prove \eqref{nabla_xi_phi}. Applying the Koszul formula for the Levi-Civita connection,
for every $X,Y\in\Gamma(\H)$ we have
\begin{align}\label{nabla_g_phi1}
2g((\nabla^g_{\xi_i}\varphi)X,Y)&=2 g(\nabla^g_{\xi_i}(\varphi X),Y)
+ 2 g(\nabla^g_{\xi_i}X,\varphi Y)\nonumber\\
&=g([\xi_i,\varphi X],Y)-g([\varphi X,Y],\xi_i)+g([Y,\xi_i],\varphi X)\nonumber\\
&\quad +g([\xi_i,X],\varphi Y)-g([X,\varphi Y],\xi_i)+g([\varphi Y,\xi_i],X).
\end{align}
Since $\xi_i$ is a Killing vector field, we get
\begin{align}\label{nabla_g_phi2}
0&=({\mathcal L}_{\xi_i}g)(X,\varphi Y)+({\mathcal L}_{\xi_i}g)(\varphi X,Y)\nonumber\\
&= -g([\xi_i,X],\varphi Y)-g([\xi_i,\varphi Y],X) -g([\xi_i,\varphi X],Y)
-g([\xi_i,Y],\varphi X).
\end{align}
Therefore, from \eqref{nabla_g_phi1} and \eqref{nabla_g_phi2} it follows that
\begin{align*}
2\,g((\nabla^g_{\xi_i}\varphi)X,Y)&=2g([\xi_i,\varphi X],Y)+2\,g([\xi_i,X],\varphi Y)
-\eta_i([X,\varphi Y])-\eta_i([\varphi X,Y])\\
&=2\,g(({\mathcal L}_{\xi_i}\varphi)X,Y))+d\eta_i(X,\varphi Y)+d\eta_i(\varphi X,Y).
\end{align*}
Finally, applying the above formula and \eqref{T1},
\begin{align*}
2\,g((\nabla_{\xi_i}\varphi)X,Y)&= 2\,g((\nabla^g_{\xi_i}\varphi)X,Y)
+T(\xi_i,\varphi X,Y)+T(\xi_i,\varphi X,Y)\\
&= 2\,g(({\mathcal L}_{\xi_i}\varphi)X,Y))+2\,d\eta_i(X,\varphi Y)+2\,d\eta_i(\varphi X,Y),
\end{align*}
which proves \eqref{nabla_xi_phi}.
\end{proof}
\begin{rem}
%----------
We recognize that the right hand side of eq. \eqref{nabla_xi_phi} in the previous Proposition
is just, for $\varphi=\varphi_j$, the tensor field $A_{ji}$ introduced in equation \eqref{Aij}.
\end{rem}
%
%------------------------------------------------------------------------------------------
\section{The canonical connection of an almost $3$-contact metric manifold}
%\label{sec.can-connection}
%------------------------------------------------------------------------------------------
%
%------------------------------------------------------------------------------------------
\subsection{General existence of the canonical connection}%\label{subsec.can-connection}
%------------------------------------------------------------------------------------------
%
In the following we will provide a criterion allowing to define a unique metric connection
with skew torsion on an almost $3$-contact metric manifold $(M,\varphi_i,\xi_i,\eta_i,g)$,
called the \emph{canonical connection}. Its crucial property is captured by
equation \eqref{canonical}, from which all others will follow. Hence, we start by explaining
what singles out this particular condition.
\begin{rem}\label{rem.qK-geometry}
%----------------------------
Recall that a hyper-K\"ahler manifold can be defined as a Riemannian manifold of dimension
$4n\geq 4$ admitting an anti-commuting pair $I_1,I_2$ of integrable complex structures, relative
to both of which the metric is K\"ahler. This implies that $I_1,I_2$, and $I_3:=I_1I_2$
are parallel for the Levi-Civita connection $\nabla^g$. In contrast, on a general
quaternion-K\"ahler  manifold it is not possible to find individual structures $I_1,I_2,I_3$
that are parallel, but only a bundle of endomorphisms 
(namely, the one spanned by $I_1,I_2,I_3$)\ preserved as a whole; more precisely, there should 
exist $1$-forms $\alpha_i$ such that
\bdm
\nabla^g_X I_i \, =\, - \alpha_k(X)I_j + \alpha_j(X)I_k \quad
\forall X\in{\frak X}(M)
\edm
for every even permutation $(i,j,k)$ of $(1,2,3)$. These equations were first considered by
\cite{Ishihara74}, see also \cite{Salamon99} for a very nice survey. Now, the analogy to
equation \eqref{canonical} becomes obvious: A canonical connection is one mimicking the
derivative equations of quaternion-K\"ahler geometry, with $I_i$ replaced by $\varphi_i$
and $\alpha_i$  replaced by $-\tilde\beta\eta_i$.
\end{rem}
\begin{theo}[Existence of the canonical connection]\label{theo_canonical}
%------------------------------------------------------------------------
Let $(M,\varphi_i,\xi_i,\eta_i, g)$  be an almost $3$-contact metric manifold. Then $M$
admits a metric connection $\nabla$ with skew torsion such that for some smooth
function $\tilde\beta$,
\begin{equation}\label{canonical}
\nabla_X\varphi_i\, =\, \tilde\beta(\eta_k(X)\varphi _j -\eta_j(X)\varphi _k) \quad
\forall X\in{\frak X}(M)
\end{equation}
for every even permutation $(i,j,k)$ of $(1,2,3)$, if and only if it is a
canonical almost $3$-contact metric manifold.

If such a connection $\nabla$ exists, it is unique and it is $\varphi$-compatible for every
almost contact metric structure $\varphi$ in the associated sphere $\Sigma_M$, and
$\tilde\beta$ coincides with the  Reeb Killing function $\beta$. The torsion of $\nabla$ is
given by \eqref{TNS}-\eqref{T2} and the parameter function is
\begin{equation}\label{T_canonical}
\gamma\ :=\  T(\xi_1,\xi_2,\xi_3)\ =\ 2(\beta -\delta),
\end{equation}
where $\delta$ is the Reeb commutator function.
\end{theo}
\begin{proof}
%-------------
%
Let us assume that $M$ admits a metric connection $\nabla$ with skew torsion
satisfying \eqref{canonical}. First we show that
\begin{equation}\label{canonical_nabla_xi}
\nabla_X\xi_i=\tilde\beta(\eta_k(X)\xi_j-\eta_j(X)\xi_k)
\end{equation}
for every $X\in{\frak X}(M)$ and for every even permutation $(i,j,k)$ of $(1,2,3)$.
Indeed, from \eqref{canonical} we have
\begin{equation}\label{canonical_nabla_xi1}
(\nabla_X\varphi_i)\xi_i=-\tilde\beta(\eta_k(X)\xi_k +\eta_j(X)\xi_j).
\end{equation}
Since $(\nabla_X\varphi_i)\xi_i=-\varphi_i(\nabla_X\xi_i)$ and $\eta_i(\nabla_X\xi_i)=0$,
applying $\varphi_i$ on both sides of \eqref{canonical_nabla_xi1}, we get
\eqref{canonical_nabla_xi}. Therefore, being $\nabla$  a metric connection preserving the
distribution $\mathcal V$, it preserves the splitting $TM=\H\oplus\V$. On the
other hand, since $\nabla_X\varphi_i=0$ for every $X\in\Gamma(\H)$, it turns out that
$\nabla$ is a $\varphi_i$-compatible connection for all $i=1,2,3$. Then, conditions i)-iii)
in  Theorem \ref{theo_compatible} are satisfied. In particular, each $N_{\varphi_i}$ is
skew-symmetric on $\H$ and, by equation \eqref{TNS}, the torsion $T$ of $\nabla$ satisfies
\bdm
T(X,Y,Z)=N_{\varphi_i}(X,Y,Z)-d\Phi_i(\varphi_i X,\varphi_i Y,\varphi_i Z)
\edm
for every $X,Y,Z\in\Gamma(\H)$ and $i=1,2,3$, thus proving condition 3) of Definition
\ref{df.can-3-contact} of a canonical almost $3$-contact metric manifold.

In order to prove that each $\xi_i$ is a Killing vector field, we already know, by
the $\varphi_i$-compatibility, that $({\mathcal L}_{\xi_i}g)(X,Y)=0$ for every
$X,Y\in\Gamma(\H)$ and $i,j=1,2,3$. Furthermore, since $\nabla_X\xi_i=0$ for all
$X\in\Gamma(\H)$, Proposition \ref{H-parallel} implies that
$({\mathcal L}_{\xi_i}g)(X,\xi_j)=0$.
%
%
%Now, notice that equation \eqref{canonical_nabla_xi} can be written as
%\[\nabla_X\xi_j=\beta(\eta_i(X)\xi_k-\eta_k(X)\xi_i)\]
%where $(i,j,k)$ is an even permutation of $(1,2,3)$, or equivalently,
%\[\nabla_X\xi_j=\beta\sum_{i,k=1}^3\epsilon_{ijk}\,\eta_i(X)\xi_k,\]
%so that, for all indices $i,j,k$, we have $g(\nabla_{\xi_i}\xi_j,\xi_k)=\beta\epsilon_{ijk}$.
%
%
Now, one can easily check that equation \eqref{canonical_nabla_xi} implies
\[
g(\nabla_{\xi_i}\xi_j,\xi_k)=\tilde\beta\epsilon_{ijk}
\]
for all indices $i,j,k=1,2,3$. Setting $\gamma:=T(\xi_1,\xi_2,\xi_3)$, we have
\[
g(\nabla^g_{\xi_i}\xi_j,\xi_k)=g(\nabla_{\xi_i}\xi_j,\xi_k)-\frac12\, T(\xi_i,\xi_j,\xi_k)
=\frac{2\tilde\beta-\gamma}{2}\epsilon_{ijk},
\]
so that condition 3) in Lemma \ref{lemma_xi_i} is satisfied for
$\delta=\frac{2\tilde\beta-\gamma}{2}$.
Therefore, for all indices $i,j,k=1,2,3$, we have $({\mathcal L}_{\xi_i}g)(\xi_j,\xi_k)=0$,
or equivalently $\eta_k([\xi_i,\xi_j])=2\delta\epsilon_{ijk}$. This completes the proof
that each $\xi_i$ is Killing. Furthermore, notice that the linear connection $\nabla$
is uniquely determined, the parameter function being given by $\gamma=2(\tilde\beta-\delta)$.

Finally, we prove  that $\tilde\beta$ is the Reeb Killing function. By Proposition
\ref{prop_Killing}, equation \eqref{nabla_xi_phi} holds. On the other hand, taking into
account the definition of the tensor fields $A_{ij}$ in \eqref{Aij},  and applying
\eqref{canonical}, we have
\begin{align*}
A_i(X,Y)&=g((\nabla_{\xi_i}\varphi_i)X,Y)=0,\\
A_{ij}(X,Y)&=g((\nabla_{\xi_j}\varphi_i)X,Y)=-\tilde\beta g(\varphi_kX,Y)
=\tilde\beta\Phi_k(X,Y),\\
A_{ji}(X,Y)&=g((\nabla_{\xi_i}\varphi_j)X,Y)=\tilde\beta g(\varphi_kX,Y)
=-\tilde\beta\Phi_k(X,Y),
\end{align*}
for every $X,Y\in\Gamma(\H)$ and for every even permutation $(i,j,k)$ of $(1,2,3)$.
Hence, $\tilde\beta$ is the Reeb Killing function on $M$.

In order to prove the converse, let us assume that $M$ is a canonical almost $3$-contact metric manifold. First we notice
that, since each $\xi_i$ is a Killing vector field,  Lemma \ref{lemma_xi_i} implies
the existence of a Reeb commutator function $\delta$, i.\,e.
$\eta_k([\xi_i,\xi_j])=2\delta\epsilon_{ijk}$. By Proposition \ref{prop_Killing},  $M$ admits
$\varphi_i$-compatible connections for all $i=1,2,3$. We denote by  $\nabla^i$ the
$\varphi_i$-compatible connection with torsion $T_i$ such that
\[
T_i(\xi_1,\xi_2,\xi_3)=2(\beta-\delta),
\]
where $\beta$ is the Reeb Killing function. Then, owing to condition 3) of Definition
\ref{df.can-3-contact} and
equations \eqref{TNS}, \eqref{T1}, \eqref{T2} for the torsion of a $\varphi$-compatible
connection, we have $T_1=T_2=T_3$, hence the three
connections coincide. We denote by $\nabla$ this unique connection and we prove that it
satisfies \eqref{canonical} with $\tilde\beta=\beta$. From Proposition \ref{prop_Killing}, applying
\eqref{nabla_xi} with $\gamma=2(\beta-\delta)$, we have
\[
\nabla_X\xi_i=\beta(\eta_k(X)\xi_j-\eta_j(X)\xi_k)
\]
for every $X\in{\frak X}(M)$ and for every even permutation $(i,j,k)$ of $(1,2,3)$.
Using the above equation, one can check that
\begin{equation}\label{partial1}
(\nabla_X\varphi_i)\xi_h=\beta(\eta_k(X)\varphi_j\xi_h-\eta_j(X)\varphi_k\xi_h)
\end{equation}
for every $h=1,2,3$. Indeed, for $h=i$, we obtain
\bdm
(\nabla_X\varphi_i)\xi_i\ =\ -\varphi_i(\nabla_X\xi_i)
\ =\ -\beta(\eta_k(X)\xi_k+\eta_j(X)\xi_j)
\ =\ \beta(\eta_k(X)\varphi_j\xi_i-\eta_j(X)\varphi_k\xi_i).
\edm
Analogously, one verifies \eqref{partial1} for $h=j,k$.
From equation \eqref{nabla_xi_phi} and the fact that $M$ admits the Reeb Killing function $\beta$, for every $Y,Z\in\Gamma(\H)$ we have
\begin{align*}
g((\nabla_{\xi_i}\varphi_i)Y,Z)&=A_i(Y,Z)=0,\\
g((\nabla_{\xi_j}\varphi_i)Y,Z)&=A_{ij}(Y,Z)=\beta \Phi_k(Y,Z)=-\beta g(\varphi_kY,Z),\\
g((\nabla_{\xi_k}\varphi_i)Y,Z)&=A_{ik}(Y,Z)=-\beta \Phi_j(Y,Z)=\beta g(\varphi_jY,Z),
\end{align*}
where $(i,j,k)$ is an even permutation of $(1,2,3)$. Then, since $\nabla$ preserves
the splitting of the tangent bundle, we have
\begin{equation}\label{partial2}
(\nabla_{\xi_i}\varphi_i)Y=0,\quad (\nabla_{\xi_j}\varphi_i)Y= -\beta \varphi_kY,
\quad (\nabla_{\xi_k}\varphi_i)Y= \beta \varphi_jY.
\end{equation}
On the other hand, since $\nabla$ is $\varphi_i$-compatible, we have
\begin{equation}\label{partial3}
(\nabla_X\varphi_i)Y=0
\end{equation}
for all $X,Y\in\Gamma(\H)$. Then, taking into account \eqref{partial1},
\eqref{partial2}, and \eqref{partial3}, equation \eqref{canonical} is verified.
\end{proof}
\begin{df}
%-----------
If $(M,\varphi_i,\xi_i,\eta_i, g)$ is a canonical almost $3$-contact metric manifold,
the connection $\nabla$ satisfying \eqref{canonical} of Theorem \ref{theo_canonical}
will be called the \emph{canonical connection of $M$}.
\end{df}
\begin{rem}\label{derivatives}
%------------------------------
If $(M,\varphi_i,\xi_i,\eta_i, g)$ is a canonical almost $3$-contact metric manifold
with canonical connection $\nabla$, the covariant derivatives of the structure
tensors are completely determined by the Reeb Killing function $\beta$,
\begin{align*}
\nabla_X\varphi_i&=\beta(\eta_k(X)\varphi _j -\eta_j(X)\varphi _k),\\
\nabla_X\xi_i&=\beta(\eta_k(X)\xi_j-\eta_j(X)\xi_k),\\
\nabla_X\eta_i&=\beta(\eta_k(X)\eta_j-\eta_j(X)\eta_k),
\end{align*}
where $(i,j,k)$ is an even permutation of $(1,2,3)$. The first identity holds by definition,
whereas the second and third identity follow from Proposition \ref{prop_Killing}.
In particular, we observe that each structure $(\varphi_i,\xi_i,\eta_i, g)$ is parallel
along $\H$ and its Reeb vector field $\xi_i$. In analogy to quaternion-K\"ahler geometry,
we may consider the \emph{fundamental $4$-form} $\Psi$
\bdm
\Psi\ :=\ \Phi_1\wedge \Phi_1 + \Phi_2\wedge \Phi_2 + \Phi_3\wedge \Phi_3 .
\edm
It is independent of choice of basis and $\Psi^n\neq 0$. The last formulas imply immediately:
\end{rem}
\begin{cor}
%-----------
The canonical connection of a canonical almost $3$-contact metric manifold
$(M^{4n+3},\varphi_i,\xi_i,\eta_i, g)$ leaves the associated bundle of endomorphisms $\Upsilon_M$
invariant and it satisfies
\bdm
\nabla\Psi =0,\quad \nabla \eta_{123} =0.
\edm
In particular, its holonomy algebra $\mathfrak{hol}(\nabla)$ is contained in
$(\mathfrak{sp}(n)\oplus\mathfrak{sp}(1) )\oplus \so(3)\subset \so(4n)\oplus \so(3)$.
\end{cor}
By Remark \ref{rem.torsion-compatible}, we know that $\eta_{123}$ is one
summand of the torsion of the canonical connection (and in fact of any $\varphi$-compatible
connection). In general, however, the torsion will not be parallel. In Theorem
\ref{thm.canonical-3AD-Sasaki}, we shall prove that it is parallel for
$3$-$(\alpha,\delta)$-Sasaki manifolds.

Finally, let us look at the special case of vanishing Reeb Killing function $\beta$.
The canonical connection satisfies then $\nabla \varphi_i=\nabla\xi_i=0 \ \forall \ i$,
i.\,e.~all structure tensors are parallel;
by uniqueness of the characteristic connection, we conclude:
\begin{cor}
%-----------
The canonical connection $\nabla$ of a parallel canonical almost $3$-contact metric manifold
$(M,\varphi_i,\xi_i,\eta_i, g)$ coincides with the characteristic connection
$\nabla^i$ of each of its almost contact metric
structures $(\varphi_i,\xi_i,\eta_i, g), \ i=1,2,3.$ Furthermore, its holonomy algebra
satisfies  $\mathfrak{hol}(\nabla)\subset \mathfrak{sp}(n)$.
\end{cor}
In general, it is not possible to derive a simple formula for the canonical torsion of a
canonical almost $3$-contact metric manifold. Here are some exceptional cases:
\begin{ex}
%-----------
We already know from Corollary \ref{cor.3-d-sympl-is-pc} that any
$3$-$\delta$-cosymplectic manifold %$(M,\varphi_i,\xi_i,\eta_i,g)$
is a  parallel canonical almost $3$-contact metric manifold.
Therefore, the identities \eqref{TNS}-\eqref{T2} and \eqref{T_canonical} imply directly
that the torsion of the canonical connection is given by
\[
T=-2\delta\,\eta_{123}.
\]
Of course, this is not surprising if we recall that $M$ is locally isometric
to the product of a hyper-K\"ahler manifold with either
 a $3$-dimensional flat group ($\delta=0$, \cite{CM-DN})
or a $3$-dimensional sphere ($\delta\neq 0$, Proposition \ref{prop.3-d-cosymplectic-is-hn}).
\end{ex}
\begin{ex}
%-----------
Similarly, one shows that for the product of a HKT-manifold $M$ with a certain Lie group $G$
considered in  Example \ref{ex.from-HKT}, the torsion $T$ of the canonical
connection $\nabla$ is given by
\[
T\, =\, T_0 -2\delta\, \eta_{123},
\]
where we extend the $3$-form $T_0$ on the product $M\times G$ in such a way that
$\xi_i\lrcorner\, T_0=0$.
\end{ex}
A well-celebrated theorem of Cartan and Schouten states that the only manifolds carrying
a flat metric connection with torsion are compact Lie groups and the $7$-sphere 
(\cite{Cartan&Sch26b}; see \cite{DAtri&N68,Wolf72a,Wolf72b} for proofs and \cite{Agricola&F10b} 
for a classification-free short proof).
The following example shows that this flat connection is in fact the canonical connection
of a natural almost $3$-contact metric structure on $S^7$. We follow the notations used
 in \cite{Agricola&F10b}, hence we shall be brief.
\begin{ex}[$S^7$ as a non-hypernormal parallel canonical almost $3$-contact
metric manifold]\label{ex.S7}
%---------------------------------------------------------------------------------------
In dimension $7$, the complex $\Spin(7)$-representation $\Delta^\C_7$
is the complexification of a real $8$-dimensional representation
$\kappa: \, \Spin(7)\ra \End(\Delta_7)$,
since the real Clifford algebra $\mathcal{C}(7)$ is isomorphic to
$\M(8)\oplus \M(8)$. Thus, we may identify  $\R^8$ with the vector space
$\Delta_7$ and embed therein the sphere $S^7$ as the set of all (algebraic)
spinors of length one, equipped with the induced metric $g$.
Fix your favourite explicit realisation of the
spin representation by skew matrices,
$\kappa_i:=\kappa(e_i)\in\so(8)\subset \End(\R^8)$, $i=1,\ldots,7$.
Define  vector fields $e_1,\ldots,e_7$  on $S^7$ by
\bdm
e_i(x)\ =\ \kappa_i \cdot x \text{ for }x\in S^7\subset \Delta_7.
\edm
From the antisymmetry of $\kappa_1,\ldots,\kappa_7$, one deduces that they form an
orthonormal global frame of $TS^7$.  Hence, they constitute an explicit parallelisation
of $S^7$ by Killing vector fields.
We define an almost  $3$-contact metric structure by setting
$\xi_i:=e_i\ (i=1,2,3)$, $\V=\langle\xi_1,\xi_2,\xi_3\rangle,\ \H=\langle e_4,\ldots,e_7\rangle$,
and
\bdm
\Phi_1 = - (e_{23}+ e_{45}+e_{67}),\quad
\Phi_2 = e_{13}- e_{46}+e_{57},\quad
\Phi_3 = - (e_{12}+ e_{47}+e_{56}).
\edm
Furthermore, we define functions on $S^7$ by
$\alpha_{ijk}(x):=- g(\kappa_i\kappa_j\kappa_k x,x)$: they are quadratic functions in
the coordinates $x_i$ and hence never constant (for $i,j,k$ all different) and the
properties of Clifford multiplication imply that they are totally skew-symmetric in all
indices. The commutator of vector fields is inherited from the
ambient space, hence  $[e_i(x),e_j(x)]= 2\kappa_i\kappa_j x$
for $i\neq j$. This implies
\bdm
[e_i(x),e_j(x)]= 2 \sum_{k=1}^7 \alpha_{ijk}(x)\, e_k(x) \quad \forall i,j=1,\ldots, 7,
\edm
and hence $\delta(x):=\alpha_{123}(x)$ is the Reeb commutator function of the
almost  $3$-contact metric structure.  One further checks that all $N_{\varphi_i}$ are
skew-symmetric but non-trivial. Hence, the structure is not
hypernormal, but each $(\xi_i, \Phi_i,\eta_i)$ admits a characteristic connection.
We now define a connection $\nabla$ on $TS^7$ by $\nabla e_i(x)=0 \ \forall i$;
observe that this implies that all tensor fields with constant coefficients like the
$\Phi_i$'s are parallel as well. In particular, by its uniqueness, $\nabla$
has to coincide with the
characteristic connection of all three almost contact structures. This connection  is
trivially flat and metric and just the one claimed to exist by the Cartan-Schouten
result. By Theorem \ref{theo_canonical}, it coincides with the canonical connection
of $(S^7,\xi_i, \Phi_i,\eta_i,g)$ and has vanishing Reeb Killing function $\beta$.
Alltogether, we conclude that $(S^7,\xi_i, \Phi_i,\eta_i,g)$ is a non-hypernormal
parallel canonical
almost $3$-contact metric manifold, as claimed. Just as an additional piece of information,
let us observe that it is proved in \cite{Agricola&F10b} that $\nabla$ does not have
parallel torsion (in fact, $T(e_i,e_j,e_k)=-g([e_i,e_j], e_k)$) and it is a
characteristic $G_2$-connection of Fernandez-Gray class
$\mathfrak{X}_1\oplus \mathfrak{X}_3\oplus\mathfrak{X}_4$.
\end{ex}
For general $\beta$, the difference $\nabla -\nabla^i$ is computed in the next theorem.
%
%------------------------------------------------------------------------------------------
\subsection{Properties of the canonical connection}
%------------------------------------------------------------------------------------------
%
By Theorem \ref{theo_canonical-implies-char}, the three almost contact metric structures
$(\varphi_i,\xi_i,\eta_i, g)$ of a canonical  almost $3$-contact metric manifold
admit characteristic connections $\nabla^i, \ i=1,2,3$. As a first result, we compare them to
the canonical connection.
\begin{theo}\label{theo_canonical--char-torsion}
%----------------------------------------------
Let  $(M,\varphi_i,\xi_i,\eta_i, g)$ be a canonical  almost $3$-contact metric manifold,
$\nabla$ its canonical connection, $\beta$ its Reeb Killing function, and
$\nabla^i$ the characteristic connections of the three  almost contact metric structures
$(\varphi_i,\xi_i,\eta_i, g)$. The connections $\nabla$ and $\nabla^i$ are related by
\begin{equation}\label{connections}
\nabla=\nabla^i-\frac{\beta}{2}(\eta_j\wedge\Phi_j+\eta_k\wedge\Phi_k)
\end{equation}
for every even permutation $(i,j,k)$ of $(1,2,3)$.
\end{theo}
\begin{proof}
%----------------
By  Theorem \ref{theo-contact}, the torsion $T_i$ of $\nabla^i$ is given by
\begin{equation}
\label{Ti}T_i=\eta_i\wedge d\eta_i+N_{\varphi_i}+d^{\varphi_i}\Phi_i
-\eta_i\wedge (\xi_i\lrcorner N_{\varphi_i}).
\end{equation}
We show that the torsions $T$ and $T_i$ of $\nabla$ and $\nabla^i$ are related by
\begin{equation}\label{T-T_i}
T-T_i=-\beta(\eta_j\wedge\Phi_j+\eta_k\wedge\Phi_k),
\end{equation}
where $(i,j,k)$ is an even permutation of $(1,2,3)$. We will proceed case by case,
computing the difference $T-T_i$ on horizontal and vertical vector fields. First of all,
we deduce from \eqref{Ti} and the general expressions \eqref{TNS}, \eqref{T1} for the
torsion of any $\varphi$-compatible connection
\begin{align*}
T(X,Y,Z)&=T_i(X,Y,Z)=N_{\varphi_i}(X,Y,Z)-d\Phi_i(\varphi_iX,\varphi_iY,\varphi_iZ),\\
T(X,Y,\xi_i)&=T_i(X,Y,\xi_i)=d\eta_i(X,Y)
\end{align*}
for every $X,Y,Z\in\Gamma(\H)$. Now, fixing an even permutation $(i,j,k)$ of $(1,2,3)$
and using ${\mathcal L}_{\xi_k}g=0$, for every $X,Y\in\Gamma(\H)$ we have
\begin{align*}
d\Phi_i(X,Y,\xi_k)
&=\xi_k(\Phi_i(X,Y))-\Phi_i([X,Y],\xi_k)
-\Phi_i([Y,\xi_k],X)-\Phi_i([\xi_k,X],Y)\\
&=-\xi_k(g(\varphi_i X,Y))+g([X,Y],\xi_j)+g([\xi_k,Y],\varphi_iX)
-g([\xi_k,X],\varphi_iY)\\
&=-g([\xi_k,\varphi_iX],Y)-g([\xi_k,X],\varphi_iY)+\eta_j([X,Y])\\
&=-g(({\mathcal L}_{\xi_k}\varphi_i)X,Y)-d\eta_j(X,Y).%\label{dPhi_i_1}
\end{align*}
One can verify that $g(({\mathcal L}_{\xi_k}\varphi_i)\varphi_iX,\varphi_iY)
=-g(({\mathcal L}_{\xi_k}\varphi_i)X,Y)$ so that
\begin{equation}\label{dPhi}
d\Phi_i(\varphi_iX,\varphi_iY,\xi_k)=g(({\mathcal L}_{\xi_k}\varphi_i)X,Y)
-d\eta_j(\varphi_iX,\varphi_iY).
\end{equation}
Using the expression \eqref{Ti} for the torsion $T_i$, and applying equations \eqref{NXYxi}
and \eqref{dPhi}, we have
\begin{align*}
T_i(X,Y,\xi_j)&=N_{\varphi_i}(X,Y,\xi_j)-d\Phi_i(\varphi_iX,\varphi_iY,\xi_k)\\
&=d\eta_j(X,Y)-d\eta_k(\varphi_iX,Y)-d\eta_k(X,\varphi_iY)-g(({\mathcal L}_{\xi_k}\varphi_i)X,Y).
\end{align*}
On the other hand, the torsion $T$ satisfies $T(X,Y,\xi_j)=d\eta_j(X,Y)$, and we have
\[
T(X,Y,\xi_j)-T_i(X,Y,\xi_j)=A_{ik}(X,Y)=-\beta\Phi_j(X,Y),
\]
where we used the fact that $\beta$ is a Reeb Killing function.
Analogously, one can check that $T(X,Y,\xi_k)-T_i(X,Y,\xi_k)=-\beta\Phi_k(X,Y)$,
coherently with \eqref{T-T_i}.

Now, using again \eqref{Ti}, we have
$$T_i(X,\xi_i,\xi_j)=d\eta_i(\xi_j,X)=-g([\xi_j,X],\xi_i).$$
On the other hand, by \eqref{T2}, $T(X,\xi_i,\xi_j)=-g([\xi_i,\xi_j],X)$. Hence,
\[
T(X,\xi_i,\xi_j)-T_i(X,\xi_i,\xi_j)=-({\mathcal L}_{\xi_j}g)(\xi_i,X)=0.
\]
In the same way one shows that $T(X,\xi_i,\xi_k)-T_i(X,\xi_i,\xi_k)=0$, and these relations
are in accordance with \eqref{T-T_i}, since for example
\[
-\beta(\eta_j\wedge\Phi_j+\eta_k\wedge\Phi_k)(X,\xi_i,\xi_j)=-\beta\Phi_j(X,\xi_i)=0.
\]
We shall compute now the difference $T-T_i$ on vector fields $X,\xi_j,\xi_k$, with
$X\in\Gamma(\H)$. First, from $\eta_r([X,\xi_r])=0$, we have
\begin{align}
d\Phi_i(X,\xi_j,\xi_k)\nonumber&=X(\Phi_i(\xi_j,\xi_k))-\Phi_i([X,\xi_j],\xi_k)
-\Phi_i([\xi_j,\xi_k],X)-\Phi_i([\xi_k,X],\xi_j)\nonumber\\
&=-X(g(\xi_j,\xi_j))+\eta_j([X,\xi_j])-g([\xi_j,\xi_k],\varphi_iX)-\eta_k([\xi_k,X])\nonumber\\
&=-g([\xi_j,\xi_k],\varphi_iX).\label{dPhi_i_2}
\end{align}
From \eqref{Ti} we have
\[
T_i(X,\xi_j,\xi_k)=N_{\varphi_i}(X,\xi_j,\xi_k)+d\Phi_i(\varphi_iX,\xi_k,\xi_j)
=-g([\xi_j,\xi_k],X),
\]
where we used \eqref{dPhi_i_2} and the fact that $N_{\varphi_i}(\xi_j,\xi_k)=0$
(see \eqref{tableN}).
Therefore, by \eqref{T2}, $T_i(X,\xi_j,\xi_k)=T(X,\xi_j,\xi_k)$, again coherently with
\eqref{T-T_i}. Finally, by \eqref{T_canonical} and \eqref{Ti}, we have
\[
T(\xi_i,\xi_j,\xi_k)-T_i(\xi_i,\xi_j,\xi_k)=2\beta-2\delta-d\eta_i(\xi_j,\xi_k)
=2\beta-2\delta+2\delta=2\beta,
\]
and one can easily check that
\[
-\beta(\eta_j\wedge\Phi_j+\eta_k\wedge\Phi_k)(\xi_i,\xi_j,\xi_k)
=-\beta(\Phi_j(\xi_k,\xi_i)+\Phi_k(\xi_i,\xi_j))=2\beta,
\]
which completes the proof of \eqref{T-T_i}. Therefore, by \eqref{nabla} and \eqref{T-T_i},
we get \eqref{connections}.
\end{proof}
\begin{rem}
%-------------
Under the hypotheses Theorem \ref{theo_canonical--char-torsion}, equation \eqref{T-T_i}
implies that the torsion $T$ of the canonical connection and the torsions
$T_1$, $T_2$, $T_3$ of the three characteristic connections satisfy
\[
3T\ =\ T_1+T_2+T_3-2\beta(\eta_1\wedge\Phi_1+\eta_2\wedge\Phi_2+\eta_3\wedge\Phi_3).
\]
\end{rem}
We shall show now that $3$-$(\alpha,\delta)$-Sasaki manifolds are the only canonical
horizontal $3$-$\alpha$-contact metric manifolds with integrable distribution $\mathcal V$.
By a result of B.~Cappelletti-Montano, we have the following
\begin{prop}[{\cite[Prop.\,3.2]{C-M}}]\label{Vintegrable}
%--------------------------------------------------------
Let $(M,\varphi_i,\xi_i,\eta_i, g)$ be an almost $3$-contact metric manifold such that
the distribution $\mathcal V$ is integrable and each $\xi_i$ is a Killing vector field.
Then, the following properties hold:
\begin{enumerate}[\normalfont i)]
\item $[\xi_i,\xi_j]=2\delta\xi_k$ for every even permutation $(i,j,k)$ of $(1,2,3)$,
and some constant $\delta$;
\item each $\xi_i$  is an infinitesimal automorphism of the horizontal distribution $\H$,
i\,.e. $[\xi_i, X]\in\Gamma(\H)$ for every $X\in\Gamma(\H)$;
\item the distribution $\V$ has totally geodesic leaves.
\end{enumerate}
\end{prop}
Observe that under the hypothesis that $\mathcal V$ is integrable, condition i) is
equivalent to the existence of a \emph{constant}
Reeb commutator function, since the projection of the commutator to $\H$ vanishes.
\begin{theo}%\label{parallel_torsion}
%------------------------------------------------------------------
Let $(M,\varphi_i,\xi_i,\eta_i, g)$ be a canonical almost $3$-contact metric manifold
with canonical connection $\nabla$ and Reeb Killing function $\beta$.
Assume that the following conditions hold:
\begin{enumerate}[\normalfont i)]
\item the distribution $\mathcal V$ is integrable;
\item $d\eta_i(X,Y)=2\alpha \Phi_i(X,Y)$ for every $X,Y\in\Gamma(\H)$ and $i=1,2,3$,
and for some real constant $\alpha\neq 0$.
\end{enumerate}
Then the structure admits a constant Reeb commutator $\delta$ and
$(M,\varphi_i,\xi_i,\eta_i,g)$  is a $3$-$(\alpha,\delta)$-Sasaki manifold.
\end{theo}
\begin{proof}
Since the distribution $\mathcal V$ is integrable and the Reeb vector fields are Killing,
from Proposition \ref{Vintegrable}, each $\xi_i$  is an infinitesimal automorphism of the
horizontal distribution $\H$, and thus
\[d\eta_r(X,\xi_s)=0,\qquad \forall X\in\Gamma(\H),\; r,s=1,2,3.\]
Furthermore, $[\xi_i,\xi_j]=2\delta \xi_k$ for every even permutation $(i,j,k)$ of $(1,2,3)$
and some constant $\delta$. Therefore,
\[d\eta_r(\xi_s,\xi_t)=-2\delta\epsilon_{rst}.\]
Taking into account condition ii) we deduce that the differential of each $1$-form $\eta_i$
is given by
\[
d\eta_i=2\alpha\Phi_i+2(\alpha-\delta)\eta_j\wedge\eta_k
\]
where $(i,j,k)$ is an even permutation of $(1,2,3)$.
\end{proof}
%
%-----------------------------------------------------------------------------------------
\subsection{The cone of a canonical almost $3$-contact metric manifold}\label{sec.cone}
%-----------------------------------------------------------------------------------------
%
In \cite{Ag-H} the authors studied cones of $G$ manifolds endowed with a characteristic
connection. Given  a Riemannian manifold $(M,g)$  equipped with a metric connection
$\nabla$ with skew-symmetric torsion $T$, the cone
$(\bar M, \bar g)=(M\times\mathbb{R}^+,a^2r^2g+dr^2)$, $a>0$, can be endowed with an
appendant connection $\bar\nabla:=\nabla^{\bar g}+\frac12\bar T$, where $\bar T$ is the
skew-symmetric torsion of $\bar\nabla$, defined by
\[
\bar T(X,Y)=T(X,Y) \mbox{ for } X,Y\perp \partial_r,\quad\partial_r\lrcorner\bar T=0.
\]
The positive real number $a$ will be called the \emph{cone constant}.

Now, let  $(M,\varphi_i,\xi_i,\eta_i, g)$ be an almost $3$-contact metric manifold. On the
cone $(\bar M, \bar g)$ one can consider  three almost hermitian structures defined by
\begin{align}\label{J}
J_1(ar\partial_r)&=\xi_1, &J_1(\xi_1)&=-ar\partial_r,\; &J_1(V)
&=-\varphi_1(V) \mbox{ for } V\perp\xi_1,\partial_r,\nonumber\\
J_2(ar\partial_r)&=\xi_2, &J_2(\xi_2)&=-ar\partial_r,\; &J_2(V)
&=-\varphi_2(V) \mbox{ for } V\perp\xi_2,\partial_r,\\
J_3(ar\partial_r)&=-\xi_3, &J_3(\xi_3)&=ar\partial_r,\; &J_3(V)
&=-\varphi_3(V) \mbox{ for } V\perp\xi_3,\partial_r,\nonumber
\end{align}
These structures satisfy $J_1J_2=J_3=-J_2J_1$, and hence $(\bar M, \bar g, J_1,J_2,J_3)$
is an almost hyperhermitian manifold. We will use the following result.
\begin{theo}[\cite{Ag-H}]\label{theo_cone}
%------------------------------------------
Let  $(M,\varphi_i,\xi_i,\eta_i, g)$ be an almost $3$-contact metric manifold such that
each structure $(\varphi_i,\xi_i,\eta_i, g)$ admits a  characteristic connection
$\nabla^{i}$ with skew torsion $T_i$. Let $\nabla$ be a metric connection with totally
skew-symmetric torsion on $M$. Then the appendant connection $\bar\nabla$ satisfies
$\bar\nabla J_1=\bar\nabla J_2=\bar\nabla J_3=0$ if and only if there exists some positive
constant $a$ (the cone constant) such that the three tensors
$S_i:=T_i-2a\eta_i\wedge\Phi_i$ coincide with the torsion $T$ of $\nabla$.
Furthermore, if $M$ is hypernormal, then $\bar M$ is an HKT manifold.
\end{theo}
Let us point out that in the preceding result, the non-existence of a
characteristic connection for almost $3$-contact metric manifolds was circumvented
by requiring the property that the three difference tensors $S_i$ should
coincide---one then views them as the torsion of a connection and lifts it to the cone.

\begin{cor}\label{cor_canonical2}
%----------------------------------
Let  $(M,\varphi_i,\xi_i,\eta_i, g)$ be a canonical  almost $3$-contact metric manifold,
$\nabla$ its canonical connection. Assume that  its Reeb Killing function $\beta$ is constant
and negative.
 If $\nabla'$ is the metric connection on $M$ with skew torsion $T'$ given by
\bdm
T':=T+\beta(\eta_1\wedge\Phi_1+\eta_2\wedge\Phi_2+\eta_3\wedge\Phi_3),
\edm
the appendant connection $\bar{\nabla'}$ on the cone $(\bar M,\bar g)$, with cone
constant  $a=-\frac{\beta}{2}$ is a hermitian connection, i.\,e.~it parallelizes
the almost hermitian structures $J_i$, $i=1,2,3$, defined by \eqref{J}.

If, furthermore, $(M,\varphi_i,\xi_i,\eta_i, g)$, is a $3$-$(\alpha,\delta)$-Sasaki manifold,
then the cone $(\bar M,\bar g)$ is an HKT manifold.
\end{cor}
\begin{proof}
%----------------
From Theorem \ref{theo_canonical-implies-char}, we know that each
almost contact metric structure $(\varphi_i,\xi_i,\eta_i, g)$ admits a
characteristic connection $\nabla^i$; Theorem \ref{theo_canonical--char-torsion} shows that the connections $\nabla$ and $\nabla_i$ are related by \eqref{connections}.
%( $(i,j,k)$ an of even permutation $(1,2,3)$)
%%
%\bdm
%\nabla=\nabla^i-\frac{\beta}{2}(\eta_j\wedge\Phi_j+\eta_k\wedge\Phi_k)
%\edm
%%
Now, taking $S_i:=T_i+\beta\eta_i\wedge\Phi_i$, by \eqref{T-T_i} we get
\[
S_i=T+\beta(\eta_1\wedge\Phi_1+\eta_2\wedge\Phi_2+\eta_3\wedge\Phi_3),
\]
so that the three tensors $S_i$, $i=1,2,3$, coincide. We can thus apply Theorem
\ref{theo_cone}. Consider the cone $(\bar M,\bar g)$ corresponding to the cone constant
$a:=-\frac{\beta}{2}$. If $\nabla'$ is the metric connection on $M$ with totally
skew-symmetric torsion $T':=S_1=S_2=S_3$, the appendant connection $\bar {\nabla'}$
parallelizes the almost hermitian structures $J_i$, $i=1,2,3$.

Theorem \ref{thm.general-Kashiwada} states that any $3$-$(\alpha,\delta)$-Sasaki manifold is
hypernormal, hence the last claim follows from the corresponding statement in
Theorem \ref{theo_cone}.
\end{proof}
\begin{rem}
%----------
Recall  that for a $3$-$(\alpha,\delta)$-Sasaki manifold, the Reeb Killing function $\beta$
is  automatically constant and may be computed from $\alpha$ and $\delta$ through
$\beta= 2\delta- 4\alpha$. Thus, the condition $\beta<0$ in Corollary \ref{cor_canonical2}
may be restated as $2\alpha>\delta$ in this situation.
\end{rem}
\begin{rem}
%----------
For example, we know that any  $3$-Sasakian manifold is a $3$-$(\alpha,\delta)$-Sasaki manifold
with $\beta=-2<0$. Comparing the expression for $T'$ above with the results
of \cite[Section 3.5]{Ag-H}, one sees that $\bar {\nabla'}$ will then coincide with the
Levi-Civita connection of the natural hyper-K\"ahler structure on the cone $\bar M$.
Similarly, the quaternionic Heisenberg group is a  $3$-$(\alpha,\delta)$-Sasaki manifold
with $\beta=-2\lambda, \ \lambda $ a positive non-zero parameter.
In \cite[Thm 11]{Ag-F-S} it was shown (by applying Theorem \ref{theo_cone}) that the
cone of the $7$-dimensional quaternionic Heisenberg group is a HKT manifold.
Hence, Corollary \ref{cor_canonical2} generalizes both results.
\end{rem}
%
%
%-----------------------------------------------------------------------------------------
\subsection{The canonical connection of a $3$-$(\alpha,\delta)$-Sasaki manifold}
%-----------------------------------------------------------------------------------------
%
We now look in detail at the canonical connection of a  $3$-$(\alpha,\delta)$-Sasaki manifold.
Recall that such a manifold is always a canonical almost $3$-contact metric manifold
(Corollary \ref{cor.3-AD-Sasaki-implies-canonical}), and hence the existence (and uniqueness)
of the  canonical connection is garanteed.
\begin{rem}[$\nabla$ as a qc connection]\label{rem.can-conn-and-qc-Einstein}
%---------------------------------------------------------------------------
By Proposition \ref{prop.3-AD-Sasaki-is-qc-Einstein}, we  know  that
every $3$-$(\alpha,\delta)$-Sasaki manifold $(M,\varphi_i,\xi_i,\eta_i,g)$
admits an underlying quaternionic contact structure which is qc-Einstein with $S=2\alpha\delta$,
with almost complex structures $I_i:=\varphi_i|_{\H}$ and $1$-forms
$\tilde\eta_i :=-\frac{1}{\alpha}\eta_i$. In general, the condition for a metric connection
$\nabla$ to preserve the qc structure reduces to the requirement that $\nabla$ preserves
the splitting $TM=\H\oplus \V$ and has the additional properties
\bdm
\nabla ( I_1\otimes I_1+ I_2\otimes I_2 +I_3\otimes I_3)=0, \quad
\nabla (\tilde\xi_1\otimes I_1 +\tilde\xi_2\otimes I_2 + \tilde\xi_3\otimes I_3)=0.
\edm
The equations of Remark \ref{derivatives} imply that the canonical connection of a
$3$-$(\alpha,\delta)$-Sasaki manifold is indeed a qc connection (this was already observed for
the quaternionic Heisenberg group in \cite{Ag-F-S}). The most commonly used such
connection is the well-known Biquard connection.
\end{rem}
\begin{rem}\label{rem.horizontal-Phi}
%------------------------------------
The relations ($(i,j,k)$ an even permutation of $(1,2,3)$)
\[
\xi_i\lrcorner\,\Phi_i=0,\qquad \xi_j\lrcorner\,\Phi_i
=-\eta_k,\qquad \xi_k\lrcorner\,\Phi_i=\eta_j
\]
holding for any  $3$-$(\alpha,\delta)$-Sasaki manifold imply that we can split
the $2$-forms $\Phi_i$ in their vertical and horizontal part,
\bdm
\Phi_1\ =\ -\eta_{23} + \Phi_1^\H,\quad
\Phi_2\ =\ \eta_{13} + \Phi_2^\H,\quad
\Phi_3\ =\ -\eta_{12} + \Phi_3^\H, \quad \Phi_i^\H\in\Lambda^2(\H) \text{ for } i=1,2,3,
\edm
which we can alternatively summarize as $\Phi_i = -\eta_{jk} + \Phi_i^\H$ for even permutations.
Furthermore, the defining condition of a $3$-$(\alpha,\delta)$-Sasaki manifold may be
reformulated as
\bdm
d\eta_i \ =\ 2\alpha\, \Phi^\H_i - 2\delta\, \eta_{jk}
\edm
for even permutations. All in all, this distinction between horizontal and vertical
contributions allows to identify the horizontal, vertical, and mixed parts of the
torsion and its derivative more clearly.
\end{rem}
\begin{theo}\label{thm.canonical-3AD-Sasaki}
%--------------------------------------------
Let $(M,\varphi_i,\xi_i,\eta_i,g)$ be a $3$-$(\alpha,\delta)$-Sasaki manifold.
The torsion of its canonical connection $\nabla$ is given by
\bdm
T\ =\ \sum_{i=1}^3\eta_i\wedge d\eta_i+8(\delta-\alpha)\,\eta_{123}\ =\
2\alpha \sum_{i=1}^3\eta_i\wedge \Phi^\H_i+2(\delta-4 \alpha)\,\eta_{123}
\edm
and satisfies $\nabla T=0$ as well as
\begin{eqnarray*}
dT & = & 4\alpha^2\sum_{i=1}^3\Phi_i\wedge\Phi_i + 8\alpha(\delta-\alpha)
\cyclic{i,j,k}\Phi_i\wedge\eta_{jk}\\
& = &  4\alpha^2\sum_{i=1}^3\Phi^\H_i\wedge\Phi^\H_i +
8\alpha(\delta-2 \alpha) \cyclic{i,j,k}\Phi^\H_i\wedge\eta_{jk}.
\end{eqnarray*}
Here, the symbol $\cyclic{i,j,k}$ means the sum over all even permutations of
$(1,2,3)$.
\end{theo}
\begin{proof}
%---------------
%
Each almost contact metric structure $(\varphi_i,\xi_i,\eta_i,g)$ admits a characteristic
connection $\nabla^i$ (Theorem \ref{theo_canonical-implies-char}) with torsion
\[
T_i=\eta_i\wedge d\eta_i+d^{\varphi_i}\Phi_i,
\]
since the structure is hypernormal. Applying \eqref{differential_Phi}, since
$\eta_k\circ\varphi_i=\eta_j$ and $\eta_j\circ\varphi_i=-\eta_k$, and using also
equations \eqref{2-form1} and \eqref{2-form2}, we have
\begin{eqnarray*}
d^{\varphi_i}\Phi_i(X,Y,Z)&=&
2(\alpha-\delta)\{(\eta_k\wedge\Phi_j)(\varphi_iX,\varphi_iY,\varphi_iZ)
-(\eta_{j}\wedge\Phi_k)(\varphi_iX,\varphi_iY,\varphi_iZ)\}\\
&=&2(\alpha-\delta)\{-\eta_j(X)(\Phi_j+\eta_{ki})(Y,Z)\\
&&{}-\eta_j(Y)(\Phi_j+\eta_{ki})(Z,X)-\eta_j(Z)(\Phi_j+\eta_{ki})(X,Y)\}\\
&&{}-2(\alpha-\delta)\{\eta_k(X)(\Phi_k+\eta_{ij})(Y,Z)\\
&&{}+\eta_k(Y)(\Phi_k+\eta_{ij})(Z,X)+\eta_k(Z)(\Phi_k+\eta_{ij})(X,Y)\}\\
&=&2(\alpha-\delta)(-\eta_j\wedge\Phi_j-\eta_{jki})(X,Y,Z)\\
&&{}-2(\alpha-\delta)(\eta_k\wedge\Phi_k+\eta_{kij})(X,Y,Z)\\
&=&2(\delta-\alpha)(\eta_j\wedge\Phi_j+\eta_k\wedge\Phi_k+2\,\eta_{123})(X,Y,Z).
\end{eqnarray*}
Therefore, the torsion $T_i$ is given by
\bdm
T_i=\eta_i\wedge d\eta_i+2(\delta-\alpha)(\eta_j\wedge\Phi_j+\eta_k\wedge\Phi_k+2\,\eta_{123}),
\edm
and by \eqref{T-T_i}, the torsion of the canonical connection is
\begin{eqnarray*}
T&=&T_i-2(\delta-2\alpha)(\eta_j\wedge\Phi_j+\eta_k\wedge\Phi_k)\\
&=&\eta_i\wedge d\eta_i+2\alpha(\eta_j\wedge\Phi_j+\eta_k\wedge\Phi_k)
+4(\delta-\alpha)\eta_{123}\\
&=&\eta_i\wedge d\eta_i+\eta_j\wedge \{d\eta_j+2(\delta-\alpha)\eta_{ki}\}
+\eta_k\wedge \{d\eta_k+2(\delta-\alpha)\eta_{ij}\}+4(\delta-\alpha)\eta_{123}\\
&=&\sum_{i=1}^3\eta_i\wedge d\eta_i+8(\delta-\alpha)\eta_{123}\, .
\end{eqnarray*}
The alternative expression in terms of $\Phi_i^\H$ follows by substituting their definitions
from Remark \ref{rem.horizontal-Phi}.

One can verify that both expressions for $T$ are coherent with
equations \eqref{TNS}, \eqref{T1}, \eqref{T2} and \eqref{T_canonical}. In particular,
\begin{equation}\label{T_3S1}
T(X,Y,Z)=T(X,\xi_i,\xi_j)=0,\end{equation}
\begin{equation}\label{T_3S2}
T(X,Y,\xi_i)=2\alpha\Phi_i(X,Y),\quad T(\xi_i,\xi_j,\xi_k)=2(\beta-\delta)=2(\delta-4\alpha),
\end{equation}
for every $X,Y,Z\in\Gamma(\H)$, and where $(i,j,k)$ is an even permutation of $(1,2,3)$.
 Now, since the canonical connection $\nabla$ preserves the splitting
$TM=\H\oplus\V$, from \eqref{T_3S1} we obtain
\[
(\nabla_U T)(X,Y,Z)=(\nabla_UT)(X,\xi_i,\xi_j)=0,
\]
for every $U\in{\frak X}(M)$ and $X,Y,Z\in\Gamma(\H)$. By Remark \ref{derivatives}, the
covariant derivatives $\nabla\xi_i$ are given by
\begin{equation}\label{canonical_nabla_xi_3S}
 \nabla_U\xi_i=\beta(\eta_k(U)\xi_j-\eta_j(U)\xi_k)
 \end{equation}
for every $U\in{\frak X}(M)$ and where $\beta=2(\delta-2\alpha)$. In
particular, $\nabla_U\xi_i\in\langle\xi_j,\xi_k\rangle$  and thus, using also the second
identity in \eqref{T_3S2}, we get
\[
(\nabla_UT)(\xi_i,\xi_j,\xi_k)=0.
\]
Furthermore, using the first identity in \eqref{T_3S2}, and \eqref{canonical_nabla_xi_3S},
we have
\begin{align*}
&(\nabla_U T)(X,Y,\xi_i)\\&=2\alpha\nabla_U(\Phi_i(X,Y))-2\alpha\Phi_i(\nabla_UX,Y)
-2\alpha\Phi_i(X,\nabla_UY)-T(X,Y,\nabla_U\xi_i)\\
&=2\alpha(\nabla_U\Phi_i)(X,Y)+\beta\eta_k(U)T(X,Y,\xi_j)-\beta\eta_j(U)T(X,Y,\xi_k)\\
&=2\alpha g(X,(\nabla_U\varphi_i)Y)+2\beta\alpha\eta_k(U)\Phi_j(X,Y)
-2\beta\alpha\eta_j(U)\Phi_k(X,Y)\\
&=2\alpha\{g(X,(\nabla_U\varphi_i)Y)+\beta g(X,\eta_k(U)\varphi_jY-\eta_j(U)\varphi_kY)\}
\end{align*}
which vanishes because of \eqref{canonical}. This completes the proof that $\nabla T=0$.
Finally,  differentiating the expression for $T$, we obtain
\begin{align*}
dT&=\sum_{i=1}^3d\eta_i\wedge d\eta_i+8(\delta-\alpha)\cyclic{i,j,k}d\eta_i\wedge\eta_{jk}\\
&=\cyclic{i,j,k}\big(2\alpha\Phi_i+2(\alpha-\delta)\eta_{jk}\big)
\wedge\big(2\alpha\Phi_i+2(\alpha-\delta)\eta_{jk}\big)
+16\alpha(\delta-\alpha)\cyclic{i,j,k}\Phi_i\wedge\eta_{jk}\\
&=4\alpha^2\sum_{i=1}^3\Phi_i\wedge\Phi_i-8\alpha(\alpha-\delta)
\cyclic{i,j,k}\Phi_i\wedge\eta_{jk}.
\end{align*}
This completes the proof.
\end{proof}
\begin{rem}[$\varphi$-compatible connections of  $3$-$(\alpha,\delta)$-Sasaki manifolds]
%--------------------------------------------------------------------------------------
An immediate computation based on Proposition \ref{prop_Killing} and the preceding
Theorem \ref{thm.canonical-3AD-Sasaki} shows that for a $\varphi$-compatible
connection $\nabla^\gamma$ with parameter function $\gamma$ of a $3$-$(\alpha,\delta)$-Sasaki
manifold,  the general expression for the torsion $T_\gamma$ is
\bdm
T_\gamma\ =\ \sum_{i=1}^3\eta_i\wedge d\eta_i+(8\delta-4\alpha+\gamma)\,\eta_{123}\ =\
2\alpha \sum_{i=1}^3\eta_i\wedge \Phi^\H_i+\gamma\,\eta_{123}.
\edm
The canonical connection corresponds to the choice $\gamma=2(\delta-4\alpha)$.
\end{rem}
The following lemma is purely computational, hence we omit the proof.
The formulas are, however, quite useful, for example in the next section.
\begin{lem}\label{lem.Phi-H}
%---------------------------
\begin{enumerate}[\normalfont 1)]\item[]
\item
$d\Phi^\H_i= 2\delta (\Phi^\H_j\wedge \eta_k - \Phi^\H_k\wedge\eta_j)$ \ and \
$\alpha\, d\Phi^\H_i= \delta \, d(\eta_{jk})$ for even
permutations,
\item
The form $\Psi^\H := \sum_{i=1}^3\Phi^\H_i\wedge\Phi^\H_i$ satisfies
$ d\Psi^\H =0$,
\item
$ d \left[\sum_{i=1}^3\eta_i\wedge\Phi^\H_i \right]= 2\alpha\Psi^\H +
2\delta\cyclic{i,j,k}\Phi^\H_i\wedge\eta_{jk} $,
\item
$ d \eta_{123} = 2\alpha \cyclic{i,j,k}\Phi^\H_i\wedge\eta_{jk}$.
\end{enumerate}
\end{lem}
Since any $3$-$\alpha$-Sasakian manifold is $3$-$(\alpha,\delta)$-Sasaki with
$\delta=\alpha$, we have the following
\begin{cor}
%--------------
Any $3$-$\alpha$-Sasakian manifold $(M,\varphi_i,\xi_i,\eta_i,g)$  admits a
canonical connection $\nabla$ with torsion
\begin{equation*}\label{T_3S}
T=\eta_1\wedge d\eta_1+\eta_2\wedge d\eta_2+\eta_3\wedge d\eta_3
\end{equation*}
which satisfies $\nabla T=0$.
\end{cor}
\begin{rem}[Canonical connection of a $3$-Sasakian manifold]
%-----------------------------------------------------------
The canonical connection defined in the above Corollary coincides with the linear connection
considered in \cite{Ag-Fr} for $7$-dimensional $3$-Sasakian manifolds. In this case the
canonical connection $\nabla$ coincides with the characteristic connection of the
canonical (cocalibrated) $G_2$-structure of the $3$-Sasakian manifold. This fact will be
generalized in Section \ref{sec.G2-and-3AD-S}.
\end{rem}
\begin{rem}[Connections on $S^7$]
%--------------------------------
The  $7$-sphere carries a dazzling array of interesting metric connections with skew
torsion---we saw one of them in Example \ref{ex.S7}, and of course
the natural $3$-Sasaki structure on $S^7$ is covered by the previous Corollary and Remark. The
$7$-sphere can be endowed with several natural metrics: the round metric, the family of
Berger metrics, or the naturally reductive metric stemming from the realisation as the
homogeneous space $\Spin(7)/G_2$. A thorough
investigation of metric connections invariant under Lie groups was carried
out in \cite{Draper&G&P16} ($G=\Spin(6)$) and \cite{Chrysikos16} ($G=\Spin(7)$).
A comparison of their results with ours shows that none of these connections is
$\varphi$-compatible for the underlying $3$-Sasaki structure, because they do not preserve
the distributions $\V$ and $\H$
(for details, see \cite[5.13--5.19]{Draper&G&P16} and \cite[Ex. 4.8]{Chrysikos16}).
\end{rem}
For parallel $3$-$(\alpha,\delta)$-Sasaki manifolds, corresponding to $\delta=2\alpha$, we can state the following
\begin{cor}
%------------
The canonical connection of a parallel $3$-$(\alpha,\delta)$-Sasaki manifold $(M,\varphi_i,\xi_i,\eta_i,g)$ has torsion
\begin{equation*}
T=\sum_{i=1}^3\eta_i\wedge d\eta_i+8\alpha\,\eta_{123}
\end{equation*}
and satisfies $\nabla T=0$.
\end{cor}
\begin{rem}
%------------
Notice that any $3$-$\alpha$-Sasakian structure admits  $\H$-homothetic deformations which are
$3$-$(\alpha',\delta')$-Sasaki with $\delta'=2\alpha'$. Indeed, if
$(M,\varphi_i,\xi_i,\eta_i,g)$ is a $3$-$\alpha$-Sasakian manifold, the $\H$-deformed
structure \eqref{deformation} is  $3$-$(\alpha',\delta')$-Sasaki  with
$\alpha'=\alpha\frac{c}{a},\ \delta'=\frac{\alpha}{c}$
by Proposition \ref{deformation-coefficients}.
Therefore, these coefficients satisfy $\delta'=2\alpha'$ if and only if $a=2c^2$. On the
other hand, we know that $c^2=a+b$. Hence, we can conclude that all the deformed structures
\[
\eta'_i=c\eta_i,\qquad \xi'_i=\frac{1}{c}\xi_i,\qquad\varphi'_i
=\varphi_i,\qquad g'=2c^2g-c^2\sum_{i=1}^3\eta_i\otimes\eta_i
\]
are $3$-$(\alpha',\delta')$-Sasaki with $\alpha'=\frac{\alpha}{2c}$ and
$\delta'=\frac{\alpha}{c}=2\alpha'$, each one admitting a canonical connection which
parallelizes the structure tensor fields.
\end{rem}
We showed that the quaternionic Heisenberg group $(N_p,\varphi_i,\xi_i,\eta_i,g_\lambda)$ is
a degenerate $3$-$(\alpha,\delta)$-Sasaki manifold  ($\delta=0$) with $2\alpha=\lambda$.
Therefore,
\begin{cor}
%--------------
The quaternionic Heisenberg group $(N_p,\varphi_i,\xi_i,\eta_i,g_\lambda)$ admits a
canonical connection $\nabla$ with  torsion $T$  given by
\begin{equation*}\label{T_H}
T=\eta_1\wedge d\eta_1+\eta_2\wedge d\eta_2+\eta_3\wedge d\eta_3-
4\lambda\eta_{123}
\end{equation*}
which satisfies $\nabla T=0$.
\end{cor}
\begin{rem}
%-----------
The canonical connection of the quaternionic Heisenberg group determined in the above
corollary coincides with the canonical connection defined in \cite{Ag-F-S}. In
\cite{Ag-F-S} the authors prove that this connection $\nabla$ parallelizes its torsion
and curvature tensors, and the holonomy algebra of $\nabla$ is isomorphic to
$\frak{su}(2)$. In the $7$-dimensional case this connection $\nabla$ is also the
characteristic connection of a cocalibrated $G_2$-structure.
\end{rem}
\begin{rem}
%-----------
In \cite{C-M}, Cappelletti-Montano investigated hypernormal almost $3$-contact metric manifolds
admitting metric connections with certain invariance properties, but only under the
assumption of having skew-symmetric torsion on $\H$. He calls these manifolds \emph{almost
$3$-contact metric manifolds with torsion}. One easily checks with the properties of
$3$-$(\alpha,\delta)$-Sasaki manifolds we compiled that these are always
almost $3$-contact metric manifolds with torsion by \cite[Thm 4.3]{C-M}.
\end{rem}
Finally, we formulate the result for the Ricci curvature $\Ric$ of the canonical connection.
The computation is rather lengthy, but standard, hence we omit it. Observe that
the property of being symmetric follows for $\Ric$ from $\nabla T=0$.
\begin{theo}[$\nabla$-Ricci curvature for $3$-$(\alpha,\delta)$-Sasaki manifolds]
\label{theo.nabla-Ricci}
%----------------------------------------------------------------------------------
Let $(M,\varphi_i,\xi_i,\eta_i,g)$ be a $3$-$(\alpha,\delta)$-Sasaki manifold of dimension
$4n+3$. The
Ricci tensor of the canonical connection $\nabla$ is for all $X,Y\in\frak{X}(M)$ given by
\begin{eqnarray*}
\Ric & =& 4\alpha \{\delta(n+2)-3\alpha \}\, g
+4\alpha \{\delta(2-n)-5\alpha\} \sum_{i=1}^3\eta_i\otimes \eta_i \\
&=& 4\alpha \{\delta(n+2)-3\alpha \}\,\Id_{\H} + 16\alpha(\delta-2\alpha)\, \Id_{\V}.
\end{eqnarray*}
In particular, the manifold is $\nabla$-Einstein if and only if $\delta (2-n) = 5\alpha$, and it is never $\nabla$-Ricci flat.
\end{theo}
For details on the notion of $\nabla$-Einstein manifolds, we refer to \cite{Agricola&Fe14}.
Together with the expression for the Riemannian Ricci curvature stated in
Proposition \ref{prop.3-AD-Sasaki-is-qc-Einstein}, we can conclude after a short calculation:
\begin{cor}
%-----------
A $3$-$(\alpha,\delta)$-Sasaki manifold $(M,\varphi_i,\xi_i,\eta_i,g)$ is both Riemannian Einstein and $\nabla$-Einstein if and only if $\dim M=7$ and $\delta=5\alpha$.
\end{cor}
\begin{ex}[$\H$-Deformations of $3$-Sasaki manifolds II]\label{ex.3-S-homo-II}
%------------------------------------------------------------------------------
Consider now a $7$-dimensional $3$-Sasaki manifold $(M^7,\varphi_i,\xi_i,\eta_i,g)$,
i.\,e.~$\alpha_0=\delta_0=1$, and the particular one-parameter family of $\H$-homothetic
deformations given by
$a>0$  arbitrary, $b=1-a$, $c=1$,  hence  $\alpha=\frac{1}{a}$  and $\delta=1$ for the
resulting $3$-$(\alpha,\delta)$-Sasaki manifold
(Proposition \ref{deformation-coefficients}). From Corollary
\ref{cor.3-AD-Sasaki-implies-canonical}, we conclude that its Reeb Killing function
is $\beta=2(1-2/a)$. This is a well-known family of deformations, see for example
\cite{FKMS, Friedrich07}. We conclude at once that $a$ assumes the following particular
values:
\bdm
\begin{tabular}{|c|c|}\hline
$a$ & properties\\ \hline\hline
1 &  Einstein and  $3$-Sasakian \\ \hline
2 & parallel $3$-$(\alpha,\delta)$-Sasakian  \\ \hline
5 & Einstein and $\nabla$-Einstein $3$-$(\alpha,\delta)$-Sasakian \\ \hline
\end{tabular}
\edm
\end{ex}
%

%-----------------------------------------------------------------------------------------
\subsection{The canonical $G_2$-structure of a $7$-dimensional
$3$-$(\alpha,\delta)$-Sasaki manifold}\label{sec.G2-and-3AD-S}
%-----------------------------------------------------------------------------------------
%
In this Section, we restrict our attention to the situation
that $(M^7,\varphi_i,\xi_i,\eta_i,g)$ is a (simply connected)
$7$-dimensional $3$-$(\alpha,\delta)$-Sasaki manifold.
In the adapted frame $e_1=\xi_1,\ e_2=\xi_2,\ e_3=\xi_3$, $e_4$ any vector field
orthonormal to $\V$, $e_5=\varphi_1 e_4, \ e_6=\varphi_2 e_4, \ e_7=\varphi_3 e_4$ with
dual $1$-forms $\eta_i, \ i=1,\ldots,7,$ the horizontal
fundamental forms are given by
\bdm
\Phi^\H_1\, =\, - \eta_{45}-\eta_{67},\quad
\Phi^\H_2\, =\, -\eta_{46}+\eta_{57},\quad
\Phi^\H_3\, =\, -\eta_{47}-\eta_{56}.
\edm
We shall prove that as in the $7$-dimensional $3$-Sasaki case \cite{Ag-Fr}, its canonical
connection is, in fact, a $G_2$-connection:
\begin{theo}
%-------------
Let $(M^7,\varphi_i,\xi_i,\eta_i,g)$ be a $7$-dimensional $3$-$(\alpha,\delta)$-Sasaki
manifold. The $3$-form
\bdm
\omega\ :=\ \sum_{i=1}^3 \eta_i\wedge \Phi^\H_i + \eta_{123} \ =\
-\eta_{145}-\eta_{167}- \eta_{246}+\eta_{257}-\eta_{347}-\eta_{356}+\eta_{123}
\edm
defines a cocalibrated $G_2$-structure, i.\,e.~it is of Fernandez-Gray type
$W_1\oplus W_3$, that we shall call the \emph{canonical $G_2$-structure}.
Its characteristic connection $\nabla$ coincides with the canonical connection.
\end{theo}
\begin{proof}
%-------------
As this is not a paper on $G_2$-manifolds, we shall be brief, the necessary details on
$G_2$-manifolds and their characteristic connection may for example be  found in
\cite{FrIv,Ag}. For commodity, let us split $\omega$ as
\bdm
\omega\, =\, \omega_1+\omega_2 \ \text{ with }\
\omega_1\, =\, \sum_{i=1}^3 \eta_i\wedge \Phi^\H_i, \quad
\omega_2\, =\, \eta_{123}.
\edm
One checks by an explicit calculation that their Hodge duals are given by
\bdm
* \omega_1\, =\, -\eta_{2367}-\eta_{2345}-\eta_{1357}+\eta_{1346}-\eta_{1256}-\eta_{1247}
\, =\, \cyclic{i,j,k}\Phi^\H_i\wedge\eta_{jk},\quad
* \omega_2\, =\, \eta_{4567}\, =\, \frac{1}{6} \Psi^\H .
\edm
By Lemma \ref{lem.Phi-H}, $* \omega_1$ and $* \omega_2$  are closed, hence
$d * \omega =0$, which is the condition for being a cocalibrated $G_2$-structure. It is
proved in \cite{FrIv} that any
cocalibrated $G_2$-manifold admits a unique characteristic connection $\nabla$ with torsion
\bdm
T\ =\  -*d\omega+\frac{1}{6}\langle d\omega,*\omega\rangle\omega.
\edm
 Again, $d \omega_i$  ($i=1,2$) may be read off directly from Lemma \ref{lem.Phi-H},
from which we
conclude that $*d\omega_2=2\alpha\,\omega_1$ and $*d\omega_1= 2\delta\,\omega_1
+12\alpha\,\omega_2$.
One checks that $\langle d\omega,*\omega\rangle=24\,\alpha+12\,\delta$, and thus
\bdm
T = 2\alpha\omega_1 + 2(\delta-4\alpha)\omega_2 ,
\edm
in full agreement with Theorem \ref{thm.canonical-3AD-Sasaki}.
\end{proof}
This result allows us to use the full machinery of $G_2$-geometry for further investigations of
the connection $\nabla$. In particular, since $G_2$ is the stabilizer of a generic spinor inside
$\Spin(7)$, there exists a $\nabla$-parallel spinor field $\psi_0$. For a systematic investigation
of $G_2$-manifolds via spinors, we refer to \cite{Agricola&C&F&H14}.
\begin{df}[Canonical spinor field]
%----------------------------------
Let $\Sigma$ be the real spin bundle of $M^7$. The $G_2$-form $\omega$ acts via Clifford
multiplication on $\Sigma$ as a symmetric endomorphism field with eigenvalue $-7$
(multiplicity one) and  eigenvalue $+1$  (multiplicity seven). Consequently, it defines
(assuming $M^7$ simply connected) a unique \emph{canonical spinor field} $\psi_0$ such that
\bdm
\nabla \psi_0\, =\, 0,\quad \omega\cdot \psi_0 = -7\psi_0, \quad |\psi_0|=1.
\edm
Furthermore, a cocalibrated $G_2$-manifold with characteristic torsion $T$ satisfies
$T\cdot\psi_0= -\frac{1}{6}\langle d\omega,*\omega\rangle \, \psi_0$, which in our situation
means
\bdm
T\cdot\psi_0\ =\ -(4\alpha+2\delta)\, \psi_0.
\edm
\end{df}
\begin{df}
%----------
Recall that a spinor field $\psi$ is called a \emph{generalized Killing spinor}
if there exists a symmetric endomorphism field $S$ such that
$\nabla^g_X\psi=S(X)\cdot \psi$; being symmetric, we may assume that $S$ is given in diagonal
form. We call the eigenvalues of $S$ the \emph{generalized Killing numbers} of $\psi$.
If they coincide and are non-zero (i.\,e.~$S$ is a non-trivial multiple of the identity),
we have a (classical) Riemannian Killing spinor. The following results on
$7$-dimensional compact simply connected spin manifolds
are well-known (see \cite{FK90, BFGK,Ag,CS06, ABBK13}):
\begin{enumerate}[\normalfont 1)]
\item $M^7$ admits (at least) one generalized Killing spinor if it is a cocalibrated
$G_2$-manifold;
\item $M^7$ admits exactly one resp.~two resp.~three Killing spinor(s) if it is
a nearly parallel $G_2$-manifold resp.~Einstein-$\alpha$-Sasaki manifold resp.~$3$-$\alpha$-Sasaki manifold.
\end{enumerate}
\end{df}
Thus, our $7$-dimensional $3$-$(\alpha,\delta)$-Sasaki manifold should have at least one
generalized Killing spinor. In fact, we will show that it has four of them: the canonical spinor
$\psi_0$ and the three Clifford products $\xi_i\cdot\psi_0$, which are linearly independent
due to general properties of the spin representation.
\begin{theo}[Existence of generalized Killing spinors]
%-------------------------------------------------------
Let $(M^7,\varphi_i,\xi_i,\eta_i,g)$ be a simply connected
$7$-dimensional $3$-$(\alpha,\delta)$-Sasaki manifold, $\psi_0$ its canonical spinor.
\begin{enumerate}[\normalfont 1)]
\item
The canonical spinor field $\psi_0$
is a generalized Killing spinor:
\bdm
\nabla^g_X\psi_0 \ =\ -\frac{3\alpha}{2} X\cdot\psi_0 \ \text{ for }X\in \H,\quad
\nabla^g_Y\psi_0 \ =\ \frac{2\alpha - \delta}{2} Y\cdot\psi_0 \ \text{ for }Y\in \V.
\edm
The two generalized Killing numbers coincide if and only if $\delta=5\alpha$;
in this case, the $G_2$-form $\omega$ defines a nearly parallel $G_2$-structure
(Gray-Fernandez type $W_1$).
\item The Clifford products $\psi_i :=\xi_i\cdot \psi_0$, $i=1,2,3,$ are
generalized Killing spinors:
\bdm
\nabla^g_{\xi_i} \psi_i = \frac{2\alpha-\delta}{2}\,\xi_i\cdot \psi_i,\quad
\nabla^g_{\xi_j} \psi_i =  \frac{3\delta-2\alpha}{2}\,\xi_j\cdot \psi_i \ (i\neq j),\quad
\nabla^g_X \psi_i = \frac{\alpha}{2} \,X\cdot \psi_i\ \text{ for } X\in \H.
\edm
Any two of the generalized Killing numbers coincide if and only if $\alpha=\delta$, i.\,e.~if
$M^7$ is $3$-$\alpha$-Sasakian.
\end{enumerate}
\end{theo}
\begin{proof}
%----------------
By Theorem \ref{thm.canonical-3AD-Sasaki}, we know that $(M^7,\omega)$ is not only
cocalibrated, but that it also has  parallel torsion. We are thus in the context
of cocalibrated
$G_2$-manifolds with parallel torsion, studied by Friedrich in \cite{Friedrich07}. The shape
of the torsion identifies the right subcase to consider (Section 10,
$\mathfrak{hol}(\nabla)=\mathfrak{su}(2)\oplus\mathfrak{su}_c(2)$, the correct identification
of parameters is, up to an irrelevant change of numeration, $a=2\alpha$ and
$b=2(\delta-5\alpha)$).
Thus, the results of \cite[p.646]{Friedrich07} yield the stated equations for $\psi_0$.
Let us write this equation  in a more uniform way by introducing the two generalized Killing
numbers $\mu_\H,\mu_\V\in\R$
\bdm
\nabla^g_X\psi_0 = \mu_\H \, X\cdot\psi_0 \ \text{ for }X\in \H,\quad
\nabla^g_Y\psi_0 \ =\ \mu_\V\, Y\cdot\psi_0 \ \text{ for }Y\in \V.
\edm
To compute the Levi-Civita derivatives of $\psi_i=\xi_i\cdot\psi_0$, we begin with
\bdm
\nabla^g_X (\xi_i\cdot \psi_0) = \nabla^g_X (\xi_i)\cdot \psi_0 + \xi_i \cdot \nabla^g_X \psi_0
\quad \forall \ X\in \mathfrak{X}(M).
\edm
By Corollary \ref{corollary_nabla_xi}, we know that $\nabla^g_X (\xi_i)=-\alpha\varphi_i^\H (X)$
for $X\in\H$, $\nabla^g_{\xi_i} (\xi_i)=0$, and $\nabla^g_{\xi_j} (\xi_i)=-\delta\, \xi_k$
for even permutations. Hence, it makes sense to distinguish these three cases;
in particular, the first part of the claim follows immediatly.
For the other two cases, a computer-assisted computation in the spin representation is
needed: one checks that
\bdm
\varphi^\H_i(X)\cdot\psi_0\ = \  X\cdot\xi_i\cdot\psi_0 \ \text{ for }X\in \H,\quad
\xi_i\cdot\xi_j\cdot\psi_0\ =\ \xi_k\cdot\psi_0 \ \text{ for even permutations}.
\edm
The remaining claims now follow from a short calculation\footnote{Observe that
the theorem corrects a small computation error for the generalized Killing number in case
 $X\in\H$ from \cite[Cor.9]{Ag-F-S}; this was due to a wrong sign in the expression for
$\varphi^\H_i(X)\cdot\psi_0$.}.
\end{proof}
%
%-----------------------------------------------------------------------------------------
\section{Appendix: Examples on Lie groups}
%-----------------------------------------------------------------------------------------
%
In this section we denote by $G$ a $(4n+3)$-dimensional Lie group with Lie algebra $\frak g$ spanned by
vector fields $\xi_1,\xi_2,\xi_3, \tau_r,\tau_{n+r},\tau_{2n+r},\tau_{3n+r}$, $r=1,\ldots,n$.
We also consider the left invariant almost $3$-contact metric structure $(\varphi_i,\xi_i,\eta_i,g)$, where $g$ is the Riemannian metric
with respect to which the basis is orthonormal, $\eta_i$ is the dual $1$-form of
$\xi_i$, and $\varphi_i$ is given by
\begin{equation}\label{lie_group}
\varphi_i=\eta_j\otimes\xi_k-\eta_k\otimes\xi_j+\sum_{r=1}^n[\theta_r\otimes\tau_{in+r}
-\theta_{in+r}\otimes\tau_{r}
+\theta_{jn+r}\otimes\tau_{kn+r}-\theta_{kn+r}\otimes\tau_{jn+r}]
\end{equation}
where $\theta_l$, $l=1,\dots,4n$, is the dual $1$-form of $\tau_l$, and $(i,j,k)$ is an
even permutation of $(1,2,3)$. The fundamental $2$-forms of the structure are given by
\begin{equation}\label{Phi_H_examples}
\Phi_i={}-\eta_j\wedge\eta_k-\sum_{r=1}^n[\theta_r\wedge\theta_{in+r}
+\theta_{jn+r}\wedge\theta_{kn+r}].
\end{equation}
All our examples will be on nilpotent Lie groups, except the following one:
\begin{exs}\label{example_3delta}
%--------------------------------
Let $\frak g$ be the Lie algebra with nonvanishing commutators given by
\[[\xi_i,\xi_j]=2\delta\xi_k\]
where $\delta\in\R$, $\delta\ne 0$, and $(i,j,k)$ is any even permutation  of $(1,2,3)$.
It is isomorphic to $\so(3)\oplus \R^{4n} $.
The differential of each $1$-form $\eta_i$ is given by
\begin{equation*}\label{deta_3delta}
d\eta_i=-2\delta\eta_j\wedge\eta_k,
\end{equation*}
Since the $1$-forms $\theta_l$ are closed, from \eqref{Phi_H_examples} we have
\[d\Phi_i=-d\eta_j\wedge\eta_k+\eta_j\wedge d\eta_k=2\delta\,\eta_k\wedge\eta_i\wedge\eta_k-2\delta\,\eta_j\wedge\eta_i\wedge\eta_j=0.\]
Then the corresponding Lie group $(G,\varphi_i,\xi_i,\eta_i,g)$ is a $3$-$\delta$-cosymplectic manifold.
%Since the $1$-forms $\theta_l$ are closed, then
%%
%\begin{equation}\label{dPhi_H}
%d\Phi_i={}-d\eta_j\wedge\eta_k+\eta_j\wedge d\eta_k.
%\end{equation}
%%}
\end{exs}
\begin{exs}
%-------------
Here we will construct an example of a hypernormal canonical non-parallel 
almost $3$-contact metric manifold $(M,\varphi_i,\xi_i,\eta_i, g)$ that is not
$3$-$(\alpha,\delta)$-Sasaki (see Figure \ref{fig.classes}).
Let $\frak g$ be the Lie algebra with nonvanishing commutators given by
\begin{equation*}\label{differentials_example}
[\tau_r,\tau_{n+r}]=\xi_1,\qquad [\tau_r,\tau_{2n+r}]=\xi_2,\qquad [\tau_r,\tau_{3n+r}]=\xi_3.
\end{equation*}
Therefore,
\[d\eta_i=-\sum_{r=1}^n\theta_r\wedge\theta_{in+r},\qquad d\theta_l=0,\]
for every $i=1,2,3$ and $l=1,\ldots,4n$. First, let us check
that the left invariant almost $3$-contact
metric structure $(\varphi_i,\xi_i,\eta_i,g)$ defined on the Lie group $G$ is not
$3$-$(\alpha,\delta)$-Sasaki, nor in fact $3$-$\delta$-cosymplectic. Indeed, we have
\[
d\eta_i=\Phi_i+\eta_j\wedge\eta_k+\sum_{r=1}^n\theta_{jn+r}\wedge\theta_{kn+r}.
\]
The differential of the fundamental $2$-forms are given by
\[d\Phi_i=-d\eta_j\wedge\eta_k+\eta_j\wedge d\eta_k.\]
Therefore, for every $X,Y,Z\in\mathcal H$ we have $d\Phi_i(X,Y,Z)=0$ and
$N_{\varphi_i}(X,Y,Z)=0$. Since each $\xi_i$ is a Killing vector field, in order to prove
that the structure is canonical, we show that it admits a Reeb Killing function.
Notice that \eqref{lie_group} implies
\[\theta_r\circ\varphi_i=-\theta_{in+r},\qquad
\theta_{in+r}\circ\varphi_i=\theta_r,\qquad\theta_{in+r}\circ\varphi_j=\theta_{kn+r}
=-\theta_{jn+r}\circ\varphi_i,\]
for every even permutation $(i,j,k)$ of $(1,2,3)$.
Now, since ${\mathcal L}_{\xi_j}\varphi_i=0$, for every $X,Y\in\mathcal H$, we have
\begin{align*}
&A_{ij}(X,Y)=d\eta_j(X,\varphi_iY)+d\eta_j(\varphi_iX,Y)\\
&=-\sum_{r=1}^n(\theta_r\wedge\theta_{jn+r})(X,\varphi_iY)-\sum_{r=1}^n(\theta_r\wedge\theta_{jn+r})(\varphi_iX,Y)\\
&=\sum_{r=1}^n[\theta_r(X)\theta_{kn+r}(Y)-\theta_{jn+r}(X)\theta_{in+r}(Y)+\theta_{in+r}(X)\theta_{jn+r}(Y)-\theta_{kn+r}(X)\theta_r(Y)]\\
&=\sum_{r=1}^n (\theta_r\wedge\theta_{kn+r})(X,Y)+\sum_{r=1}^n(\theta_{in+r}\wedge\theta_{jn+r})(X,Y)\\
&=-\Phi_k(X,Y).
\end{align*}
Analogously, one shows that $A_{ji}(X,Y)=\Phi_k(X,Y)$. Furthermore,
\begin{align*}
&A_i(X,Y)=-\sum_{r=1}^n(\theta_r\wedge\theta_{in+r})(X,\varphi_iY)-\sum_{r=1}^n(\theta_r\wedge\theta_{in+r})(\varphi_iX,Y)\\
&=-\sum_{r=1}^n[\theta_r(X)\theta_r(Y)+\theta_{in+r}(X)\theta_{in+r}(Y)-\theta_{in+r}(X)\theta_{in+r}(Y)-\theta_r(X)\theta_r(Y)] =0.
\end{align*}
Therefore, the structure admits constant Reeb Killing function $\beta=-1$, and it is
canonical. Observe that this allows us to conclude that the structure is hypernormal:
Indeed, we already showed that $N_{\varphi_i}(X,Y,Z)=0$ for every $X,Y,Z\in\Gamma({\mathcal H})$. 
Since the structure is canonical, by Theorem \ref{theo_canonical-implies-char}, each tensor 
field $N_{\varphi_i}$ is skew-symmetric on $TM$. Furthermore, each $\xi_i$ lies in the center 
of the Lie algebra. Then, taking into account equations \eqref{tableNij} one easily 
verifies that $N_{\varphi_i}=0$.

Using \eqref{TNS}-\eqref{T2} and \eqref{T_canonical}, one can check that the
torsion $T$ of the canonical connection is given by
%satisfies
%\[T(X,Y,Z)=0,\quad T(X,Y,\xi_i)=d\eta_i(X,Y),\quad T(X,\xi_i,\xi_j)=0,\quad T(\xi_1,\xi_2,\xi_3)=-2\]
%for every $X,Y,Z\in\mathca H$ and $i,j=1,2,3$. Equivalently,
\[T=\sum_{i=1}^3\eta_i\wedge d\eta_i-2\,\eta_1\wedge\eta_2\wedge\eta_3.\]
By Theorem \ref{theo_canonical--char-torsion}, the characteristic connection of the
structure $(\varphi_i,\xi_i,\eta_i,g)$ has torsion
\begin{align*}
&T_i=T-\eta_j\wedge\Phi_j-\eta_k\wedge\Phi_k\\
&=T-\eta_j\wedge (d\eta_j-\eta_k\wedge\eta_i-\sum_{r=1}^n\theta_{kn+r}\wedge\theta_{in+r}) -\eta_k\wedge (d\eta_k-\eta_i\wedge\eta_j-\sum_{r=1}^n\theta_{in+r}\wedge\theta_{jn+r})\\
&=\eta_i\wedge d\eta_i+\sum_{r=1}^n(\eta_j\wedge\theta_{kn+r}-\eta_k\wedge\theta_{jn+r})\wedge\theta_{in+r}.
\end{align*}
\end{exs}
We end with two examples of almost $3$-contact metric manifolds $(M,\varphi_i,\xi_i,\eta_i, g)$ 
that admit $\varphi_i$-compatible connections despite not being canonical (and, furthermore,
not hypernormal).
\begin{exs}
%-----------
Let $\frak g$ be the Lie algebra with non-vanishing commutators
\[[\tau_r,\tau_{n+r}]=[\tau_{2n+r},\tau_{3n+r}]
=\xi_1. \]
Assuming the corresponding Lie group $G$ to be connected and simply connected, it is the product $H^{2n}_{\R}\times \R^2$, where $H^{2n}_{\R}$ is
the real Heisenberg group of dimension $4n+1$. The left invariant structure
$(\varphi_i,\xi_i,\eta_i,g)$ satisfies
\begin{equation*}\label{deta_H2n}
d\eta_1=-\sum_{r=1}^n[\theta_r\wedge\theta_{n+r}+\theta_{2n+r}\wedge\theta_{3n+r}]=\Phi_1+\eta_2\wedge\eta_3,\qquad d\eta_2=0,\qquad d\eta_3=0.
\end{equation*}
Being also $d\theta_l=0$, we have
\[d\Phi_1=0,\qquad d\Phi_2=\eta_3\wedge d\eta_1=\eta_3\wedge \Phi_1,\qquad d\Phi_3
=-d\eta_1\wedge\eta_2=-\Phi_1\wedge\eta_2.\]
Then, for every $i=1,2,3$, and for every $X,Y,Z\in\mathcal{H}$, we have $d\Phi_i(X,Y,Z)=0$
and  $N_{\varphi_i}(X,Y,Z)=0$. Each $\xi_i$ is a Killing vector field. Proposition
\ref{prop_Killing} implies that the manifold admits $\varphi_i$-compatible connections
for every $i=1,2,3$. Nevertheless, the structure is not canonical. Indeed, one can
easily verify that, for every $X,Y\in\mathcal H$,
\[A_{i2}(X,Y)=A_{i3}(X,Y)=0,\quad i=1,2,3,\]
\[A_1(X,Y)=0,\quad A_{21}(X,Y)=2\Phi_3(X,Y),\quad A_{31}(X,Y)=-2\Phi_2(X,Y).\]
One can also notice that this structure is not hypernormal. Indeed, using \eqref{tableN},
we see that $N_{\varphi_1}=0$, but
\[N_{\varphi_2}(X,Y,\xi_1)=N_{\varphi_3}(X,Y,\xi_1)=2\, \Phi_1(X,Y)\quad \forall  X,Y \in \H.
\]
\end{exs}
\begin{exs}
%---------------
Let $\frak g$ be the Lie algebra with non-vanishing commutators
\[[\tau_r,\tau_{n+r}]=[\tau_{2n+r},\tau_{3n+r}]=\xi_1, \qquad[\tau_r,\tau_{2n+r}]
=[\tau_{3n+r},\tau_{n+r}]=\xi_2.\]
The corresponding connected simply connected Lie group $G$ is the product $H_\mathbb{C}^{2n}\times \R$, where
$H_\mathbb{C}^{2n}$ is the complex Heisenberg group of real dimension $4n+2$. The left
invariant  structure $(\varphi_i,\xi_i,\eta_i,g)$ satisfies
\begin{align*}\label{deta_H2nC}
d\eta_1&=-\sum_{r=1}^p[\theta_r\wedge\theta_{n+r}+\theta_{2n+r}\wedge\theta_{3n+r}]=\Phi_1+\eta_2\wedge\eta_3,\\
d\eta_2&=-\sum_{r=1}^p[\theta_r\wedge\theta_{2n+r}+\theta_{3n+r}\wedge\theta_{n+r}]=\Phi_2+\eta_3\wedge\eta_1,
\end{align*}
and $d\eta_3=0$.
Since $d\theta_l=0$, we have
\[d\Phi_1=-d\eta_2\wedge\eta_3=-\Phi_2\wedge\eta_3,\qquad d\Phi_2=\eta_3\wedge d\eta_1=\eta_3\wedge \Phi_1,\]
\[d\Phi_3=-d\eta_1\wedge\eta_2+\eta_1\wedge d\eta_2=-\Phi_1\wedge\eta_2+\eta_1\wedge\Phi_2.\]
Again, for every $i=1,2,3$, and for every $X,Y,Z\in\mathcal H$, we have $d\Phi_i(X,Y,Z)=0$
and  $N_{\varphi_i}(X,Y,Z)=0$. Each $\xi_i$ is a Killing vector field. From Proposition
\ref{prop_Killing} the manifold admits $\varphi_i$-compatible connections for every
$i=1,2,3$. Nevertheless, the structure is not canonical. Indeed, for every
$X,Y\in\mathcal H$,
\[A_{13}(X,Y)=A_{23}(X,Y)=0,\quad  A_{31}(X,Y)=-2\Phi_2(X,Y),\quad A_{32}(X,Y)=2\Phi_1(X,Y).\]
The structure is not hypernormal. Indeed, using \eqref{tableN}, we see that
$N_{\varphi_3}=0$, but
\[N_{\varphi_1}(X,Y,\xi_3)=N_{\varphi_2}(X,Y,\xi_3)=-2\, \Phi_3(X,Y)\quad \forall X,Y\in\H.
\]
\end{exs}
%
%-----------------------------------------------------------------------------------------
%-----------------------------------------------------------------------------------------
%

\noindent
Ilka Agricola, Fachbereich Mathematik und Informatik, Philipps-Universit\"at Marburg,
Campus Lahnberge, 35032 Marburg, Germany. \texttt{agricola@mathematik.uni-marburg.de}

\medskip\noindent
Giulia Dileo, Dipartimento di Matematica, Universit\`a degli Studi di Bari Aldo Moro,
Via E. Orabona 4, 70125 Bari, Italy.
\texttt{giulia.dileo@uniba.it}

\end{document}